\documentclass[a4paper,10pt,twoside]{article}\usepackage{styleforarxiv}\pdfoutput=1
\AUTHORS{Aim\'e LACHAL\footnote{Affiliation: Universit\'e de Lyon/Institut Camille Jordan CNRS UMR5208 \hspace{.3\linewidth}\mbox{} Postal address: \textsl{Institut National des Sciences Appliqu\'ees de Lyon} \hspace{.4\linewidth}\mbox{} P\^ole de Math\'ematiques, B\^atiment L\'eonard de Vinci \hspace{.5\linewidth}\mbox{} 20 avenue Albert Einstein, 69621 Villeurbanne Cedex, France \hspace{.4\linewidth}\mbox{} \mbox{E-mail:~\texttt{aime.lachal@insa-lyon.fr}} \hspace{.6\linewidth}\mbox{} URL:~\url{http://maths.insa-lyon.fr/\~lachal}}}
\TITLE{From pseudo-random walk\\[-.5ex] to pseudo-Brownian motion:\\[-.5ex] first exit time from a one-sided\\[.5ex] or a two-sided interval}
\SHORTTITLE{From pseudo-random walk to pseudo-Brownian motion}
\ABSTRACT{Let $N$ be a positive integer, $c$ be a positive constant and $(U_n)_{n\ge 1}$ be a sequence of independent identically distributed pseudo-random variables. We assume that the $U_n$'s take their values in the discrete set $\{-N,-N+1,\dots,N-1,N\}$ and that their common pseudo-distribution is characterized by the \textit{(positive or negative) real} numbers \[\mathbb{P}\{U_n=k\}=\delta_{k0}+(-1)^{k-1} c\binom{2N}{k+N}\] for any $k\in\{-N,-N+1,\dots,N-1,N\}$. Let us finally introduce $(S_n)_{n\ge 0}$ the associated pseudo-random walk defined on $\mathbb{Z}$ by $S_0=0$ and $S_n=\sum_{j=1}^n U_j$ for $n\ge 1$.\\ In this paper, we exhibit some properties of $(S_n)_{n\ge 0}$. In particular, we explicitly determine the pseudo-distribution of the first overshooting time of a given threshold for $(S_n)_{n\ge 0}$ as well as that of the first exit time from a bounded interval.\\ Next, with an appropriate normalization, we pass from the pseudo-random walk to the pseudo-Brownian motion driven by the high-order heat-type equation $\partial/\partial t=(-1)^{N-1} c\;\partial^{2N}\!/\partial x^{2N}$. We retrieve the corresponding pseudo-distribution of the first overshooting time of a threshold for the pseudo-Brownian motion (Lachal, A.: First hitting time and place, monopoles and multipoles for pseudo-processes driven by the equation $\frac{\partial}{\partial t}=\pm \frac{\partial^N}{\partial x^N}$. Electron. J. Probab. 12 (2007), 300--353 [MR2299920]). In the same way, we get the pseudo-distribution of the first exit time from a bounded interval for the pseudo-Brownian motion which is a new result for this pseudo-process.}
\AMSSUBJ{60G15 ; 60G25}\AMSSUBJSECONDARY{60J65}

\usepackage{arydshln}
\newcommand{\Xa}{X\raisebox{-0.5ex}{$\scriptstyle\tau_a^-$}}

\newcommand{\Xb}{X\raisebox{-0.5ex}{$\scriptstyle\tau_b^+$}}

\newcommand{\ind}{1\hspace{-.27em}\mbox{\rm l}}
\newcommand{\lqn}[1]{\noalign{\noindent $\displaystyle{#1}$}}

\begin{document}

\section{Introduction}

Throughout the paper, we denote by $\mathbb{Z}$ the set of integers, by $\mathbb{N}$ that of
non-negative integers and by $\mathbb{N}^*$ that of positive integers:
$\mathbb{Z}=\{\dots,-1,0,1,\dots\}$, $\mathbb{N}=\{0,1,2,\dots\}$, $\mathbb{N}^*=\{1,2,\dots\}$.
More generally, for any set of numbers $E$, we set $E^*=E\backslash\{0\}$.

Let $N$ be a positive integer, $c$ be a positive constant and set $\kappa_{_{\!N}}=(-1)^{N-1}$.
Let $(U_n)_{n\in\mathbb{N}^*}$ be a sequence of independent identically distributed
pseudo-random variables taking their values in the set of integers
$\{-N,-N+1,\dots,-1,0,1,\dots,N-1,N\}$. By \textit{pseudo-random variable},
we mean a measurable function defined on a space endowed with a signed measure
with a total mass equaling the unity.
We assume that the common pseudo-distribution of the $U_n$'s is characterized by the
\textit{(positive or negative) real} pseudo-probabilities $p_k=\mbox{$\mathbb{P}\{U_n=k\}$}$ for
any $k\in\{-N,-N+1,\dots,N-1,N\}$. The parameters $p_k$ sum to the unity: $\sum_{k=-N}^N p_k=1$.

Now, let us introduce $(S_n)_{n\in\mathbb{N}}$ the associated pseudo-random walk defined
on $\mathbb{Z}$ by $S_0=0$ and $S_n=\sum_{j=1}^n U_j$ for $n\in\mathbb{N}^*$. 
The infinitesimal generator associated with $(S_n)_{n\in\mathbb{N}}$ is defined,
for any function $f$ defined on $\mathbb{Z}$, as
\[
\mathcal{G}_{_S} f(j)=\mathbb{E}[f(U_1+j)]-f(j)=\sum_{k=-N}^N p_k f(j+k)-f(j),\quad j\in\mathbb{Z}.
\]
Here we consider the pseudo-random walk which admits the discrete $N$-iterated
Laplacian as a generator infinitesimal. More precisely, by introducing the
so-called discrete Laplacian $\mathbf{\Delta}$ defined, for any function $f$ defined
on $\mathbb{ Z}$, by
\[
\mathbf{\Delta} f(j)=f(j+1)-2f(j)+f(j-1),\quad j\in\mathbb{Z},
\]
the discrete $N$-iterated Laplacian is the operator
$\mathbf{\Delta}^{\!N}=\underbrace{\mathbf{\Delta}\circ\dots\circ\mathbf{\Delta}}_{N\text{ times}}$
given by
\[
\mathbf{\Delta}^{\!N}\!f(j)=\sum_{k=-N}^{N}(-1)^{k+N} \binom{2N}{k+N} f(j+k),\quad j\in\mathbb{Z}.
\]
We then choose the $p_k$'s such that $\mathcal{G}_{_S}=\kappa_{_{\!N}} c\,\mathbf{\Delta}^{\!N}$
which yields, by identification, for any $k\in\{-N,-N+1,\dots,-1,1,\dots,N-1,N\}$,
\begin{equation}\label{proba-pk}
p_k=p_{-k}=(-1)^{k-1} c\binom{2N}{k+N}\quad\text{ and }\quad p_0=1-c\binom{2N}{N}.
\end{equation}

When $N=1$, $(S_n)_{n\in\mathbb{N}}$ is the pseudo-random walk to the closest neighbors
with a possible stay at its current location; it is characterized by the numbers
$p_1=p_{-1}=c$ and $p_0=1-2c$. Moreover, if $0<c<1/2$, then $p_0>0$; in this
case, we are dealing with an ordinary random walk (with positive probabilities).

Actually, with the additional assumption that $p_{-k}=p_k$ for any $k\in\{-N,-N+1,\dots,$
$N-1,N\}$ (i.e., the $U_n$'s are symmetric, or the pseudo-random walk has no drift),
the $p_k$'s are the unique numbers such that
\begin{equation}\label{infinitesimal-generator1}
\mathcal{G}_{_S} f(j)=\kappa_{_{\!N}} c\,\tilde{f}^{(2N)}(j)+\text{\it terms with higher
order derivatives}
\end{equation}
where $\tilde{f}$ is an analytical extension of $f$ and $\tilde{f}^{(2N)}$
stands for the $(2N)^{\mathrm{th}}$ derivative of $\tilde{f}$:
$\tilde{f}^{(2N)}(x)=\frac{\mathrm{d}^{2N}\tilde{f}}{\mathrm{d} x^{2N}}(x)$.

Our motivation for studying the pseudo-random walk associated with
the parameters defined by~(\ref{proba-pk})
is that it is the discrete counterpart of the pseudo-Brownian
motion as the classical random walk is for Brownian motion.
Let us recall that pseudo-Brownian motion is the pseudo-Markov process $(X_t)_{t\ge 0}$
with independent and stationary increments, associated with the \textit{signed}
heat-type kernel $p(t;x)$ which is the elementary solution of the high-order
heat-type equation $\partial/\partial t=\kappa_{_{\!N}} c\,\partial^{2N}\!/\partial x^{2N}\!.$
The kernel $p(t;x)$ is characterized by its Fourier transform:
\[
\mathbb{E}\!\left(\mathrm{e}^{\mathrm{i} uX_t}\right)=\int_{-\infty}^{+\infty}
\mathrm{e}^{\mathrm{i} ux}\,p(t;x)\,\mathrm{d} x=\mathrm{e}^{-ctu^{2N}}\!.
\]
The corresponding infinitesimal generator is given, for any $C^{2N}$-function
$f$, by
\begin{equation}\label{infinitesimal-generator2}
\mathcal{G}_{_X} f(x)=\lim_{h\to 0^+} \frac{1}{h}\left[\mathbb{E}[f(X_h+x)]-f(x)\right]
=\kappa_{_{\!N}} c\,f^{(2N)}(x).
\end{equation}
The reader can find an extensive literature on pseudo-Brownian motion.
For instance, let us quote the works of Beghin, Cammarota, Hochberg, Krylov,
Lachal, Nakajima, Nikitin, Nishioka, Orsingher, Ragozina, Sato (\cite{bho} to \cite{ors})
and the references therein.

We observe that (\ref{infinitesimal-generator1}) and~(\ref{infinitesimal-generator2})
are closely related to the continuous $N$-iterated Laplacian
$\mathrm{d}^{2N}\!/\mathrm{d} x^{2N}$.
For $N=2$, the operator $\mathbf{\Delta}^{\!2}$ is the two-Laplacian related to the
famous biharmonic functions: in the discrete case,
\[
\mathbf{\Delta}^{\!2} f(j)=f(j+2)-4f(j+1)+6f(j)-4f(j-1)+f(j-2),\quad j\in\mathbb{Z},
\]
and in the continuous case,
\[
\mathbf{\Delta}^{\!2} f(x)=\frac{\mathrm{d}^4f}{\mathrm{d} x^4}(x),\quad x\in\mathbb{R}.
\]
In the discrete case, it has been considered by Sato~\cite{sato} and Vanderbei~\cite{vanderbei}.

The link between the pseudo-random walk and pseudo-Brownian motion
is the following one: when normalizing the pseudo-random walk $(S_n)_{n\in\mathbb{N}}$
on a grid with small spatial step $\varepsilon^{2N}$ and temporal step $\varepsilon$ (i.e., we construct
the pseudo-process $\big(\varepsilon S_{\lfloor t/\varepsilon^{2N}\rfloor}\big)_{t\ge 0}$ where
$\lfloor\,.\,\rfloor$ denotes the usual floor function),
the limiting pseudo-process as $\varepsilon\to 0^+$ is exactly the pseudo-Brownian motion.

Now, we consider the first overshooting time of a fixed single threshold $a<0$ or $b>0$
($a,b$ being integers) for $(S_n)_{n\in\mathbb{N}}$:
\[
\sigma_a^-=\min\{n\in\mathbb{N}^*: S_n\le a\},\quad \sigma_b^+=\min\{n\in\mathbb{N}^*: S_n\ge b\}
\]
as well as the first exit time from a bounded interval $(a,b)$:
\[
\sigma_{ab}=\min\{n\in \mathbb{N}^*: S_n\le a \text{ or } S_n\ge b\}
=\min\{n\in \mathbb{N}^*: S_n\notin(a,b)\}
\]
with the usual convention that $\min\emptyset=+\infty$.
Since, when $\sigma_b^+<+\infty$, $S_{\sigma_b^+-1}\le b-1$ and $S_{\sigma_b^+}\ge b$,
the overshoot at time
$\sigma_b^+$ which is $S_{\sigma_b^+}-b$ can take the values
$0,1,2,\dots,N-1$, that is $S_{\sigma_b^+}\in\{b,b+1,b+2,\dots,b+N-1\}$.
Similarly, when $\sigma_a^-<+\infty$, $S_{\sigma_a^-}\in\{a-N+1,a-N+2,\dots,a\}$, and when
$\sigma_{ab}<+\infty$, $S_{\sigma_{ab}}\in\{a,a-1,\dots,a-N+1\}\cup\{b,b+1,\dots,b+N-1\}$.
We put $S_b^+=S_{\sigma_b^+}$, $S_a^-=S_{\sigma_a^-}$ and $S_{ab}=S_{\sigma_{ab}}$.

In the same way, we introduce the first overshooting times of the
thresholds $a<0$ and $b>0$ ($a,b$ being now real numbers) for $(X_t)_{t\ge 0}$:
\[
\tau_a^-=\inf\{t\ge 0: X_t\le a\},\quad \tau_b^+=\inf\{t\ge 0: X_t\ge b\}
\]
as well as the first exit time from a bounded interval $(a,b)$:
\[
\tau_{ab}=\inf\{t\ge 0: X_t \le a \text{ or } X_t\ge b\}=\inf\{t\ge 0: X_t\notin(a,b)\}
\]
with the similar convention that $\inf\emptyset=+\infty$, and we set, when
the corresponding time is finite,
\[
X_a^-=\Xa,\quad X_b^+=\Xb,\quad X_{ab}=X_{\tau_{ab}}.
\]

In this paper we provide a representation for the generating function of the joint
distributions of the couples $(\sigma_a^-,S_a^-)$, $(\sigma_b^+,S_b^+)$ and $(\sigma_{ab},S_{ab})$.
In particular, we derive simple expressions for the marginal distributions of
$S_a^-$, $S_b^+$ and $S_{ab}$. We also obtain explicit expressions for the
famous ``ruin pseudo-probabilities'' $\mathbb{P}\{\sigma_a^-<\sigma_b^+\}$ and
$\mathbb{P}\{\sigma_b^+<\sigma_a^-\}$.
The main tool employed in this paper is the use of generating functions.

Taking the limit as $\varepsilon$ goes to zero, we retrieve the joint distributions of
the couples $(\tau_a^-,X_a^-)$ and  $(\tau_b^+,X_b^+)$ obtained in~\cite{la2,la3}.
Therein, we used Spitzer's identity for deriving these distributions.
Moreover, we obtain the joint distribution of the couple $(\tau_{ab},X_{ab})$
which is a new and an important result for the study of pseudo-Brownian motion.
In particular, we deduce the ``ruin pseudo-probabilities'' $\mathbb{P}\{\tau_a^-<\tau_b^+\}$
and $\mathbb{P}\{\tau_b^+<\tau_a^-\}$; the results have been announced without
any proof in a survey on pseudo-Brownian motion, \cite{la5}, after a conference
held in Madrid (IWAP 2010).

In~\cite{la2}, \cite{nish1} and~\cite{nish2}, the authors observed a curious fact concerning
the pseudo-distributions of $X_a^-$ and $X_b^+$: they are linear combinations
of the Dirac distribution and its successive derivatives (in the sense of
Schwarz distributions). For instance,
\begin{equation}\label{dirac}
\mathbb{P}\{X_b^+\in \mathrm{d} z\}/\mathrm{d} z=\sum_{j=0}^{N-1}
\frac{b^{\,j}}{j!}\,\delta_b^{(j)}(z).
\end{equation}
The quantity $\delta_b^{(j)}$ is to be understood as the functional acting
on test functions $\phi$ according as $\langle \delta_b^{(j)},\phi \rangle
=(-1)^j\phi^{(j)}(b)$.
The appearance of the $\delta_b^{(j)}$'s in~(\ref{dirac}), which is quite
surprising for probabilists, can be better understood thanks
to the discrete approach. Indeed, the $\delta_b^{(j)}$'s
come from the location at the overshooting time of $b$ for the normalized
pseudo-random walk: the location takes place in the ``cluster'' of points
$b,b+\varepsilon,b+2\varepsilon,\dots,b+(N-1)\varepsilon$.

In order to facilitate the reading of the paper, we have divided it
into three parts:
\begin{itemize}
\item[]
Part I --- Some properties of the pseudo-random walk
\item[]
Part II --- First overshooting time of a single threshold
\item[]
Part III --- First exit time from a bounded interval
\end{itemize}

\newpage
\begin{center}
\textbf{\Large Part I --- Some properties of the pseudo-random walk}
\end{center}
\vspace{\baselineskip}


\section{Pseudo-distribution of $U_1$ and $S_n$}

We consider the pseudo-random walk $(S_n)_{n\in\mathbb{N}}$ related to a family of real
parameters $\{p_k,\,k\in\{-N,\dots,N\}\}$ satisfying $p_k=p_{-k}$ for any
$k\in\{1,\dots,N\}$ and $\sum_{k=-N}^N p_k=1$.
Let us recall that the infinitesimal generator associated with
$(S_n)_{k\ge 0}$ is defined by
\[
\mathcal{G}_{_S} f(j)=\sum_{k=-N}^N p_k f(j+k)-f(j)=(p_0-1) f(j)+\sum_{k=1}^N p_k [f(j+k)+f(j-k)].
\]
In this section, we look for the values of $p_k,\,k\in\{-N,\dots,N\}$, for which
the infinitesimal generator  $\mathcal{G}$ is of the form~(\ref{infinitesimal-generator1}).
Next, we provide several properties for the corresponding pseudo-random walk.

Suppose that $f$ can be extended into an analytical function $\tilde{f}$.
In this case, we can expand
\[
f(j+k)+f(j-k)=2\sum_{\ell=0}^{\infty} \frac{k^{2\ell}}{(2\ell)!}\, \tilde{f}^{(2\ell)}(j).
\]
Therefore,
\begin{align*}
\mathcal{G} f(j)
&=(p_0-1) f(j)+2\sum_{k=1}^N p_k\sum_{\ell=0}^{\infty} \frac{k^{2\ell}}{(2\ell)!}\,
\tilde{f}^{(2\ell)}(j)
\\
&=\!\left(p_0+2\sum_{k=1}^N p_k-1\right) \!f(j)+
2\sum_{\ell=1}^{\infty} \left(\sum_{k=1}^N k^{2\ell}p_k \right)
\frac{\tilde{f}^{(2\ell)}(j)}{(2\ell)!}.
\end{align*}
Since $p_0+2\sum_{k=1}^N p_k=1$, we see that the
expression~(\ref{infinitesimal-generator1}) of $\mathcal{G}$ holds if and only
if the $p_k$'s satisfy the equations
\begin{equation}\label{equation-proba}
\sum_{k=1}^N k^{2\ell}p_k=0 \text{ for $1\le \ell\le N-1$}\quad\text{and}
\quad \sum_{k=1}^N k^{2N}p_k =\frac12 \kappa_{_{\!N}} c(2N)!.
\end{equation}
%
\begin{proposition}\label{proposition-pk}
The numbers $p_k,\,k\in\{1,\dots,N\}$, satisfying~(\ref{equation-proba}) are given by
\[
p_k=(-1)^{k-1}c\binom{2N}{k+N}.
\]
In particular, $p_{_N}=\kappa_{_{\!N}} c$.
\end{proposition}
%
\begin{proof}
First, we recall that the solution of a Vandermonde system of the form
$\sum_{k=1}^N a_k^{\ell} x_k=\alpha_{\ell}$, $1\le \ell \le N$,
is given by
\[
x_k=\frac{\Delta_k\!\left(\!\!\!\begin{array}{c} a_1,\dots,a_N\\[-.2ex]
\displaystyle\alpha_1,\dots,\alpha_N\end{array}\!\!\!\right)\!}{\Delta(a_1,\dots,a_N)},
\quad 1\le k\le N
\]
with
\[
\Delta(a_1,\dots,a_N)=\begin{vmatrix}
a_1    & \dots & a_N    \\
a_1^2  & \dots & a_N^2  \\
\vdots &       & \vdots \\
a_1^N  & \dots & a_N^N
\end{vmatrix}
\]
and, for any $k\in\{1,\dots,N\}$,
\[
\Delta_k\!\left(\!\!\!\begin{array}{c} a_1,\dots,a_N\\[-.2ex]
\displaystyle\alpha_1,\dots,\alpha_N\end{array}\!\!\!\right)\!
=\begin{vmatrix}
a_1    & \dots & a_{k-1}   & \alpha_1   & a_{k+1}   & \dots & a_N    \\
a_1^2  & \dots & a_{k-1}^2 & \alpha_2   & a_{k+1}^2 & \dots & a_N^2  \\
\vdots &       & \vdots    & \vdots & \vdots    &       & \vdots \\
a_1^N  & \dots & a_{k-1}^N & \alpha_N   & a_{k+1}^N & \dots & a_N^N
\end{vmatrix}\!.
\]
In the notation of $\Delta_k$ and that of forthcoming determinants, we adopt
the convention that when the index of certain entries in the determinant
lies out of the range of $k$, the corresponding column is discarded.
That is, for $k=1$ and $k=N$, the respective determinants write
\[
\Delta_1\!\left(\!\!\!\begin{array}{c} a_1,\dots,a_N\\[-.2ex]
\displaystyle\alpha_1,\dots,\alpha_N\end{array}\!\!\!\right)\!
=\begin{vmatrix}
\alpha_1   & a_2    & \dots  & a_N    \\
\vdots & \vdots &        & \vdots \\
\alpha_N   & a_2^N  & \dots  & a_N^N
\end{vmatrix},\quad
\Delta_N\!\left(\!\!\!\begin{array}{c} a_1,\dots,a_N\\[-.2ex]
\displaystyle\alpha_1,\dots,\alpha_N\end{array}\!\!\!\right)\!
=\begin{vmatrix}
a_1    & \dots & a_{N-1}   & \alpha_1   \\
\vdots &       & \vdots    & \vdots \\
a_1^N  & \dots & a_{N-1}^N & \alpha_N
\end{vmatrix}\!.
\]
It is well-known that, for any $k\in\{1,\dots,N\}$,
\begin{align*}
\Delta(a_1,\dots,a_N)
&
=\prod_{1\le j\le N} a_j \prod_{1\le \ell< m\le N} (a_m-a_{\ell})
\\
&
=(-1)^{k+N} \prod_{1\le j\le N} a_j \prod_{1\le j\le N\atop j\neq k} (a_k-a_j)
\prod_{1\le \ell< m\le N\atop \ell,m\neq k} (a_m-a_{\ell}).
\end{align*}
In the particular case where $\alpha_{\ell}=0$ for $1\le \ell\le N-1$,
we have, for any $k\in\{1,\dots,N\}$, that
\begin{align*}
\Delta_k\!\left(\!\!\!\begin{array}{c} a_1,\dots,a_N\\[-.2ex]
\displaystyle\alpha_1,\dots,\alpha_N\end{array}\!\!\!\right)\!
&=(-1)^{k+N}\alpha_N\begin{vmatrix}
a_1        & \dots & a_{k-1}       & a_{k+1}       & \dots  & a_N    \\
a_1^2      & \dots & a_{k-1}^2     & a_{k+1}^2     & \dots  & a_N^2  \\
\vdots     &       & \vdots        & \vdots        &        & \vdots \\
a_1^{N-1}  & \dots & a_{k-1}^{N-1} & a_{k+1}^{N-1} & \dots  & a_N^{N-1}
\end{vmatrix}
\\
&=(-1)^{k+N}\alpha_N \prod_{1\le j\le N\atop j\neq k} a_j
\prod_{1\le \ell< m\le N\atop \ell,m\neq k} (a_m-a_{\ell}).
\end{align*}
Therefore, the solution simply writes
\[
x_k=\frac{\alpha_N}{a_k \prod_{1\le j\le N\atop j\neq k} (a_k-a_j)},
\quad 1\le k\le N.
\]

Now, we see that system~(\ref{equation-proba}) is a Vandermonde system
with the choices $a_k=k^2$, $x_k=p_k$ and $\alpha_{\ell}=0$ for $1\le \ell\le N-1$,
$\alpha_N=\kappa_{_{\!N}} c(2N)!/2$.
With these settings at hands, we explicitly have
\begin{align*}
a_k \prod_{1\le j\le N\atop j\neq k} (a_k-a_j)
&=k^2 \prod_{1\le j\le N\atop j\neq k} (k^2-j^2)
=k^2 \prod_{1\le j\le N\atop j\neq k} (k-j)
\prod_{1\le j\le N\atop j\neq k}(k+j)
\\
&=\frac12\,(-1)^{N-k} (N+k)!(N-k)!
\end{align*}
and the result of Proposition~\ref{proposition-pk} ensues.
\end{proof}

Finally, the value of $p_0$ is obtained as follows: by using the fact that
$\sum_{k=0}^{2N} (-1)^k\binom{2N}{k}=0$,
\begin{align*}
p_0
&
=1-\sum_{-N\le k\le N\atop k\neq 0} p_k
=1+c\sum_{-N\le k\le N\atop k\neq 0} (-1)^k\binom{2N}{k+N}
\\
&
=1-c\binom{2N}{N}+(-1)^N c\sum_{k=0}^{2N} (-1)^k\binom{2N}{k}
=1-c\binom{2N}{N}.
\end{align*}
We find it interesting to compute the cumulative sums of the $p_j$'s:
for $k\in\{-N,\dots,N\}$,
\[
\sum_{j=-N}^k p_j=\sum_{j=-N}^k \left[\delta_{j0}+(-1)^{j-1}c\binom{2N}{j+N}\right]
=\ind_{\{k\ge 0\}} +(-1)^{N-1} c\sum_{j=0}^{k+N} (-1)^j\binom{2N}{j}.
\]
The last displayed sum is classical and easy to compute by appealing to Pascal's
formula which leads to a telescopic sum:
\[
\sum_{j=0}^{k+N} (-1)^j\binom{2N}{j}=\sum_{j=0}^{k+N} \left[(-1)^j\binom{2N-1}{j-1}
-(-1)^{j+1}\binom{2N-1}{j}\right]=(-1)^{k+N}\binom{2N-1}{k+N}.
\]
Thus, for $k\in\{-N,\dots,N\}$,
\[
\sum_{j=-N}^k p_j=\ind_{\{k\ge 0\}}+(-1)^{k-1}c\binom{2N-1}{k+N}.
\]
Observe that this sum is nothing but $\mathbb{P}\{U_1\le k\}$.
Next, we compute the total sum of the $|p_j|$'s: by using the fact that
$\sum_{k=0}^{2N} \binom{2N}{k}=4^N$,
\begin{align*}
\sum_{k=-N}^N |p_k|
&=\left|1-c\binom{2N}{N}\right|+\sum_{-N\le k\le N\atop k\neq 0} c\binom{2N}{k+N}
\\
&
=\left|1-c\binom{2N}{N}\right|+c\left[4^N-\binom{2N}{N}\right]\!
=c\,4^N-1+ 2\left[1-c\binom{2N}{N}\right]^+\!.
\end{align*}
As previously, there is an interpretation to this sum: this is the
total variation of the pseudo-distribution of $U_1$.
We can also explicitly determine the generating function of $U_1$:
for any $\zeta\in\mathbb{C}^*$,
\begin{align*}
\mathbb{E}\!\left(\zeta^{U_1}\right)
&
=\sum_{k=-N}^{N} p_k\zeta^k
=1+c \left[\,\sum_{k=-N}^{N} (-1)^{k-1} \binom{2N}{k+N} \zeta^k\right]
\\
&
=1+(-1)^{N-1}\frac{c}{\zeta^N}\left[\,\sum_{k=0}^{2N} (-1)^k \binom{2N}{k} \zeta^k\right]\!
=1+\kappa_{_{\!N}} c\,\frac{(1-\zeta)^{2N}}{\zeta^N}.
\end{align*}
We sum up below the  results we have obtained concerning the pseudo-distribution of $U_1$.
%
\begin{proposition}
The pseudo-distribution of $U_1$ is determined, for $k\in\{-N,\dots,N\}$, by
\[
\mathbb{P}\{U_1=k\}=\delta_{k0}+(-1)^{k-1}c\binom{2N}{k+N},
\]
or, equivalently, by
\[
\mathbb{P}\{U_1\le k\}=\ind_{\{k\ge 0\}}+(-1)^{k-1}c\binom{2N-1}{k+N}.
\]
The total variation of the pseudo-distribution of $U_1$ is given by
\[
\|\mathbb{P}_{U_1}\|=
\begin{cases}
1+ c\left[4^N-2\binom{2N}{N}\right] & \text{if $0<c\le 1/\binom{2N}{N}$,}
\\
c\,4^N-1 & \text{if $c\ge 1/\binom{2N}{N}$.}
\end{cases}
\]
The generating function of $U_1$ is given, for any $\zeta\in\mathbb{C}^*$, by
\begin{equation}\label{gene-function-U1}
\mathbb{E}\!\left(\zeta^{U_1}\right)\!=1+\kappa_{_{\!N}} c\,\frac{(1-\zeta)^{2N}}{\zeta^N}.
\end{equation}
In particular, the Fourier transform of $U_1$ admits the following expression:
for any $\theta\in[0,2\pi]$, by
\begin{equation*}
\mathbb{E}\big(\mathrm{e}^{\mathrm{i} \theta U_1}\big)=1-c\,4^N \sin^{2N}(\theta/2).
\end{equation*}
\end{proposition}
%
In the sequel, we shall use the total variation of $U_1$ as an upper bound
which we call $M_1$:
\begin{equation}\label{bound1}
M_1=\begin{cases}
1+ c\left[4^N-2\binom{2N}{N}\right] & \text{if $0<c\le 1/\binom{2N}{N}$,}
\\
c\,4^N-1 & \text{if $c\ge 1/\binom{2N}{N}$.}
\end{cases}
\end{equation}
Set $\mathfrak{f}(\theta)=\mathbb{E}\big(\mathrm{e}^{\mathrm{i} \theta U_1}\big)$
for any $\theta\in[0,2\pi]$.
We notice that $\mathfrak{f}(\theta)\in[1-c\,4^N,1]$ and, more precisely,
\[
\|\mathfrak{f}\|_{_\infty}=\sup_{\theta\in[0,2\pi]}|\mathfrak{f}(\theta)|=
\max\!\left(\left|1-c\,4^N\right|,1\right)=\begin{cases}
1 & \text{if $0<c\le 1/2^{2N-1}$,}
\\
c\,4^N-1 & \text{if $c\ge 1/2^{2N-1}$.}
\end{cases}
\]
Let us denote this bound by $M_{\infty}$:
\begin{equation}\label{bound}
M_{\infty}=\begin{cases}
1 & \text{if $0<c\le 1/2^{2N-1}$,}
\\
c\,4^N-1 & \text{if $c\ge 1/2^{2N-1}$.}
\end{cases}
\end{equation}
In view of~(\ref{bound1}) and~(\ref{bound}), since $\binom{2N}{N}\le 2^{2N-1}$,
we see that $M_1\ge M_{\infty}\ge 1$.
%
\begin{proposition}
The pseudo-distribution of $S_n$ is given, for any $k\in\{-Nn,\dots,Nn\}$, by
\begin{equation}\label{dist-Sn}
\mathbb{P}\{S_n=k\}=(-1)^{k} \sum_{\ell=0}^n (-c)^{\ell}\binom{n}{\ell}\!
\binom{2N\ell}{k+N\ell}.
\end{equation}
Actually, the foregoing sum is taken over the $\ell$ such that $\ell\ge |k|/N$.
We also have that
\begin{equation}\label{dist-Snbis}
\mathbb{P}\{S_n\le k\}
=\ind_{\{k\ge0\}}+(-1)^{k} \sum_{\ell=1}^n (-c)^{\ell}\binom{n}{\ell}\!
\binom{2N\ell-1}{k+N\ell}.
\end{equation}
\end{proposition}
%
\begin{proof}
By the independence of the $U_j$'s which have the same pseudo-probability
distribution, we plainly have that
\[
\mathbb{E}\big(\mathrm{e}^{\mathrm{i} \theta S_n}\big)=\mathfrak{f}(\theta)^n
=\left[1-c\,4^{N} \sin^{2N}(\theta/2)\right]^n.
\]
Hence, by inverse Fourier transform, we extract that
\begin{align}
\mathbb{P}\{S_n=k\}
&
=\frac{1}{2\pi}\int_{0}^{2\pi} \mathfrak{f}(\theta)^n \mathrm{e}^{-\mathrm{i} k\theta}
\,\mathrm{d}\theta \label{PSn-inter0}\\
&
=\sum_{\ell=0}^n (-4^Nc)^{\ell} \binom{n}{\ell}
\frac{1}{2\pi}\int_{0}^{2\pi} \sin^{2N\ell}(\theta/2) \mathrm{e}^{-\mathrm{i} k\theta}
\,\mathrm{d}\theta.
\label{PSn-inter1}
\end{align}
By writing $\sin(\theta/2)=(\mathrm{e}^{\mathrm{i} \theta/2}
-\mathrm{e}^{-\mathrm{i}\theta/2})/(2\mathrm{i})$, we get for the
integral lying in~(\ref{PSn-inter1}) that
\begin{align}
\frac{1}{2\pi}\int_{0}^{2\pi} \sin^{2N\ell}(\theta/2) \mathrm{e}^{-\mathrm{i} k\theta}
\,\mathrm{d}\theta &
=\frac{1}{2\pi}\int_{0}^{2\pi} \left(\frac{\mathrm{e}^{\mathrm{i} \theta/2}-
\mathrm{e}^{-\mathrm{i} \theta/2}}{2\mathrm{i}}\right)^{\!2N\ell}
\mathrm{e}^{-\mathrm{i} k\theta}\,\mathrm{d}\theta \nonumber\\
&
=\frac{(-1)^{N\ell}}{(2\pi) 4^{N\ell}}\int_{0}^{2\pi} \sum_{m=0}^{2N\ell}
(-1)^m\binom{2N\ell}{m} \mathrm{e}^{\mathrm{i} (m-k-N\ell)\theta}\,
\mathrm{d}\theta \nonumber\\
&
=\frac{(-1)^{N\ell}}{4^{N\ell}}\sum_{m=0}^{2N\ell}(-1)^m\binom{2N\ell}{m}
\frac{1}{2\pi}\int_{0}^{2\pi}\mathrm{e}^{\mathrm{i} (m-k-N\ell)\theta}\,
\mathrm{d}\theta \nonumber\\
&
=\frac{(-1)^k}{4^{N\ell}}\binom{2N\ell}{k+N\ell}.
\label{PSn-inter2}
\end{align}
By plugging~(\ref{PSn-inter2}) into~(\ref{PSn-inter1}), we derive~(\ref{dist-Sn}).
Next, we write, for $k\in\{-Nn,\dots,Nn\}$, that
\begin{align}
\mathbb{P}\{S_n\le k\}
&
=\sum_{j=-Nn}^k \mathbb{P}\{S_n=j\}
=\sum_{\ell=0}^n (-c)^{\ell}\binom{n}{\ell}
\sum_{j=(-Nn)\vee(-N\ell)}^{k\wedge(N\ell)} (-1)^j\binom{2N\ell}{j+N\ell}
\nonumber\\
&
=\sum_{\ell=0}^n (-c)^{\ell}\binom{n}{\ell}
\sum_{j=-N\ell}^{k\wedge(N\ell)} (-1)^j\binom{2N\ell}{j+N\ell}.
\label{sum-Sn}
\end{align}
If $k<0$, then the term in sum~(\ref{sum-Sn}) corresponding to $\ell=0$
vanishes and
\[
\mathbb{P}\{S_n\le k\}=\sum_{\ell=1}^n (-c)^{\ell}\binom{n}{\ell}
\sum_{j=-N\ell}^{k} (-1)^j\binom{2N\ell}{j+N\ell}.
\]
The second sum in the foregoing equality is easy to compute:
\begin{align}
\sum_{j=-N\ell}^{k} (-1)^j\binom{2N\ell}{j+N\ell}
&
=\sum_{j=-N\ell}^k \left[(-1)^j\binom{2N\ell-1}{j+N\ell-1}
-(-1)^{j+1}\binom{2N\ell-1}{j+N\ell}\right]
\nonumber\\
&
=(-1)^k\binom{2N\ell-1}{k+N\ell}.
\label{sum-Snbis}
\end{align}
If $k\ge0$, then the term in sum~(\ref{sum-Sn}) corresponding to $\ell=0$
is $1$ and
\[
\mathbb{P}\{S_n\le k\}=1+\sum_{\ell=1}^n (-c)^{\ell}\binom{n}{\ell}
\sum_{j=-N\ell}^{k\wedge(N\ell)} (-1)^j\binom{2N\ell}{j+N\ell}.
\]
By using the convention that $\binom{\alpha}{\beta}=0$
if $\beta>\alpha$, we see that the second sum above also coincides
with (\ref{sum-Snbis}). Formula~(\ref{dist-Snbis}) ensues in both cases.
\end{proof}
%
\begin{proposition}
The upper bound below holds true: for any positive integer $n$ and any integer $k$,
\begin{equation}\label{bound-Sn-part}
|\mathbb{P}\{S_n=k\}|\le \sqrt{\mathbb{P}\{S_{2n}=0\}}\le M_{\infty}^n.
\end{equation}
Assume that $0<c\le 1/2^{2N-1}$. The asymptotics below holds true:
for any $\delta\in(0,1/(2N))$,
\begin{equation}\label{asymptotics-Sn}
\mathbb{P}\{S_n=0\} \underset{n\to +\infty}{=}\mathcal{O}\left(1/n^{\delta}\right).
\end{equation}
\end{proposition}
%
\begin{proof}
Let us introduce the usual norms of any suitable function $\phi$:
\[
\|\phi\|_{_1}=\frac{1}{2\pi}\int_{0}^{2\pi} |\phi(\theta)| \,\mathrm{d}\theta,\quad
\|\phi\|_{_2}=\sqrt{\frac{1}{2\pi}\int_{0}^{2\pi} |\phi(\theta)|^2 \,\mathrm{d}\theta},\quad
\|\phi\|_{_{\infty}}=\sup_{\theta\in[0,2\pi]} |\phi(\theta)|
\]
and recall the elementary inequalities $\|\phi\|_{_1}\le \|\phi\|_{_2}
\le \|\phi\|_{_{\infty}}$.

It is clear from~(\ref{PSn-inter0}) that, for any integer $k$,
\begin{align*}
|\mathbb{P}\{S_n=k\}|
&
\le\frac{1}{2\pi}\int_{0}^{2\pi} |\mathfrak{f}(\theta)|^n \,\mathrm{d}\theta
=\|\mathfrak{f}^n\|_{_1}
\\
&
\le \|\mathfrak{f}^n\|_{_2}
=\sqrt{\frac{1}{2\pi}\int_{0}^{2\pi} \mathfrak{f}(\theta)^{2n} \,\mathrm{d}\theta}
=\sqrt{\mathbb{P}\{S_{2n}=0\}}
\\
&
\le \|\mathfrak{f}\|_{_\infty}^n=M_{\infty}^n.
\end{align*}
This proves (\ref{bound-Sn-part}). Next, by (\ref{PSn-inter0}),
since $\mathfrak{f}(2\pi-\theta)=\mathfrak{f}(\theta)$,
we have, for any $\varepsilon\in(0,\pi)$, that
\begin{align*}
\mathbb{P}\{S_n=0\}
&
=\frac{1}{2\pi}\int_{0}^{2\pi} \mathfrak{f}(\theta)^n \,\mathrm{d}\theta
=\frac{1}{\pi}\int_{0}^{\pi} \mathfrak{f}(\theta)^n \,\mathrm{d}\theta
\\
&
=\frac{1}{\pi}\left(\int_{0}^{\varepsilon} \mathfrak{f}(\theta)^n \,\mathrm{d}\theta
+\int_{\varepsilon}^{\pi} \mathfrak{f}(\theta)^n \,\mathrm{d}\theta\right)\!.
\end{align*}
The assumption $0<c<1/2^{2N-1}$ entails that $|\mathfrak{f}(\theta)|<1$
for any $\theta\in(0,\pi)$. We see that
$|\mathfrak{f}(\theta)|\le 1$ on $[0,\varepsilon]$, and
$|\mathfrak{f}(\theta)|\le|\mathfrak{f}(\varepsilon)|$
on $[\varepsilon,\pi]$ for any $\varepsilon\in (0,\pi)$.
Hence,
\[
|\mathbb{P}\{S_n=0\}|\le \varepsilon+|\mathfrak{f}(\varepsilon)|^n.
\]
Now, choose $\varepsilon=1/n^{\delta}$ for a positive $\delta$.
We have that
\[
\ln(|\mathfrak{f}(\varepsilon)|^n)=n\ln\!\left(1-c\,4^N\sin^{2N}
\!\Big(\frac{1}{2n^{\delta}}\Big)\right)
\sim -c\, n^{1-2N\delta}
\]
which clearly entails, for large enough $n$, that
$|\mathfrak{f}(\varepsilon)|^n\le \exp\left(-\frac{c}{2} \, n^{1-2N\delta}\right)$.
Thus, if $\delta\in(0,1/(2N))$, $|\mathfrak{f}(\varepsilon)|^n=o(\varepsilon)$
which proves (\ref{asymptotics-Sn}).

If $c=1/2^{2N-1}$, $\mathfrak{f}(\theta)=1-2 \sin^{2N}(\theta/2)$.
In this case, the same holds true upon splitting the integral
$\int_{0}^{\pi}$ into $\int_{0}^{\varepsilon}
+\int_{\varepsilon}^{\pi-\varepsilon}+\int_{\pi-\varepsilon}^{\pi}$.
\end{proof}
%
\begin{remark}
A better estimate for $|\mathbb{P}\{S_n=k\}|$ can be obtained in the same way:
\[
\forall k\in\mathbb{Z},\,|\mathbb{P}\{S_n=k\}| \le \left\{ \begin{array}{ll}
\mathbb{P}\{S_n=0\} & \mbox{ if $n$ is even,}
\\[1ex]
M_{\infty}\,\mathbb{P}\{S_{n-1}=0\} & \mbox{ if $n$ is odd.}
\end{array}\right.
\]
Nevertheless, we shall not use it.
We also have the following inequality for the total variation of $S_n$:
\[
\|\mathbb{P}_{S_n}\|=\sum_{k=-Nn}^{Nn} |\mathbb{P}\{S_n=k\}|
\le \|\mathbb{P}_{U_1}\|^n=M_1^n.
\]
\end{remark}
%
\begin{proposition}
For any bounded function $F$ defined on $\mathbb{Z}^n$,
\begin{equation}\label{bound-Fbis}
|\mathbb{E}[F(S_1,\dots,S_n)]| \le \|F\|_{_\infty} M_1^n.
\end{equation}
\end{proposition}
%
\begin{proof}
Recall that we set $p_k=\mathbb{P}\{U_1=k\}$ for any $k\in\{-N,\dots,N\}$.
We extend these settings by putting $p_k=0$ for
$k\in\mathbb{Z}\backslash\{-N,\dots,N\}$. We have that
\begin{align*}
|\mathbb{E}[F(S_1,\dots,S_n)]|
&
=\left|\sum_{(k_1,\dots,k_n)\in\mathbb{Z}^n} F(k_1,\dots,k_n)
\mathbb{P}\{S_1=k_1,\dots,S_n=k_n\}\right|
\\
&
\le \|F\|_{_\infty} \sum_{(k_1,\dots,k_n)\in\mathbb{Z}^n}
|p_{k_1}p_{k_2-k_1}\dots p_{k_n-k_{n-1}}|.
\end{align*}
The foregoing sum can be easily evaluated as follows:
\begin{align*}
\lqn{\sum_{(k_1,\dots,k_n)\in\mathbb{Z}^n}
|p_{k_1}p_{k_2-k_1}\cdots p_{k_n-k_{n-1}}|}
\\[-5ex]
&
=\sum_{k_1\in\mathbb{Z}} |p_{k_1}| \left(\sum_{k_2\in\mathbb{Z}} |p_{k_2-k_1}|
\left(\cdots \left(\sum_{k_n\in\mathbb{Z}} |p_{k_n-k_{n-1}}|\right)\cdots\right)\!\!\right)
\\
&
=\left(\sum_{k\in\mathbb{Z}} |p_{k}|\right)^{\!n}=M_1^n
\end{align*}
which proves~(\ref{bound-Fbis}).
\end{proof}
%

\section{Generating function of $S_n$}

Let us introduce the generating functions, defined for complex numbers $z,\zeta$, by
\begin{align*}
G_k(z)
&=\sum_{n\in\mathbb{N}} \mathbb{P}\{S_n=k\}z^n=\sum_{n\in\mathbb{N}:
\atop n\ge |k|/N} \mathbb{P}\{S_n=k\}z^n\text{ for $k\in\mathbb{Z}$},
\\
\mathbf{G}(\zeta,z)
&
=\sum_{k\in\mathbb{Z},n\in\mathbb{N}} \mathbb{P}\{S_n=k\}\zeta^kz^n
=\sum_{k\in\mathbb{Z},n\in\mathbb{N}: \atop |k|\le Nn} \mathbb{P}\{S_n=k\}\zeta^kz^n.
\end{align*}
We first study the problem of convergence of the foregoing series.
We start from
\[
\sum_{k\in\mathbb{Z},n\in\mathbb{N}} \left|\mathbb{P}\{S_n=k\}\zeta^kz^n\right|
\le\sum_{k\in\mathbb{Z},n\in\mathbb{N}:\atop |k|\le Nn} |\zeta|^k|M_{\infty}z|^n.
\]
If $\zeta\neq 0$ and $|\zeta|\neq 1$, then
\begin{align*}
\sum_{k\in\mathbb{Z},n\in\mathbb{N}:\atop |k|\le Nn} |\zeta|^k|M_{\infty}z|^n
&
= \sum_{n=0}^{\infty} \left(\frac{1-|\zeta|^{Nn+1}}{1-|\zeta|}
+\frac{1-1/|\zeta|^{Nn+1}}{1-1/|\zeta|}-1\right) |M_{\infty}z|^n
\\
&
=\frac{1}{1-|\zeta|} \left(\sum_{n=0}^{\infty} |M_{\infty}z|^n -|\zeta|
\sum_{n=0}^{\infty} \left|M_{\infty}z\zeta^N\right|^n\right)
\\
&
\hphantom{=\;}
+\frac{1}{1-1/|\zeta|}\left(\sum_{n=0}^{\infty} |M_{\infty}z|^n -\frac{1}{|\zeta|}
\sum_{n=0}^{\infty} \left|M_{\infty}z/\zeta^N\right|^n\right)
- \sum_{n=0}^{\infty} |M_{\infty}z|^n.
\end{align*}
If we choose $z,\zeta$ such that $|M_{\infty}z|<1$, $\left|M_{\infty}z\zeta^N\right|<1$
and $\left|M_{\infty}z/\zeta^N\right|<1$ (which is equivalent to
$|z|<\frac{1}{M_{\infty}} [\min(|\zeta|,1/|\zeta|)]^N$,
or $\sqrt[N\!\!]{|M_{\infty}z|}<|\zeta|<1/\!\sqrt[N\!\!]{|M_{\infty}z|}$), then
the double sum defining the function $\mathbf{G}(\zeta,z)$ is absolutely summable.
If $|\zeta|=1$, then
\[
\sum_{k\in\mathbb{Z},n\in\mathbb{N}:\atop |k|\le Nn} |\zeta|^k|M_{\infty}z|^n
=\sum_{n=0}^{\infty} (2Nn+1)|M_{\infty}z|^n.
\]
If we choose $z$ such that $|M_{\infty}z|<1$, then the same conclusion holds.

Now, we have that
\[
\mathbf{G}(\zeta,z)=\sum_{n\in\mathbb{N}}\left(\sum_{k\in\mathbb{Z}}
\mathbb{P}\{S_n=k\}\zeta^k\right)\!z^n
=\sum_{n\in\mathbb{N}} \mathbb{E}\!\left( \zeta^{S_n}\right)\!z^n
=\sum_{n\in\mathbb{N}} \left[z\mathbb{E}\!\left(\zeta^{U_1}\right)\right]^n
=\frac{1}{1-z\mathbb{E}(\zeta^{U_1})}
\]
and, thanks to~(\ref{gene-function-U1}), we can state the following result.
%
\begin{proposition}\label{double-gene}
The double generating function of the $\mathbb{P}\{S_n=k\}$,
$k\in\mathbb{Z},n\in\mathbb{N}$, is given,
for any complex numbers $z,\zeta$ such that $\sqrt[N\!\!]{|M_{\infty}z|}<|\zeta|<1/\!\sqrt[N\!\!]{|M_{\infty}z|}$, by
\begin{equation}\label{functionG}
\mathbf{G}(\zeta,z)=\frac{\zeta^N}{(1-z)\zeta^N-\kappa_{_{\!N}} cz(1-\zeta)^{2N}}.
\end{equation}
In particular, for any $\theta\in[0,2\pi]$ and $z\in\mathbb{C}$ such that $|z|<1/M_{\infty}$,
\begin{equation}\label{functionGbis}
\mathbf{G}(\mathrm{e}^{\mathrm{i} \theta},z)=\frac{1}{1-z+c\,4^Nz\sin^{2N}(\theta/2)}.
\end{equation}
\end{proposition}
%
On the other hand,
\[
\mathbf{G}(\zeta,z)=\sum_{k\in\mathbb{Z}}\left(\sum_{n\in\mathbb{N}}
\mathbb{P}\{S_n=k\}z^n\right)\zeta^k
=\sum_{k\in\mathbb{Z}} G_k(z)\zeta^k.
\]
By substituting $\zeta=\mathrm{e}^{\mathrm{i} \theta}$ in the foregoing equality, we get the
Fourier series of the function $\theta\mapsto \mathbf{G}(\mathrm{e}^{\mathrm{i} \theta},z)$:
\[
\mathbf{G}(\mathrm{e}^{\mathrm{i} \theta},z)=\sum_{k\in\mathbb{Z}}
G_k(z) \mathrm{e}^{\mathrm{i} k\theta}
\]
from which we extract the sequence of the coefficients $(G_k(z))_{k\in\mathbb{N}}$.
Indeed, since $\mathbb{P}\{S_n=-k\}=\mathbb{P}\{S_n=k\}$, we have that
$G_k(z)=G_{-k}(z)$ and
\begin{align*}
G_k(z)=\frac{1}{2\pi}\int_{0}^{2\pi} \mathbf{G}(\mathrm{e}^{\mathrm{i} \theta},z)
\mathrm{e}^{-\mathrm{i} k\theta}\,\mathrm{d}\theta
=\frac{1}{2\pi}\int_{0}^{2\pi} \mathbf{G}(\mathrm{e}^{\mathrm{i} \theta},z)
\mathrm{e}^{\mathrm{i} k\theta}\,\mathrm{d}\theta
=\frac{1}{2\pi \mathrm{i}}\int_{\mathcal{C}} \mathbf{G}(u,z)u^{k-1}\,\mathrm{d} u
\end{align*}
where $\mathcal{C}$ is the circle of radius 1 centered at the origin and unclockwise
orientated. Then, referring to~(\ref{functionG}), we obtain, for any $z\in\mathbb{C}$
satisfying $|z|<1/M_{\infty}$, that
\[
G_k(z)=\frac{1}{2\pi \mathrm{i}}\int_{\mathcal{C}} \frac{u^{k+N-1}}{P_z(u)}\,\mathrm{d} u
\]
where $P_z$ is the polynomial given by
\[
P_z(u)=(1-z)u^N-\kappa_{_{\!N}} cz(u-1)^{2N}.
\]
We are looking for the roots of $P_z$ which lie inside the circle $\mathcal{C}$.
For this, we introduce the $N^{\mathrm{th}}$ roots of $\kappa_{_{\!N}}$:
$\theta_j=-\,\mathrm{e}^{\mathrm{i} \frac{2j-1}{N}\pi}$ for $1\le j\le N$;
$\theta_j^N=\kappa_{_{\!N}}$.

From now on, in order to simplify the writing of the roots of $P_z$,
we make the assumption that $z$ is a real number lying in $(0,1/M_{\infty})$
(and then $z\in(0,1)$).
The roots of $P_z$ are those of the equations
$u^2-2[1+\theta_j w(z)]u+1$, $1\le j\le N$, where
\[
w(z)=\frac{\sqrt[N\!]{1-z} }{2\sqrt[N\!] {cz}}.
\]
They can be written as
\begin{align*}
u_j(z)
&
=1+\theta_j w(z)-\theta_j\sqrt{w(z)}\,[a_j(z)+\mathrm{i}\epsilon_j b_j(z)],\quad 1\le j\le N,
\\
v_j(z)
&
=1+\theta_j w(z)+\theta_j\sqrt{w(z)}\,[a_j(z)+\mathrm{i}\epsilon_j b_j(z)],\quad 1\le j\le N,
\end{align*}
with
\begin{align*}
\epsilon_j
&
\textstyle=\mathop{\mathrm{sgn}}\!\Big(\sin\!\big(\frac{2j-1}{N}\pi \big)\Big)
\;\text{(with the convention that $\mathop{\mathrm{sgn}}(0)=0$)},
\\
a_j(z)
&
\textstyle=\frac{1}{\sqrt2}\Big[\sqrt{w(z)^2-4\cos\!\big(\frac{2j-1}{N}\pi\big)
w(z)+4}+w(z)-2\cos\!\big(\frac{2j-1}{N}\pi\big)\!\Big]^{1/2},
\\
b_j(z)
&
\textstyle=\frac{1}{\sqrt2}\Big[\sqrt{w(z)^2-4\cos\!\big(\frac{2j-1}{N}\pi\big)
w(z)+4}-w(z)+2\cos\!\big(\frac{2j-1}{N}\pi\big)\!\Big]^{1/2}.
\end{align*}
We notice that $a_j(z)b_j(z)=\big|\sin\big(\frac{2j-1}{N}\pi \big)\big|$.
Because of the last coefficient $1$ in the polynomial $u^2-2[1+\theta_j w(z)]u+1$,
it is clear that the roots $u_j(z)$ and $v_j(z)$ are inverse: $v_j(z)=1/u_j(z)$.

Let us check that $|u_j(z)|<1<|v_j(z)|$ for any $j\in\{1,\dots,N\}$.
Straightforward computations yield that
\[
|u_j(z)|^2=A_j(z)-B_j(z)\quad\text{and}\quad |v_j(z)|^2=A_j(z)+B_j(z)
\]
where
\begin{align*}
A_j(z)
&
\textstyle=w(z)^2+w(z)\Big[\sqrt{w(z)^2-4\cos\!\big(\frac{2j-1}{N}\pi\big)
w(z)+4}-2\cos\!\big(\frac{2j-1}{N}\pi\big)\Big]+1,
\\
B_j(z)
&
\textstyle=2\sqrt{w(z)}\Big[\big(w(z)-\cos\!\big(\frac{2j-1}{N}\pi\big)\big)a_j(z)
+\big|\sin\!\big(\frac{2j-1}{N}\pi\big)\big|\, b_j(z)\Big].
\end{align*}
Since $v_j(z)=1/u_j(z)$, checking that $|u_j(z)|<1<|v_j(z)|$ is equivalent
to checking that $|u_j(z)|^2<|v_j(z)|^2$, that is $B_j(z)>0$.
If $\sin\!\big(\frac{2j-1}{N}\pi\big)\neq 0$, we have $a_j(z)b_j(z)\neq 0$,
$a_j(z)=\big|\sin\!\big(\frac{2j-1}{N}\pi\big)\big|/b_j(z)$ and then
\begin{align*}
B_j(z)
&
\textstyle =2a_j(z)\sqrt{w(z)}\Big[w(z)-\cos\!\big(\frac{2j-1}{N}\pi\big)
+b_j(z)^2\Big]
\\
&
\textstyle =a_j(z)\sqrt{w(z)}\Big[\sqrt{w(z)^2-4\cos\!\big(\frac{2j-1}{N}
\pi\big)w(z)+4}+w(z)\Big]>0.
\end{align*}
If $\sin\!\big(\frac{2j-1}{N}\pi\big)=0$ (which happens only when $N$ is odd
and $j=(N+1)/2$), $a_j(z)=\sqrt{w(z)+2}$, $b_j(z)=0$ and then
$B_j=2a_j(z)\sqrt{w(z)}(w(z)+1)>0$.

The above discussion ensures that the roots we are looking for
(i.e. those lying inside $\mathcal{C}$) are $u_j(z)$, $1\le j\le N$; we discard
the $v_j(z)$'s.

%
\begin{remark}
We notice that
\[
\lim_{z\to 1^-} a_j(z) =\textstyle\sqrt2 \sin\!\left(\!\!
\vphantom{\frac 2N}\right.\frac{2j-1}{2N}\pi \left.\!\!\vphantom{\frac 2N}\right),
\quad
\displaystyle \lim_{z\to 1^-} b_j(z) =\textstyle\epsilon_j\sqrt2 \cos\!\left(\!\!
\vphantom{\frac 2N}\right.\frac{2j-1}{2N}\pi \left.\!\!\vphantom{\frac 2N}\right)
\]
and then
\[
\lim_{z\to 1^-} \theta_j\,[a_j(z)+\mathrm{i}\epsilon_j b_j(z)]=\sqrt2 \,\varphi_j
\]
where we set $\varphi_j=-\mathrm{i}\,\mathrm{e}^{\mathrm{i} \frac{2j-1}{2N}\pi}$.
The $\varphi_j$, $1\le j\le N$
are the $(2N)^{\mathrm{th}}$ roots of $\kappa_{_{\!N}}$ with positive real part:
$\varphi_j^{2N}=\kappa_{_{\!N}}$ and $\Re(\varphi_j)>0$.
As a result, we derive the asymptotic, which will be used further,
\begin{equation}\label{asymptotic-u}
u_j(z)=1+\varepsilon_j(z)\quad\text{with}\quad
\varepsilon_j(z)\underset{z\to 1^-}{\sim} -\frac{\varphi_j}{\sqrt[2N\!] c}\sqrt[2N\!]{1-z}
=\mathcal{O}\!\left(\!\sqrt[2N\!]{1-z}\,\right)\!.
\end{equation}
\end{remark}
%
\begin{example}\label{example-roots}
For $N=2$, the roots explicitly write as
\begin{align*}
u_1(z)&=[1-b_1(z)w(z)]-\mathrm{i}\left[w(z)-a_1(z)\sqrt{w(z)}\,\right]\!,
\\
u_2(z)&=\overline{u_1(z)},\quad v_1(z)=1/u_1(z),\quad v_2(z)=\overline{v_1(z)},
\end{align*}
with
\[
w(z)=\frac{\sqrt{1-z}}{2\sqrt{cz}},\quad
a_1(z)=\frac{1}{\sqrt 2}\sqrt{\sqrt{w(z)^2+4}+w(z)},\quad
b_1(z)=\frac{1}{\sqrt 2}\sqrt{\sqrt{w(z)^2+4}-w(z)}.
\]
For $N=3$, the roots explicitly write as
\begin{align*}
u_1(z)&=\frac{1}{2}\left[2-w(z)-\left(a_1(z)-b_1(z)\sqrt3\,\right)\sqrt{w(z)}\,\right]
\\
&\hphantom{=\;}
-\frac{\mathrm{i}}{2}\left[w(z)\sqrt3-\left(a_1(z)\sqrt3+b_1(z)\right)\sqrt{w(z)}\,\right]\!,
\\
u_2(z)&=1-w(z)-a_2(z)\sqrt{w(z)} ,\quad u_3(z)=\overline{u_1(z)},
\\
v_1(z)&=1/u_1(z),\quad v_2(z)=1/u_2(z),\quad v_3(z)=\overline{v_1(z)},
\end{align*}
with
\begin{align*}
a_1(z)&=\frac{1}{\sqrt 2}\sqrt{\sqrt{w(z)^2-2w(z)+4}+w(z)-1},
\\
b_1(z)&=\frac{1}{\sqrt 2}\sqrt{\sqrt{w(z)^2-2w(z)+4}-w(z)+1},
\\
a_2(z)&=\sqrt{w(z)+2},\quad w(z)=\frac{\sqrt[3]{1-z}}{2\sqrt[3]{cz}}.
\end{align*}
\end{example}

Now, $G_k(z)$ can be evaluated by residues theorem. Suppose first that $k\ge 0$
(then $k+N-1\ge 0$) so that $0$ is not a pole in the integral defining $G_k(z)$:
\begin{align}\label{expression-Gk1}
G_k(z)
&
=\sum_{j=1}^N \mathrm{Res}\!\left(\frac{u^{k+N-1}}{P_z(u)},u_j(z)\right)
=\sum_{j=1}^N \frac{u_j(z)^{k+N-1}}{P'_z(u_j(z))}
\nonumber\\
&
=\frac{1}{N(1-z)} \sum_{j=1}^N \frac{1-u_j(z)}{1+u_j(z)}\,u_j(z)^{k}.
\end{align}
The foregoing representation of $G_k(z)$ is valid \textit{a priori} for any
$z\in(0,1/M_{\infty})$. Actually, in view of the expressions of $w(z)$
and $u_j(z)$, we can see that (\ref{expression-Gk1}) defines an analytical
function in the interval $(0,1)$.
Since $G_k(z)$ is a power series, by analytical continuation,
equality~(\ref{expression-Gk1}) holds true for any $z\in(0,1)$.
Moreover, by symmetry, we have that $G_k(z)=G_{-k}(z)$ for $k\le 0$.
We display this result in the theorem below.
%
\begin{theorem}\label{theo-gene-function}
For any $k\in\mathbb{Z}$, the generating function of the $\mathbb{P}\{S_n=k\}$,
$n\in\mathbb{N}$, is given, for any $z\in(0,1)$, by
\begin{equation}\label{expression-Gk}
G_k(z)=\frac{1}{N(1-z)} \sum_{j=1}^N \frac{1-u_j(z)}{1+u_j(z)}\,u_j(z)^{|k|}.
\end{equation}
\end{theorem}
%
\begin{remark}
Another proof of Theorem~\ref{theo-gene-function} consists in expanding the
rational fraction $\zeta\mapsto \mathbf{G}(\zeta,z)$ into partial fractions.
We find it interesting to outline the main steps of this method.
We can write that
\[
\mathbf{G}(\zeta,z)=\sum_{j=1}^N \frac{A_j(z)}{\zeta-u_j(z)}
+\sum_{j=1}^N \frac{B_j(z)}{\zeta-v_j(z)}
\]
with
\[
A_j(z)=\frac{u_j(z)}{N(1-z)}\frac{1-u_j(z)}{1+u_j(z)},\quad
B_j(z)=\frac{v_j(z)}{N(1-z)}\frac{1-v_j(z)}{1+v_j(z)}
=-\frac{1/u_j(z)}{N(1-z)}\frac{1-u_j(z)}{1+u_j(z)}.
\]
We next expand the partial fractions $\frac{1}{\zeta-u_j(z)}$ and
$\frac{1}{\zeta-v_j(z)}$ into power series as follows. We have
checked that $|u_j(z)|<1<|v_j(z)|$ for any $j\in\{1,\dots,N\}$.
Now, if $|u_j(z)|<|\zeta|<|v_j(z)|$ for any $j\in\{1,\dots,N\}$,
\[
\frac{1}{\zeta-u_j(z)}=\sum_{k=0}^{\infty}\frac{u_j(z)^k}{\zeta^{k+1}}
=\sum_{k=-\infty}^{-1}\frac{\zeta^k}{u_j(z)^{k+1}},
\quad\frac{1}{\zeta-v_j(z)}=-\sum_{k=0}^{\infty}\frac{\zeta^k}{v_j(z)^{k+1}}
\]
from which~(\ref{expression-Gk}) can be easily extracted.
\end{remark}
%

\section{Limiting pseudo-process}\label{section-limit}

In this section, by pseudo-process it is meant a continuous-time process
driven by a signed measure. Actually, this object is not properly defined
on all continuous times but only on dyadic times $k/2^j$, $j,k\in\mathbb{N}$.
A proper definition consists in seeing it
as the limit of a step process associated with the observations of the
pseudo-process on the dyadic times. We refer the reader to~\cite{la3}
and~\cite{nish2} for precise details which are cumbersome to reproduce here.

Below, we give an ad hoc definition for the convergence of a family of
pseudo-processes $((X_t^{\varepsilon})_{t\ge 0})_{\varepsilon>0}$ towards a
pseudo-process $(X_t)_{t\ge 0}$.
%
\begin{definition}\label{def1}
Let $((X_t^{\varepsilon})_{t\ge 0})_{\varepsilon>0}$ be a family of pseudo-processes and
$(X_t)_{t\ge 0}$ be a pseudo-process. We say that
\[
(X_t^{\varepsilon})_{t\ge 0}\underset{\varepsilon\to 0^+}{\longrightarrow} (X_t)_{t\ge 0}
\]
if and only if
\[
\forall n\in\mathbb{N}^*, \, \forall t_1,\dots,t_n\ge 0, \,
\forall \mu_1,\dots,\mu_n\in\mathbb{R},
\quad\mathbb{E}\!\left(\mathrm{e}^{\mathrm{i} \sum_{k=1}^n\mu_k
X_{t_k}^{\varepsilon}}\right)
\underset{\varepsilon\to 0^+}{\longrightarrow}\mathbb{E}\!\left(\mathrm{e}^{\mathrm{i}
\sum_{k=1}^n\mu_k X_{t_k}}\right)\!.
\]
\end{definition}
%
This is the weak convergence of the finite-dimensional projections of the family
of pseudo-processes.

In this part, we choose for the family $((X_t^{\varepsilon})_{t\ge 0})_{\varepsilon>0}$
the continuous-time pseudo-processes defined, for any $\varepsilon>0$, by
\[
X_t^{\varepsilon}=\varepsilon S_{\lfloor t/\varepsilon^{2N}\rfloor},\quad t\ge 0,
\]
where $\lfloor\,.\,\rfloor$ stands for the usual floor function.
The quantity $X_t^{\varepsilon}$ takes its values on the discrete set $\varepsilon\mathbb{Z}$.
Roughly speaking, we normalize the pseudo-random walk on the time$\times$space grid
$\varepsilon^{2N}\mathbb{N}\times \varepsilon\mathbb{Z}$. Let $(X_t)_{t\ge 0}$
be the pseudo-Brownian motion.
It is characterized by the following property:
for any $n\in\mathbb{N}^*$, any $t_1,\dots,t_n\ge0$ such that
$t_1<\dots<t_n$ and any $\mu_1,\dots,\mu_n\in\mathbb{R}$,
\begin{equation}\label{FT-BM}
\mathbb{E}\!\left(\mathrm{e}^{\mathrm{i} \sum_{k=1}^n\mu_k X_{t_k}}\right)\!
=\mathrm{e}^{-c \sum_{k=1}^n(\mu_1+\dots+\mu_k)^{2N}(t_k-t_{k-1})}.
\end{equation}
We refer to~\cite{la3} and~\cite{nish2} for a proper definition of pseudo-Brownian
motion, and to references therein for interesting properties of this pseudo-process.
%
\begin{theorem}\label{theo-limit}
Suppose that $c\le 1/2^{2N-1}$. The following convergence holds:
\[
(X_t^{\varepsilon})_{t\ge 0}\underset{\varepsilon\to 0^+}{\longrightarrow} (X_t)_{t\ge 0}.
\]
\end{theorem}
%
\begin{proof}
$\bullet$ We begin by computing the Laplace-Fourier transform of $X_t^{\varepsilon}$.
By definition of $X_t^{\varepsilon}$, we have that
$\mathbb{E}\!\left(\mathrm{e}^{\mathrm{i} \mu X_t^{\varepsilon}}\right)
=\mathbb{E}\!\left(\mathrm{e}^{\mathrm{i} \mu\varepsilon S_{\lfloor t/
\varepsilon^{2N}\rfloor}}\right)$ and then
\begin{align*}
\int_0^{+\infty} \mathrm{e}^{-\lambda t} \,\mathbb{E}\!\left(\mathrm{e}^{\mathrm{i}
\mu X_t^{\varepsilon}}\right) \mathrm{d} t
&=
\sum_{n=0}^{\infty} \left(\int_{n\varepsilon^{2N}}^{(n+1)\varepsilon^{2N}}
\mathrm{e}^{-\lambda t} \,\mathrm{d} t\right)
\mathbb{E}\!\left(\mathrm{e}^{\mathrm{i} \mu \varepsilon S_n}\right)
\\
&=
\frac{1-\mathrm{e}^{-\lambda \varepsilon^{2N}}}{\lambda} \sum_{n=0}^{\infty}
\left(\mathrm{e}^{-\lambda \varepsilon^{2N}}\right)^{\!n}
\mathbb{E}\!\left(\mathrm{e}^{\mathrm{i} \mu\varepsilon S_n}\right)
\\
&=
\frac{1-\mathrm{e}^{-\lambda \varepsilon^{2N}}}{\lambda}\sum_{n\in\mathbb{N},
k\in\mathbb{Z}}
\left(\mathrm{e}^{-\lambda \varepsilon^{2N}}\right)^{\!n}\!\left(\mathrm{e}^{\mathrm{i}
\mu\varepsilon}\right)^{\!k} \mathbb{P}\{S_n=k\}
\\
&=
\frac{1-\mathrm{e}^{-\lambda \varepsilon^{2N}}}{\lambda} \,\mathbf{G}\!
\left(\mathrm{e}^{\mathrm{i} \mu\varepsilon},\mathrm{e}^{-\lambda\varepsilon^{2N}}\right)\!.
\end{align*}
By (\ref{functionGbis}), we have that
\begin{equation}\label{functionGter}
\mathbf{G}\!\left(\mathrm{e}^{\mathrm{i} \mu\varepsilon},\mathrm{e}^{-\lambda\varepsilon^{2N}}\right)\!
=\frac{1}{1-\mathrm{e}^{-\lambda\varepsilon^{2N}}+ c\,4^N
\mathrm{e}^{-\lambda\varepsilon^{2N}}\sin^{2N}(\mu\varepsilon/2)}.
\end{equation}
Actually, equality~(\ref{functionGter}) is valid for $\lambda$ such that
$\mathrm{e}^{-\lambda\varepsilon^{2N}}<1/M_{\infty}$, that is, $\lambda>(\ln M_{\infty})/\varepsilon^{2N}$.
Since $c$ is assumed not to be greater than $1/2^{2N-1}$, by~(\ref{bound}), we have
that $M_{\infty}=1$ and (\ref{functionGter}) is valid for any $\lambda>0$.

Now, by using the elementary asymptotics $\sin(\mu\varepsilon/2)
\underset{\varepsilon\to 0^+}{=} \mu\varepsilon/2+o(\varepsilon)$
and $\mathrm{e}^{-\lambda\varepsilon^{2N}}\underset{\varepsilon\to 0^+}{=}
1-\lambda\varepsilon^{2N}+o(\varepsilon)$, we obtain that
\[
\mathbf{G}\!\left(\mathrm{e}^{\mathrm{i} \mu\varepsilon},
\mathrm{e}^{-\lambda\varepsilon^{2N}}\right)\underset{\varepsilon\to 0^+}{\sim}
\frac{1}{\lambda+c\mu^{2N}} \,\frac{1}{\varepsilon^{2N}}.
\]
As a result, for any $\lambda>0$,
\[
\lim_{\varepsilon\to 0^+}\int_0^{+\infty} \mathrm{e}^{-\lambda t} \,
\mathbb{E}\!\left(\mathrm{e}^{\mathrm{i} \mu X_t^{\varepsilon}}\right) \mathrm{d} t
=\frac{1}{\lambda+c\mu^{2N}}=\int_0^{+\infty} \mathrm{e}^{-(\lambda+c\mu^{2N})t}\,\mathrm{d} t
\]
from which and (\ref{FT-BM}) we deduce that
\begin{equation}\label{limit-FT}
\lim_{\varepsilon\to 0^+} \mathbb{E}\!\left(\mathrm{e}^{\mathrm{i}
\mu X_t^{\varepsilon}}\right)\!= \mathrm{e}^{-c\mu^{2N}t}
=\mathbb{E}\!\left(\mathrm{e}^{\mathrm{i} \mu X_t}\right)\!.
\end{equation}
Notice that the Laplace-Fourier of $X_t$ takes the simple form
\[
\int_0^{+\infty} \mathrm{e}^{-\lambda t} \,\mathbb{E}\!
\left(\mathrm{e}^{\mathrm{i} \mu X_t}\right) \mathrm{d} t
=\frac{1}{\lambda+c\mu^{2N}}.
\]

$\bullet$ In the same way, we compute the Laplace-Fourier transform of
$X_{t+\varepsilon^{2n}}^{\varepsilon}$
which will be used further. We have $\mathbb{E}\!\left(\mathrm{e}^{\mathrm{i}
\mu X_{t+\varepsilon^{2n}}^{\varepsilon}}\right)
=\mathbb{E}\!\left(\mathrm{e}^{\mathrm{i} \mu\varepsilon S_{\lfloor t/
\varepsilon^{2N}\rfloor+1}}\right)$. Then
\begin{align*}
\int_0^{+\infty} \mathrm{e}^{-\lambda t} \,\mathbb{E}\!\left(\mathrm{e}^{\mathrm{i}
\mu X_{t+\varepsilon^{2n}}^{\varepsilon}}\right) \mathrm{d} t
&=
\sum_{n=0}^{\infty} \left(\int_{n\varepsilon^{2N}}^{(n+1)\varepsilon^{2N}}
\mathrm{e}^{-\lambda t} \,\mathrm{d} t\right)
\mathbb{E}\!\left(\mathrm{e}^{\mathrm{i} \mu \varepsilon S_{n+1}}\right)
\\
&=
\frac{1-\mathrm{e}^{-\lambda \varepsilon^{2N}}}{\lambda} \sum_{n=0}^{\infty}
\left(\mathrm{e}^{-\lambda \varepsilon^{2N}}\right)^{\!n}
\mathbb{E}\!\left(\mathrm{e}^{\mathrm{i} \mu\varepsilon S_{n+1}}\right)
\\
&=
\frac{1-\mathrm{e}^{-\lambda \varepsilon^{2N}}}{\lambda} \,
\mathrm{e}^{\lambda\varepsilon^{2N}} \sum_{n\in\mathbb{N}^*,k\in\mathbb{Z}}
\left(\mathrm{e}^{-\lambda \varepsilon^{2N}}\right)^{\!n}\!
\left(\mathrm{e}^{\mathrm{i} \mu\varepsilon}\right)^{\!k} \mathbb{P}\{S_n=k\}
\\
&=
\frac{\mathrm{e}^{\lambda \varepsilon^{2N}}-1}{\lambda} \left[\mathbf{G}\!
\left(\mathrm{e}^{\mathrm{i} \mu\varepsilon},\mathrm{e}^{-\lambda\varepsilon^{2N}}\right)
-1\right]\!.
\end{align*}
As for~(\ref{limit-FT}), we immediately extract the following limit:
\begin{equation}\label{limit-FT-bis}
\lim_{\varepsilon\to 0^+} \mathbb{E}\!\left(\mathrm{e}^{\mathrm{i}
\mu X_{t+\varepsilon^{2N}}^{\varepsilon}}\right)\!=
\mathbb{E}\!\left(\mathrm{e}^{\mathrm{i} \mu X_t}\right)\!.
\end{equation}

$\bullet$ We now compute the joint Fourier transform of
$(X_{t_1}^{\varepsilon},X_{t_2}^{\varepsilon})$
for two times $t_1,t_2$ such that $t_1<t_2$. Using the elementary fact that
$\lfloor x\rfloor-\lfloor y\rfloor\in\{\lfloor x-y\rfloor,\lfloor x-y\rfloor+1\}$, we observe that
\[
X_{t_2}^{\varepsilon}-X_{t_1}^{\varepsilon}\stackrel{d}{=}\varepsilon
S_{\lfloor t_2/\varepsilon^{2N}\rfloor-\lfloor t_1/\varepsilon^{2N}\rfloor}
\in\{\varepsilon S_{\lfloor(t_2-t_1)/\varepsilon^{2N}\rfloor},\varepsilon
S_{\lfloor(t_2-t_1)/\varepsilon^{2N}\rfloor+1}\}
=\{X_{t_2-t_1}^{\varepsilon},X_{t_2-t_1+\varepsilon^{2N}}^{\varepsilon}\}.
\]
Then, we get, for $\mu_1,\mu_2\in\mathbb{R}$, that
\[
\mathbb{E}\Big(\mathrm{e}^{\mathrm{i} \big(\mu_1 X_{t_1}^{\varepsilon}
+\mu_2 X_{t_2}^{\varepsilon}\big)}\Big)
\in\left\{\mathbb{E}\!\left(\mathrm{e}^{\mathrm{i} (\mu_1+\mu_2) X_{t_1}^{\varepsilon}}\right)
\mathbb{E}\!\left(\mathrm{e}^{\mathrm{i} \mu_2 X_{t_2-t_1}^{\varepsilon}}\right)\!,
\mathbb{E}\!\left(\mathrm{e}^{\mathrm{i} (\mu_1+\mu_2) X_{t_1}^{\varepsilon}}\right)
\mathbb{E}\!\left(\mathrm{e}^{\mathrm{i} \mu_2 X_{t_2-t_1+\varepsilon^{2N}}^{\varepsilon}}\right)
\right\}\!.
\]
By~(\ref{limit-FT}) and~(\ref{limit-FT-bis}), we obtain the following limit:
\[
\lim_{\varepsilon\to 0^+} \mathbb{E}\!\left(\mathrm{e}^{\mathrm{i}
\mu_2 X_{t_2-t_1}^{\varepsilon}}\right)
=\lim_{\varepsilon\to 0^+} \mathbb{E}\!\left(\mathrm{e}^{\mathrm{i}
\mu_2 X_{t_2-t_1+\varepsilon^{2N}}^{\varepsilon}}\right)
=\mathbb{E}\!\left(\mathrm{e}^{\mathrm{i} \mu_2 X_{t_2-t_1}}\right)
\]
which yields that
\begin{align*}
\lim_{\varepsilon\to 0^+}\mathbb{E}\Big(\mathrm{e}^{\mathrm{i}
\big(\mu_1 X_{t_1}^{\varepsilon}+\mu_2 X_{t_2}^{\varepsilon}\big)}\Big)
&
=\lim_{\varepsilon\to 0^+}\mathbb{E}\!\left(\mathrm{e}^{\mathrm{i}
(\mu_1+\mu_2) X_{t_1}^{\varepsilon}}\right)
\times \lim_{\varepsilon\to 0^+} \mathbb{E}\!\left(\mathrm{e}^{\mathrm{i}
\mu_2 X_{t_2-t_1}^{\varepsilon}}\right)
\\
&
=\mathbb{E}\Big(\mathrm{e}^{\mathrm{i} (\mu_1+\mu_2) X_{t_1}}\Big)
\times \mathbb{E}\!\left(\mathrm{e}^{\mathrm{i} \mu_2 X_{t_2-t_1}}\right)
=\mathbb{E}\Big(\mathrm{e}^{\mathrm{i} \big(\mu_1 X_{t_1}+\mu_2 X_{t_2}\big)}\Big)\!.
\end{align*}

$\bullet$ Finally, we can easily extend the foregoing limiting result
by recurrence as follows:
for $n\in\mathbb{N}^*$, $\mu_1,\dots,\mu_n\in\mathbb{R}$ and for any times
$t_1,\dots,t_n$ such that $t_1<\dots<t_n$,
\[
\lim_{\varepsilon\to 0^+}\mathbb{E}\Big(\mathrm{e}^{\mathrm{i}
\big(\mu_1X_{t_1}^{\varepsilon}+\dots+\mu_nX_{t_n}^{\varepsilon}\big)}\Big)
=\mathbb{E}\!\left(\mathrm{e}^{\mathrm{i}(\mu_1X_{t_1}+\dots+\mu_nX_{t_n})}\right)\!.
\]
The proof of Theorem~\ref{theo-limit} is complete.
\end{proof}

We find it interesting to compute in a similar way
the $\lambda$-potential of the pseudo-process $(X_t)_{t\ge 0}$.
By definition of $X_t^{\varepsilon}$, we have, for any
$\alpha,\beta\in\mathbb{R}$ such that $\alpha<\beta$,
$\mathbb{P}\{X_t^{\varepsilon}\in[\alpha,\beta)\}=
\mathbb{P}\{S_{\lfloor t/\varepsilon^{2N}\rfloor}\in[\alpha/\varepsilon,\beta/\varepsilon)\}.$
Thus,
\begin{align*}
\int_0^{+\infty} \mathrm{e}^{-\lambda t} \,\mathbb{P}\{X_t^{\varepsilon}
\in[\alpha,\beta)\}\,\mathrm{d} t
&=
\sum_{n=0}^{\infty} \left(\int_{n\varepsilon^{2N}}^{(n+1)\varepsilon^{2N}}
\mathrm{e}^{-\lambda t} \,\mathrm{d} t\right)
\mathbb{P}\{S_n\in[\alpha/\varepsilon,\beta/\varepsilon)\}
\\
&=
\frac{1-\mathrm{e}^{-\lambda \varepsilon^{2N}}}{\lambda} \sum_{n=0}^{\infty}
\left(\mathrm{e}^{-\lambda \varepsilon^{2N}}\right)^{\!n}
\left[\!\vphantom{\sum_{n=0}^{\infty}}\right.
\sum_{k\in\mathbb{Z}:\atop\alpha/\varepsilon\le k< \beta/\varepsilon}
\mathbb{P}\{S_n=k\}\left.\vphantom{\sum_{n=0}^{\infty}}\!\right]
\\
&=
\frac{1-\mathrm{e}^{-\lambda \varepsilon^{2N}}}{\lambda}
\sum_{k\in\mathbb{Z}:\atop\alpha/\varepsilon\le k< \beta/\varepsilon} \left[
\sum_{n=0}^{\infty} \mathbb{P}\{S_n=k\}\left(\mathrm{e}^{-\lambda
\varepsilon^{2N}}\right)^{\!n} \right]
\\
&=
\frac{1-\mathrm{e}^{-\lambda \varepsilon^{2N}}}{\lambda}\sum_{k\in\mathbb{Z}:
\atop\alpha/\varepsilon\le k< \beta/\varepsilon}
G_k\!\left(\mathrm{e}^{-\lambda\varepsilon^{2N}}\right).
\end{align*}
Interchanging the two sums in the above computations is justified
by the fact that the series
$\sum_{n=0}^{\infty} \mathbb{P}\{S_n=k\}\left(\mathrm{e}^{-\lambda
\varepsilon^{2N}}\right)^{\!n}$
is absolutely convergent because of the condition $c\le 1/2^{2N-1}$.
Indeed, by~(\ref{bound-Sn-part}), for any $\lambda>0$,
$|\mathbb{P}\{S_n=k\}|< M_{\infty}^n=1$.

Put $u_j(\lambda,\varepsilon)=u_j\!\left(\mathrm{e}^{-\lambda\varepsilon^{2N}}
\right)$. This yields that
\begin{align*}
\int_0^{+\infty} \mathrm{e}^{-\lambda t} \,\mathbb{P}\{X_t^{\varepsilon}
\in[\alpha,\beta)\}\,\mathrm{d} t
&=
\frac{1}{N\lambda} \sum_{j=1}^N \frac{1-u_j(\lambda,\varepsilon)}{1+u_j(\lambda,\varepsilon)}
\sum_{\alpha/\varepsilon\le k< \beta/\varepsilon} u_j(\lambda,\varepsilon)^{|k|}.
\end{align*}
Suppose, e.g., that $0\le\alpha<\beta$. Then,
\[
\sum_{\alpha/\varepsilon\le k<\beta/\varepsilon} u_j(\lambda,\varepsilon)^{|k|}
=\sum_{k=\lceil\alpha/\varepsilon\rceil}^{\lceil\beta/\varepsilon\rceil-1}
u_j(\lambda,\varepsilon)^k
=\frac{u_j(\lambda,\varepsilon)^{\lceil\alpha/\varepsilon\rceil}-
u_j(\lambda,\varepsilon)^{\lceil\beta/\varepsilon\rceil}}{1-u_j(\lambda,\varepsilon)},
\]
where $\lceil\,.\,\rceil$ stands for the usual ceiling function.
By using~(\ref{asymptotic-u}), we deduce that
\begin{equation}\label{asymptotic-v}
u_j(\lambda,\varepsilon)\underset{\varepsilon\to 0^+}{=}
1-\varphi_j\!\!\sqrt[2N\!\!]{\lambda/c}\,\varepsilon+o(\varepsilon)
\end{equation}
which implies that
\[
\lim_{\varepsilon\to 0^+} u_j(\lambda,\varepsilon)^{\lceil\alpha/
\varepsilon\rceil}=\mathrm{e}^{-\varphi_j\!\!\sqrt[2N\!\!]{\lambda/c}\,\alpha}.
\]
Therefore,
\begin{align*}
\lim_{\varepsilon\to 0^+} \int_0^{+\infty} \mathrm{e}^{-\lambda t}
\,\mathbb{P}\{X_t^{\varepsilon}\in[\alpha,\beta)\}\,\mathrm{d} t
&=
\frac{1}{2N\lambda} \sum_{j=1}^N\left(\mathrm{e}^{-\varphi_j\!\!
\sqrt[2N\!\!]{\lambda/c}\,\alpha}-
\mathrm{e}^{-\varphi_j\!\!\sqrt[2N\!\!]{\lambda/c}\,\beta}\right)
\\
&=
\frac{1}{2N\lambda} \sum_{j=1}^N\left(\varphi_j\!\!\sqrt[2N\!\!]{\lambda/c}
\int_{\alpha}^{\beta} \mathrm{e}^{-\varphi_j\!\!\sqrt[2N\!\!]{\lambda/c}\,x}
\,\mathrm{d} x\right)\!.
\end{align*}
The case $\alpha<\beta\le 0$ is similar to treat. We have obtained the following result.
%
\begin{proposition}
The $\lambda$-potential of the pseudo-process $(X_t)_{t\ge 0}$ is given by
\begin{align*}
\int_0^{+\infty} \mathrm{e}^{-\lambda t} \,\big(\mathbb{P}\{X_t\in \mathrm{d} x\}
/\mathrm{d} x\big)\,\mathrm{d} t
&=
\begin{cases}
\displaystyle \frac{1}{2N\!\sqrt[2N\!]{c}\,\lambda^{1-1/(2N)}} \sum_{j=1}^N \varphi_j
\mathrm{e}^{-\varphi_j\!\!\sqrt[2N\!\!]{\lambda/c}\,x} & \text{if $x\ge0$},
\\
\displaystyle -\frac{1}{2N\!\sqrt[2N\!]{c}\,\lambda^{1-1/(2N)}} \sum_{j=1}^N \varphi_j
\mathrm{e}^{\varphi_j\!\!\sqrt[2N\!\!]{\lambda/c}\,x} & \text{if $x\le0$}.
\end{cases}
\end{align*}
\end{proposition}
%

\newpage
\begin{center}
\textbf{\Large Part II --- First overshooting time of a single threshold}
\end{center}
\vspace{\baselineskip}

\section{On the pseudo-distribution of $(\sigma_b^+,S_b^+)$}\label{section-sigma-b}

Let $b\in\mathbb{N}^*$. In this section, we explicitly compute the
generating function of $(\sigma_b^+,S_b^+)$. Set, for $\ell\in\{b,b+1,\dots,b+N-1\}$,
\[
H_{b,\ell}^+(z)=\mathbb{E}\!\left(z^{\sigma_b^+}\ind_{\{S_b^+=\ell,\sigma_b^+<+\infty\}}\right)
=\sum_{k\in\mathbb{N}} \mathbb{P}\{\sigma_b^+=k,S_b^+=\ell\}z^k.
\]
We are able to provide an explicit expression of $H_{b,\ell}^+(z)$.
Before tackling this problem, we need an \textit{a priori} estimate for
$\mathbb{P}\{\sigma_b^+=k,S_b^+=\ell\}$. By (\ref{bound-Fbis}),
we immediately derive that
$|\mathbb{P}\{\sigma_b^+=k,S_b^+=\ell\}|=
|\mathbb{P}\{S_1<b,\dots, S_{k-1}<b,S_k=\ell\}|\le M_1^k.$
Hence, the power series defining $H_{b,\ell}^+(z)$ absolutely converges for
$|z|<1/M_1$.

\subsection{Joint pseudo-distribution of $(\sigma_b^+,S_b^+)$}

%
\begin{theorem}\label{theorem-joint-dist}
The pseudo-distribution of $(\sigma_b^+,S_b^+)$ is characterized by the identity,
valid for any $z\in(0,1)$ and any $\ell\in\{b,b+1,\dots,N-1\}$,
\begin{equation}\label{joint-dist}
\mathbb{E}\!\left(z^{\sigma_b^+}\ind_{\{S_b^+=\ell,\sigma_b^+<+\infty\}}\right)
=(-1)^{\ell-b}\sum_{k=1}^N\frac{s_{k,\ell-b}^+(z)}{p_k^+(z)}\,u_k(z)^{b+N-1},
\end{equation}
where $s_{k,0}^+(z)=1$ and for $k\in\{1,\dots,N\}$, $\ell\in\{1,\dots,N-1\}$,
\[
s_{k,\ell}^+(z)=\sum_{1\le i_1<\dots<i_{\ell}\le N\atop i_1,\dots,i_{\ell}\neq k}
u_{i_1}(z)\cdots u_{i_{\ell}}(z),\quad
p_k^+(z)=\prod_{1\le j\le N\atop j\neq k} [u_k(z)-u_j(z)].
\]
\end{theorem}
%
\begin{proof}
Pick an integer $k\ge b$. If $S_n=k$, then an overshoot of the threshold $b$
occurs before time $n$: $\sigma_b^+\le n$. This remark and the independence of the
increments of the pseudo-random walk entail that
\begin{align}
\mathbb{P}\{S_n=k\}
&
=\mathbb{P}\{S_n=k,\sigma_b^+\le n\}=\sum_{j=0}^n\sum_{\ell=b}^{b+N-1}
\mathbb{P}\{S_n=k,\sigma_b^+=j,S_b^+=\ell\}
\nonumber\\
&
=\sum_{j=0}^n\sum_{\ell=b}^{b+N-1} \mathbb{P}\{\sigma_b^+=j,S_b^+=\ell\}
\mathbb{P}\{S_{n-j}=k-\ell\}.
\label{convol1}
\end{align}
Since the series defining $G_k(z)$ and $H_{b,\ell}^+(z)$ absolutely
converge respectively for $z\in(0,1)$ and $|z|<1/M_1$, and since
$M_1\ge 1$, we can apply the generating function to the convolution
equality~(\ref{convol1}). We get, for $z\in(0,1/M_1)$, that
\[
G_k(z)=\sum_{\ell=b}^{b+N-1} G_{k-\ell}(z) H_{b,\ell}^+(z).
\]
Using expression~(\ref{expression-Gk}) of $G_k$, namely
$G_k(z)=\sum_{j=1}^N \alpha_j(z)\,u_j(z)^{k}$ for $k\ge 0$, where
$\alpha_j(z)=\frac{1}{N(1-z)}\frac{1-u_j(z)}{1+u_j(z)}$, we obtain that
\begin{equation}\label{system-equations-inter}
\sum_{j=1}^N \alpha_j(z)\,u_j(z)^{k}
\left(\sum_{\ell=b}^{b+N-1} \frac{H_{b,\ell}^+(z)}{u_j(z)^{\ell}}-1\right)=0,
\quad k\ge b+N-1.
\end{equation}
Recalling that $v_j(z)=1/u_j(z)$ and setting
$\tilde{\alpha}_j(z)=\alpha_j(z)\left(\sum_{\ell=b}^{b+N-1}
H_{b,\ell}^+(z)v_j(z)^{\ell}-1\right)$, system~(\ref{system-equations-inter})
reads $\sum_{j=1}^N \tilde{\alpha}_j(z)\,u_j(z)^{k}=0$, $k\ge b+N-1$.
When limiting the range of $k$ to the set $\{b+N,b+N+1,\dots,b+2N-1\}$,
this becomes an homogeneous Vandermonde system whose solution is trivial:
$\tilde{\alpha}_j(z)=0$, $1\le j\le N$. Thus, we get the Vandermonde system below:
\begin{equation}\label{system-equations}
\sum_{\ell=b}^{b+N-1} H_{b,\ell}^+(z)v_j(z)^{\ell}=1, \quad 1\le j\le N.
\end{equation}
System~(\ref{system-equations}) can be explicitly solved.
In order to simplify the settings, we shall omit the variable $z$ in the
sequel of the proof. It is convenient to rewrite
(\ref{system-equations}) as
\begin{equation}\label{system-equationsbis}
\sum_{\ell=b}^{b+N-1} H_{b,\ell}^+v_j^{\ell-b}=u_j^b, \quad 1\le j\le N.
\end{equation}
Cramer's formulae yield
\begin{equation}\label{exp-Hbl}
H_{b,\ell}^+=\frac{V_{\ell}(v_1,\dots,v_N)}{V(v_1,\dots,v_N)},
\quad b\le \ell\le b+N-1
\end{equation}
where
\[
V(v_1,\dots,v_N)=\begin{vmatrix}
1      & v_1    & \dots & v_1^{N-1} \\
1      & v_2    & \dots & v_2^{N-1} \\
\vdots & \vdots &       & \vdots    \\
1      & v_N    & \dots & v_N^{N-1}
\end{vmatrix}
=\prod_{1\le i< j\le N} (v_j-v_i)
\]
and, for any $\ell\in\{b,\dots,b+N-1\}$,
\[
V_{\ell}(v_1,\dots,v_N)
=\begin{vmatrix}
1      & v_1    & \dots & v_1^{\ell-b-1} & u_1^b  & v_1^{\ell-b+1} & \dots & v_1^{N-1} \\
1      & v_2    & \dots & v_2^{\ell-b-1} & u_2^b  & v_2^{\ell-b+1} & \dots & v_2^{N-1} \\
\vdots & \vdots &       & \vdots         & \vdots & \vdots         &       & \vdots    \\
1      & v_N    & \dots & v_N^{\ell-b-1} & u_N^b  & v_N^{\ell-b+1} & \dots & v_N^{N-1}
\end{vmatrix}\!.
\]
This last determinant can be expanded as
$\sum_{k=1}^N u_k^b V_{k\ell}(v_1,\dots,v_{k-1},v_{k+1},\dots,v_N)$
with, for $k\in\{1,\dots, N\}$,
\[
V_{k\ell}(v_1,\dots,v_{k-1},v_{k+1},\dots,v_N)=
\begin{vmatrix}
1      & v_1     & \dots & v_1^{\ell-b-1}     & 0      & v_1^{\ell-b+1}     & \dots & v_1^{N-1}     \\
\vdots & \vdots  &       & \vdots             & \vdots & \vdots             &       & \vdots        \\
1      & v_{k-1} & \dots & v_{k-1}^{\ell-b-1} & 0      & v_{k-1}^{\ell-b+1} & \dots & v_{k-1}^{N-1} \\[0.5ex]
1      & v_k     &\dots  & v_k^{\ell-b-1}     & 1      & v_k^{\ell-b+1}     & \dots & v_k^{N-1}     \\[0.5ex]
1      & v_{k+1} & \dots & v_{k+1}^{\ell-b-1} & 0      & v_{k+1}^{\ell-b+1} & \dots & v_{k+1}^{N-1} \\
\vdots & \vdots  &       & \vdots             & \vdots & \vdots             &       & \vdots        \\
1      & v_N     &\dots  & v_N^{\ell-b-1}     & 0      & v_N^{\ell-b+1}     & \dots & v_N^{N-1}
\end{vmatrix}\!.
\]
In fact, the quantity
$V_{k\ell}(v_1,\dots,v_{k-1},v_{k+1},\dots,v_N)$ is the coefficient of
$x^{\ell-b}$ in the polynomial
\[
x\longmapsto \begin{vmatrix}
1      & v_1     & \dots & v_1^{N-1}     \\
\vdots & \vdots  &       & \vdots        \\
1      & v_{k-1} & \dots & v_{k-1}^{N-1} \\[0.5ex]
1      & x       & \dots & x^{N-1}       \\
1      & v_{k+1} & \dots & v_{k+1}^{N-1} \\
\vdots & \vdots  &       & \vdots        \\
1      & v_N     & \dots & v_N^{N-1}
\end{vmatrix}
\]
which is nothing but $V(v_1,\dots,v_{k-1},x,v_{k+1},\dots,v_N)$,
the value of which is
\begin{align*}
\lqn{
\prod_{1\le i< j\le N\atop i,j\neq k} (v_j-v_i)
\prod_{1\le i\le k-1} (x-v_i)\prod_{k+1\le i\le N} (v_i-x)}
&
=\frac{\prod_{1\le i< j\le N} (v_j-v_i)}{\prod_{1\le i\le N\atop i\neq k} (v_k-v_i)}
\prod_{1\le i\le N\atop i\neq k} (x-v_i)
=V(v_1,\dots,v_N) \prod_{1\le i\le N\atop i\neq k} \frac{x-v_i}{v_k-v_i}
\\
&
=(-1)^{N-1}V(v_1,\dots,v_N) \prod_{1\le i\le N\atop i\neq k}
\left(u_k\,\frac{u_ix-1}{u_k-u_i}\right)
\\
&
=(-1)^{N-1} \frac{u_k^{N-1}}{p_k^+}\,V(v_1,\dots,v_N)
\prod_{1\le i\le N\atop i\neq k}(u_ix-1).
\end{align*}
Using the elementary expansion $\prod_{1\le i\le N\atop i\neq k}(u_ix-1)
=\sum_{\ell=0}^{N-1} (-1)^{N-1-\ell} s_{k,\ell}^+ \,x^{\ell}$, we obtain
by identification that
\[
V_{k\ell}(v_1,\dots,v_{k-1},v_{k+1},\dots,v_N)
=(-1)^{\ell-b} \frac{u_k^{N-1}}{p_k^+} \,s_{k,\ell-b}^+V(v_1,\dots,v_N).
\]
Plugging this expression into~(\ref{exp-Hbl}), we then derive
for $H_{b,\ell}^+(z)$ representation~(\ref{joint-dist})
which is valid at least for $z\in(0,1/M_1)$. Finally, we observe
that (\ref{joint-dist}) defines an analytical function in $(0,1)$
and that $H_{b,\ell}^+(z)$ is a power series. Thus, by analytical continuation,
(\ref{joint-dist}) holds true for any $z\in(0,1)$.
\end{proof}
%
\begin{example}\label{example-joint-dist}
For $N=2$, the settings of Theorem~\ref{theorem-joint-dist} write
\[
s_{1,0}^+(z)=s_{2,0}^+(z)=1,\;s_{1,1}^+(z)=u_2(z),\;s_{2,1}^+(z)=u_1(z),
\]
\[
p_1^+(z)=u_1(z)-u_2(z),\;p_2^+(z)=u_2(z)-u_1(z),
\]
where $u_1(z)$ and $u_2(z)$ are given in Example~\ref{example-roots}
and (\ref{joint-dist}) reads
\begin{align*}
\mathbb{E}\!\left(z^{\sigma_b^+}\ind_{\{S_b^+=b,\sigma_b^+<+\infty\}}\right)
&
=\frac{u_1(z)^{b+1}-u_2(z)^{b+1}}{u_1(z)-u_2(z)},
\\
\mathbb{E}\!\left(z^{\sigma_b^+}\ind_{\{S_b^+=b+1,\sigma_b^+<+\infty\}}\right)
&
=\frac{u_1(z)u_2(z)^{b+1}-u_2(z)u_1(z)^{b+1}}{u_1(z)-u_2(z)}.
\end{align*}
\end{example}
%
\begin{remark}
We have the similar expression related to $\sigma_a^-$ below. The analogous
system to (\ref{system-equations}) writes as
\[
\sum_{\ell=a-N+1}^a H_{a,\ell}^-(z)u_j(z)^{\ell}=1, \quad 1\le j\le N
\]
where $H_{a,\ell}^-(z)=\mathbb{E}\!\left(z^{\sigma_a^-}
\ind_{\{S_a^-=\ell,\sigma_a^-<+\infty\}}\right)$.
The solution is given by
\[
\mathbb{E}\!\left(z^{\sigma_a^-}\ind_{\{S_a^-=\ell,\sigma_a^-<+\infty\}}\right)
=(-1)^{\ell-a}\sum_{k=1}^N\frac{s_{k,\ell-a}^-(z)}{p_k^-(z)}\,u_k(z)^{a+N-1}
\]
where $s_{k,0}^-(z)=1$ and, for $k\in\{1,\dots,N\}$, $\ell\in\{1,\dots,N-1\}$,
\[
s_{k,\ell}^-(z)=\sum_{1\le i_1<\dots<i_{\ell}\le N\atop i_1,\dots,i_{\ell}\neq k}
v_{i_1}(z)\cdots v_{i_{\ell}}(z),\quad
p^-_k(z)=\prod_{1\le j\le N\atop j\neq k} [v_k(z)-v_j(z)].
\]
\end{remark}

The double generating function of $(\sigma_b^+,S_b^+)$ defined by
\[
\mathbb{E}\!\left(z^{\sigma_b^+}\zeta^{S_b^+}\ind_{\{\sigma_b^+<+\infty\}}\right)
=\sum_{\ell=b}^{b+N-1}\mathbb{E}\!\left(z^{\sigma_b^+}\ind_{\{S_b^+=\ell\}}\right)\zeta^{\ell}
\]
admits an interesting representation by means of Lagrange interpolation
polynomials that we display in the theorem below.
%
\begin{theorem}
The double generating function of $(\sigma_b^+,S_b^+)$ is given,
for any $z\in(0,1/M_1)$ and $\zeta\in\mathbb{C}$,
by
\begin{equation}\label{double-gene-sigma}
\mathbb{E}\!\left(z^{\sigma_b^+}\zeta^{S_b^+}\ind_{\{\sigma_b^+<+\infty\}}\right)
=\sum_{k=1}^N \mathbf{L}_k(z,\zeta) (u_k(z)\zeta)^b
\end{equation}
where
\[
\mathbf{L}_k(z,\zeta)=\prod_{1\le j\le N\atop j\neq k} \frac{\zeta -v_j(z)}{v_k(z)-v_j(z)},
\quad k\in\{1,\dots,N\},
\]
are the Lagrange interpolation polynomials with respect to the variable $\zeta$
such that $\mathbf{L}_k(z,v_j(z))=\delta_{jk}$.
\end{theorem}
%
\begin{proof}
By~(\ref{exp-Hbl}) and by omitting the variable $z$ as previously, we have that
\begin{align*}
\mathbb{E}\!\left(z^{\sigma_b^+}\zeta^{S_b^+}\ind_{\{\sigma_b^+<+\infty\}}\right)
&
=\sum_{\ell=b}^{b+N-1} H_{b,\ell}^+ \,\zeta^{\ell}
=\sum_{\ell=b}^{b+N-1} \frac{V_{\ell}(v_1,\dots,v_N)}{V(v_1,\dots,v_N)}\,\zeta^{\ell}
\\
&
=\sum_{\ell=b}^{b+N-1} \left(\sum_{k=1}^N u_k^b\frac{V_{k\ell}(v_1,\dots,v_N)}
{V(v_1,\dots,v_N)}\right) \zeta^{\ell}
\\
&
=\sum_{k=1}^N \left(\,\sum_{\ell=b}^{b+N-1}
\frac{V_{k\ell}(v_1,\dots,v_N)}{V(v_1,\dots,v_N)} \,\zeta^{\ell-b}\right)(u_k\zeta)^b
\\
&
=\sum_{k=1}^N \frac{V(v_1,\dots,v_{k-1},\zeta,v_{k+1},\dots,v_N)}
{V(v_1,\dots,v_N)} \,(u_k\zeta)^b.
\end{align*}
It is clear that the quantity $V(v_1,\dots,v_{k-1},\zeta,v_{k+1},\dots,v_N)/V(v_1,\dots,v_N)$,
which explicitly writes as
\[
\begin{vmatrix}
1      & v_1     & \dots & v_1^{N-1}     \\[-0.5ex]
\vdots & \vdots  &       & \vdots        \\
1      & v_{k-1} & \dots & v_{k-1}^{N-1} \\
1      & \zeta   & \dots & \zeta^{N-1}   \\
1      & v_{k+1} & \dots & v_{k+1}^{N-1} \\[-0.5ex]
\vdots & \vdots  &       & \vdots        \\
1      & v_N     & \dots & v_N^{N-1}
\end{vmatrix}
/
\begin{vmatrix}
1      & v_1    & \dots & v_1^{N-1} \\
\vdots & \vdots &       & \vdots    \\
1      & v_N    & \dots & v_N^{N-1}
\end{vmatrix}\!,
\]
defines a polynomial of the variable $\zeta$ of degree $N-1$ which vanishes at
$v_1,\dots,v_{k-1},$ $v_{k+1},\dots,v_N$
and equals $1$ at $v_k$. Hence, by putting back the variable $z$,
it coincides with the Lagrange polynomial $\mathbf{L}_k(z,\zeta)$ and
formula~(\ref{double-gene-sigma}) immediately ensues.
\end{proof}
%
\begin{example}
For $N=2$, (\ref{double-gene-sigma}) reads
\begin{align*}
\mathbb{E}\!\left(z^{\sigma_b^+}\zeta^{S_b^+}\ind_{\{\sigma_b^+<+\infty\}}\right)\!
&=
\zeta^b\left(u_1(z)^b\frac{\zeta-v_2(z)}{v_1(z)-v_2(z)}
+u_2(z)^b\frac{\zeta-v_1(z)}{v_2(z)-v_1(z)}\right)
\\
&=
\frac{\zeta^b}{u_1(z)-u_2(z)}\left[\left(u_1(z)^{b+1}-u_2(z)^{b+1}\right)\right.
\\
&\hphantom{=\;}
\left.+\left(u_1(z)u_2(z)^{b+1}-u_2(z)u_1(z)^{b+1}\right)\zeta\right]\!.
\end{align*}
This is in good agreement with the formulae of Example~\ref{example-joint-dist}.
We retrieve a result of~\cite{sato}.
\end{example}
%

\subsection{Pseudo-distribution of $S_b^+$}

In order to derive the pseudo-distribution of $S_b^+$ which is characterized
by the numbers $H_{b,\ell}^+(1), \ell\in\{b,b+1,\dots,b+N-1\}$, we solve the
system obtained by taking the limit in~(\ref{system-equations}) as $z\to 1^-$.
%
\begin{lemma}\label{lemma-syst-H}
The following system holds:
\begin{equation}\label{system-H(1)0}
\sum_{\ell=k+b}^{b+N-1} \binom{\ell-b}{k}  H_{b,\ell}^+(1)=(-1)^k \binom{b+k-1}{b-1},
\quad 0\le k\le N-1.
\end{equation}
\end{lemma}
%
\begin{proof}
By~(\ref{asymptotic-u}), we have the expansion $v_j(z)=1/u_j(z)=1-\varepsilon_j(z)$
where $\varepsilon_j(z)\underset{z\to 1^-}{=}\mathcal{O}\!\left(\!\!\sqrt[2N\!]{1-z}\,\right)$
for any $j\in\{1,\dots,N\}$.
Putting this into (\ref{system-equations}), we get that
\[
\sum_{\ell=b}^{b+N-1} (1-\varepsilon_j(z))^{\ell-b} H_{b,\ell}^+(z)=(1-\varepsilon_j(z))^{-b},
\]
that is,
\begin{align}\label{eqH}
\sum_{k=0}^{N-1} (-1)^k\left(\,\sum_{\ell=b+k}^{b+N-1} \binom{\ell-b}{k}
H_{b,\ell}^+(z) \right) \varepsilon_j(z)^k
&
=\sum_{k=0}^{\infty} (-1)^k\binom{-b}{k} \varepsilon_j(z)^k
\nonumber\\
&
= \sum_{k=0}^{\infty} \binom{b+k-1}{b-1} \varepsilon_j(z)^k.
\end{align}
Set
\[
M_k(z)=(-1)^k\sum_{\ell=b+k}^{b+N-1} \binom{\ell-b}{k} H_{b,\ell}^+(z)
-\binom{b+k-1}{b-1},
\quad R_j(z) =\sum_{\ell=N}^{\infty} \binom{b+\ell-1}{b-1} \varepsilon_j(z)^{\ell}.
\]
Then, equality~(\ref{eqH}) reads
\begin{align*}
\sum_{k=0}^{N-1} M_k(z)\varepsilon_j(z)^k = R_j(z),\quad 1\le j\le N.
\end{align*}
This is a Vandermonde system the solution of which is given by
\[
M_k(z)=\frac{\tilde{V}_k\!\left(\!\!\!\begin{array}{c}
\varepsilon_1(z),\dots,\varepsilon_N(z)\\[-.2ex]
\displaystyle R_1(z),\dots,R_N(z)\end{array}\!\!\!\right)\!}{\tilde{V}
(\varepsilon_1(z),\dots,\varepsilon_N(z))},
\quad 0\le k\le N-1
\]
where
\[
\tilde{V}(\varepsilon_1(z),\dots,\varepsilon_N(z))=\begin{vmatrix}
\,1      & \varepsilon_1(z) & \dots & \varepsilon_1(z)^{N-1} \\
\,1      & \varepsilon_2(z) & \dots & \varepsilon_2(z)^{N-1} \\
\,\vdots & \vdots  &       & \vdots        \\
\,1      & \varepsilon_N(z) & \dots & \varepsilon_N(z)^{N-1}
\end{vmatrix}\!,
\]
\[
\tilde{V}_k\!\left(\!\!\!\begin{array}{c} \varepsilon_1(z),\dots,\varepsilon_N(z)\\[-.2ex]
\displaystyle R_1(z),\dots,R_N(z)\end{array}\!\!\!\right)=\begin{vmatrix}
\,1      & \varepsilon_1(z) & \dots & \varepsilon_1(z)^{k-1} & R_1(z) & \varepsilon_1(z)^{k+1} & \dots & \varepsilon_1(z)^{N-1} \\
\,1      & \varepsilon_2(z) & \dots & \varepsilon_2(z)^{k-1} & R_2(z) & \varepsilon_2(z)^{k+1} & \dots & \varepsilon_2(z)^{N-1} \\
\,\vdots & \vdots           &       & \vdots                 & \vdots & \vdots                 &       & \vdots        \\
\,1      & \varepsilon_N(z) & \dots & \varepsilon_N(z)^{k-1} & R_N(z) & \varepsilon_N(z)^{k+1} & \dots & \varepsilon_N(z)^{N-1}
\end{vmatrix}\!.
\]
Since, by~(\ref{asymptotic-u}), $\varepsilon_j(z)\underset{z\to 1^-}{\sim}
constant\times \!\sqrt[2N\!]{1-z}$ for any $j\in\{1,\dots,N\}$,
we have that
\[
\tilde{V}(\varepsilon_1(z),\dots,\varepsilon_N(z))
=\prod_{1\le \ell<m\le N}[\varepsilon_m(z)-\varepsilon_{\ell}(z)]
\underset{z\to 1^-}{\sim} constant\times(1-z)^{(N-1)/4}
\]
and second,
\[
R_j(z)\underset{z\to 1^-}{\sim} \binom{b+N-1}{b-1} \varepsilon_j(z)^N
\sim constant\times \sqrt{1-z}
\]
which implies, for $k\in\{0,\dots,N-1\}$, that
\begin{align*}
\tilde{V}_k\!\left(\!\!\!\begin{array}{c} \varepsilon_1(z),\dots,\varepsilon_N(z)\\[-.2ex]
\displaystyle R_1(z),\dots,R_N(z)\end{array}\!\!\!\right)
&
\underset{z\to 1^-}{=}
\mathcal{O}\!\left[(1-z)^{1/2+\left(\sum_{1\le m\le N-1,m\neq k} m\right)/(2N)}\right]
=o\!\left[(1-z)^{(N-1)/4}\right]\!.
\end{align*}
Therefore $\lim_{z\to 1^-} M_k(z)=0$.
On the other hand, for $z\in(0,1)$, referring to the definition of $M_k(z)$,
we can see that the quantity $H_{b,\ell}^+(z)$ can be
expressed as a linear combination of the $M_k(z)$'s plus a constant.
Hence, the limit $\lim_{z\to 1^-} H_{b,\ell}^+(z)$ exists and,
by appealing to a Tauberian theorem, it coincides with $H_{b,\ell}^+(1)$.
This finishes the proof of~(\ref{system-H(1)0}).
\end{proof}
%
\begin{theorem}\label{theo-dist-Sb}
The pseudo-distribution of $S_b^+$ is characterized by the following pseudo-probabilities:
for any $\ell\in\{b,b+1,\dots,b+N-1\}$,
\begin{equation}\label{pseudo-dist-Sb}
\mathbb{P}\{S_b^+=\ell,\sigma_b^+<+\infty\}=(-1)^{b+\ell} \,\frac{b}{\ell}\binom{N-1}{\ell-b}
\!\binom{b+N-1}{b}.
\end{equation}
Moreover, $\mathbb{P}\{\sigma_b^+<+\infty\}=1$.
\end{theorem}
%
\begin{proof}
We explicitly solve system~(\ref{system-H(1)0}) rewritten as
\[
\sum_{\ell=k}^{N-1} \binom{\ell}{k}  H_{b,\ell+b}^+(1)=(-1)^k \binom{b+k-1}{b-1},
\quad 0\le k\le N-1.
\]
The matrix of the system is $\left(\binom{\ell}{k}\right)_{0\le k,\ell\le N-1}$
which admits $\left((-1)^{k+\ell}\binom{\ell}{k}\right)_{0\le k,\ell\le N-1}$
as an inverse with the convention of settings $\binom{\ell}{k}=0$ if $k>\ell$.
The solution of the system is given, for $\ell\in\{0,1,\dots,N-1\}$, by
\begin{align*}
H_{b,\ell+b}^+(1)
&
=\sum_{k=\ell}^{N-1} (-1)^{k+\ell} \binom{k}{\ell} \times(-1)^{k}\binom{b+k-1}{b-1}
=(-1)^{\ell} \binom{b+\ell-1}{b-1} \sum_{k=\ell}^{N-1} \binom{b+k-1}{b+\ell-1}
\\
&
=(-1)^{\ell} \binom{b+\ell-1}{b-1}\!\binom{b+N-1}{b+\ell}
=(-1)^{\ell} \,\frac{b}{\ell+b}\binom{N-1}{\ell}\!\binom{b+N-1}{b}.
\end{align*}
This proves~(\ref{pseudo-dist-Sb}). Now, by summing the
$\mathbb{P}\{S_b^+=\ell,\sigma_b^+<+\infty\}$,
$b\le \ell\le b+N-1$, given by~(\ref{pseudo-dist-Sb}), we obtain that
\begin{align*}
\mathbb{P}\{\sigma_b^+<+\infty\}
&
=(-1)^{b}b\binom{b+N-1}{b} \sum_{\ell=b}^{b+N-1}\frac{(-1)^{\ell}}{\ell}
\binom{N-1}{\ell-b}
\\
&
=b\binom{b+N-1}{b} \sum_{\ell=0}^{N-1} \frac{(-1)^{\ell}}{\ell+b}
\binom{N-1}{\ell}.
\end{align*}
Writing $\frac{1}{\ell+b}=\int_0^1 x^{\ell+b-1}\,\mathrm{d} x$, we see that
\begin{align*}
\sum_{\ell=0}^{N-1} \frac{(-1)^{\ell}}{\ell+b} \binom{N-1}{\ell}
&
=\int_0^1 \left(\sum_{\ell=0}^{N-1} (-1)^{\ell} \binom{N-1}{\ell}
x^{\ell}\right)x^{b-1}\,\mathrm{d} x
\\
&
=\int_0^1 (1-x)^{N-1}x^{b-1}\,\mathrm{d} x=\frac{(N-1)!(b-1)!}{(b+N-1)!}.
\end{align*}
Hence $\mathbb{P}\{\sigma_b^+<+\infty\}=1$. The proof of Theorem~\ref{theo-dist-Sb} is finished.
\end{proof}
%
In the sequel, when considering $S_b^+$, we shall omit the condition $\sigma_b^+<+\infty$.
%
\begin{example}
Let us have a look on the particular values $1,2,3,4$ of $N$.
\begin{itemize}
\item
\textsl{Case $N=1$.} Evidently, in this case $S_b^+=b$ and then
\[
\mathbb{P}\{S_b^+=b\}=1.
\]
This is the case of the ordinary random walk!

\item
\textsl{Case $N=2$.} In this case $S_b^+\in\{b,b+1\}$ and
\[
\mathbb{P}\{S_b^+=b\}=b+1,\quad \mathbb{P}\{S_b^+=b+1\}=-b.
\]

\item
\textsl{Case $N=3$.} In this case $S_b^+\in\{b,b+1,b+2\}$ and
\[
\mathbb{P}\{S_b^+=b\}=\frac12(b+1)(b+2),\quad \mathbb{P}\{S_b^+=b+1\}=-b(b+2),
\quad \mathbb{P}\{S_b^+=b+2\}=\frac12 \,b(b+1).
\]

\item
\textsl{Case $N=4$.} In this case $S_b^+\in\{b,b+1,b+2,b+3\}$ and
\begin{align*}
&\mathbb{P}\{S_b^+=b\}=\frac16(b+1)(b+2)(b+3),&&\mathbb{P}\{S_b^+=b+1\}=-\frac12 \, b(b+2)(b+3),
\\
&\mathbb{P}\{S_b^+=b+2\}=\frac12 \,b(b+1)(b+3),&&\mathbb{P}\{S_b^+=b+3\}=-\frac16 \, b(b+1)(b+2).
\end{align*}
\end{itemize}
\end{example}

\subsection{Pseudo-moments of $S_b^+$}\label{subsection-pseudo-moments-Sb}

In the sequel, we use the notation $(i)_n=i(i-1)(i-2)\cdots (i-n+1)$ for any $i\in\mathbb{Z}$
and any $n\in\mathbb{N}^*$ and $(i)_0=1$. Of course, $(i)_n=i!/(i-n)!$ and
$(i)_n/n!=\binom{i}{n}$ if $i\ge n$.
We also use the conventions $1/i!=0$ for any negative integer $i$ and
$\sum_{k=i}^j=0$ if $i>j$.

In this section, we compute several functionals related to the pseudo-moments
of $S_b^+$. More precisely, we provide formulae for
$\mathbb{E}\!\left[\left(S_b^+-\beta\right)_n\right]$ (Theorem~\ref{th-moment-beta}),
$\mathbb{E}\!\left[\left(S_b^+-b\right)_n\right]$ (Corollary~\ref{cor1}),
$\mathbb{E}\!\left[\left(S_b^+\right)_n\right]$ and
$\mathbb{E}\big[\big(S_b^+\big)^n\big]$ (Theorem~\ref{theorem-moments}).

Putting the elementary identity $1/(\ell+b)=\int_0^1 x^{\ell+b-1}\,\mathrm{d} x$
into the equality
\begin{align*}
\mathbb{E}\!\left[f\left(S_b^+\right)\right]
&
=\binom{b+N-1}{b} \sum_{\ell=b}^{b+N-1} (-1)^{b+\ell} \,\frac{b}{\ell}
\binom{N-1}{\ell-b}f(\ell)
\\
&
=b\binom{b+N-1}{b} \sum_{\ell=0}^{N-1} (-1)^{\ell}\binom{N-1}{\ell}
\frac{f(\ell+b)}{\ell+b},
\end{align*}
we get the following integral representation of $\mathbb{E}\!\left[f\left(S_b^+\right)\right]$.
%
\begin{theorem}
For any function $f$ defined on $\{b,\dots,b+N-1\}$,
\begin{equation}\label{integral-rep}
\mathbb{E}\!\left[f\left(S_b^+\right)\right]=b\binom{b+N-1}{b}
\int_0^1\left(\sum_{\ell=0}^{N-1} (-1)^{\ell}\binom{N-1}{\ell}
f(\ell+b)x^{\ell}\right) x^{b-1}\,\mathrm{d} x.
\end{equation}
\end{theorem}
%
\begin{theorem}\label{th-moment-beta}
For any integers $n\ge 0$ and $\beta$, the factorial
pseudo-moment of $(S_b^+-\beta)$ of order $n$ is given by
\begin{align}
\lqn{\mathbb{E}\!\left[\left(S_b^+-\beta\right)_n\right]}
=\begin{cases}
\displaystyle\frac{(b-\beta)!}{(b-1)!} \sum_{k=0\vee(n+\beta-b)}^{n\wedge(N-1)} (-1)^k
\frac{(k+b-1)!}{(k+b-\beta-n)!} \binom{n}{k} & \text{if $\beta\le b$,}
\\
\displaystyle\frac{(-1)^n}{(b-1)!(\beta-b-1)!} \sum_{k=0}^{n\wedge(N-1)}
(k+b-1)!(\beta-b+n-k-1)! \binom{n}{k} & \text{if $\beta\ge b+1$.}
\end{cases}
\label{moment-beta1}
\end{align}
If $n\le N-1$, we simply have that
\begin{equation}\label{moment-beta2}
\mathbb{E}\!\left[\left(S_b^+-\beta\right)_n\right]
=(-\beta)_n=(-1)^n\beta(\beta+1)\cdots(\beta+n-1).
\end{equation}
\end{theorem}
%
\begin{proof}
By~(\ref{integral-rep}), we have that
\begin{equation}\label{expect}
\mathbb{E}\!\left[\left(S_b^+-\beta\right)_n\right]
=b\binom{b+N-1}{b}\int_0^1 \left(\sum_{\ell=0}^{N-1}(-1)^{\ell}
\binom{N-1}{\ell} (\ell+b-\beta)_n\,x^{\ell+b-1}\right)\mathrm{d} x.
\end{equation}
Next, by observing that
$(\ell+b-\beta)_n\,x^{\ell+b-1}=x^{n+\beta-1}\frac{\mathrm{d}^n}{\mathrm{d} x^n}
\big(x^{\ell+b-\beta}\big)$, we obtain that
\begin{align}
\sum_{\ell=0}^{N-1} (-1)^{\ell} \binom{N-1}{\ell}(\ell+b-\beta)_n \,x^{\ell+b-1}
&
=\sum_{\ell=0}^{N-1} \left((-1)^{\ell} \binom{N-1}{\ell}
\frac{\mathrm{d}^n}{\mathrm{d} x^n}\big(x^{\ell+b-\beta}\big)\right)x^{n+\beta-1}
\nonumber\\
&
=\frac{\mathrm{d}^n}{\mathrm{d} x^n}\left(\sum_{\ell=0}^{N-1} (-1)^{\ell} \binom{N-1}{\ell}
x^{\ell+b-\beta}\right)x^{n+\beta-1}
\nonumber\\
&
=\frac{\mathrm{d}^n}{\mathrm{d} x^n}\big((1-x)^{N-1}x^{b-\beta}\big)x^{n+\beta-1}.
\label{integ1}
\end{align}
Applying Leibniz rule to (\ref{integ1}), we see that
\begin{align*}
\lqn{\sum_{\ell=0}^{N-1} (-1)^{\ell} \binom{N-1}{\ell}(\ell+b-\beta)_n \,x^{\ell+b-1}}
&
=\left(\sum_{k=0}^n \binom{n}{k} \frac{\mathrm{d}^k}{\mathrm{d} x^k}\left((1-x)^{N-1}\right)
\frac{\mathrm{d}^{n-k}}{\mathrm{d} x^{n-k}}\big(x^{b-\beta}\big) \right)x^{n+\beta-1}
\\
&
=\sum_{k=0}^{n\wedge (N-1)} (-1)^k\binom{n}{k}
(N-1)_k(b-\beta)_{n-k}(1-x)^{N-1-k}x^{k+b-1}.
\end{align*}
Therefore,
\begin{align}
\lqn{\int_0^1\left(\sum_{\ell=0}^{N-1} (-1)^{\ell} \binom{N-1}{\ell}(\ell+b-\beta)_n
\,x^{\ell+b-1}\right) \mathrm{d} x}
&
=\sum_{k=0}^{n\wedge (N-1)} (-1)^k (N-1)_k(b-\beta)_{n-k} \binom{n}{k}
\int_0^1 (1-x)^{N-1-k}x^{k+b-1}\,\mathrm{d} x
\nonumber\\
&
=\sum_{k=0}^{n\wedge (N-1)} (-1)^k (N-1)_k(b-\beta)_{n-k} \binom{n}{k}
\frac{(N-1-k)!(k+b-1)!}{(b+N-1)!}
\nonumber\\
&
=\begin{cases}
\displaystyle\frac{(b-\beta)!(N-1)!}{(b+N-1)!}\sum_{k=0\vee(n+\beta-b)}^{n\wedge (N-1)} (-1)^k
\frac{(k+b-1)!}{(k+b-\beta-n)!}\binom{n}{k} & \text{if $\beta\le b$,}
\\
\displaystyle\frac{(-1)^n(N-1)!}{(\beta-b-1)!(b+N-1)!}\sum_{k=0}^{n\wedge (N-1)}
(k+b-1)!(\beta-b+n-k-1)!\binom{n}{k} & \text{if $\beta\ge b+1$.}
\end{cases}
\label{integ2}
\end{align}
Finally, plugging (\ref{integ2}) into (\ref{integ1}) and (\ref{expect}) yields
(\ref{moment-beta1}).

Assume now that $n\le N-1$ and $\beta\le b$. If $n\ge 1-\beta$, we can write
in (\ref{moment-beta1}) that
\[
\frac{(k+b-1)!}{(k+b-\beta-n)!}=\left.\frac{\mathrm{d}^{\beta+n-1}}{\mathrm{d}
x^{\beta+n-1}}\big(x^{k+b-1}\big)\right|_{x=1}.
\]
Then,
\begin{align}
\lqn{\sum_{k=0\vee(n+\beta-b)}^{n\wedge (N-1)} (-1)^k\frac{(k+b-1)!}{(k+b-\beta-n)!}\binom{n}{k}}
&
=\sum_{k=0}^{n} (-1)^k\frac{(k+b-1)!}{(k+b-\beta-n)!}\binom{n}{k}
\nonumber\\
&
=\left.\frac{\mathrm{d}^{\beta+n-1}}{\mathrm{d} x^{\beta+n-1}}\left[\left(\sum_{k=0}^{n}
(-1)^k\binom{n}{k}x^k \right)x^{b-1}\right]\right|_{x=1}
=\left.\frac{\mathrm{d}^{\beta+n-1}}{\mathrm{d} x^{\beta+n-1}}\big((1-x)^n x^{b-1}\big)\right|_{x=1}
\nonumber\\
&
=\left.\sum_{k=0}^{n} (-1)^k (n)_{k} (b-1)_{\beta+n-k-1}\binom{\beta+n-1}{k}
\big((1-x)^{n-k} x^{k+b-\beta-n}\big)\right|_{x=1}
\nonumber\\
&
=(-1)^n (n)_{n} (b-1)_{\beta-1}\binom{\beta+n-1}{n}
=(-1)^n \frac{(b-1)!(\beta+n-1)!}{(b-\beta)!(\beta-1)!}
\nonumber\\
&
=(-1)^n \frac{(b-1)!}{(b-\beta)!}\,\beta(\beta+1)\cdots(\beta+n-1).
\label{sum3}
\end{align}
Putting (\ref{sum3}) into (\ref{moment-beta1}) yields (\ref{moment-beta2}).
If $n\le -\beta$ (which requires that $\beta\le 0$), in (\ref{moment-beta1}),
we write instead that
\[
\frac{(k+b-1)!}{(k+b-\beta-n)!}=\frac{1}{(-n-\beta)!}\int_0^1 x^{k+b-1}
(1-x)^{-n-\beta}\,\mathrm{d} x.
\]
Then,
\begin{align}
\lqn{\sum_{k=0\vee(n+\beta-b)}^{n\wedge (N-1)} (-1)^k\frac{(k+b-1)!}{(k+b-\beta-n)!}\binom{n}{k}}
&
=\sum_{k=0}^{n} (-1)^k\frac{(k+b-1)!}{(k+b-\beta-n)!}\binom{n}{k}
\nonumber\\
&
=\frac{1}{(-n-\beta)!}\int_0^1 \left(\sum_{k=0}^{n} (-1)^k\binom{n}{k} x^{k}\right)
x^{b-1}(1-x)^{-n-\beta}\,\mathrm{d} x
\nonumber\\
&
=\frac{1}{(-n-\beta)!}\int_0^1 x^{b-1}(1-x)^{-\beta}\,\mathrm{d} x
\nonumber\\
&
=\frac{(b-1)!(-\beta)!}{(b-\beta)!(-n-\beta)!}
=(-1)^n\frac{(b-1)!}{(b-\beta)!}\,\beta(\beta+1)\cdots(\beta+n-1).
\label{sum3bis}
\end{align}
Putting (\ref{sum3bis}) into (\ref{moment-beta1}) yields (\ref{moment-beta2})
in this case too.

Assume finally that $n\le N-1$ and $\beta\ge b+1$. We write in (\ref{moment-beta1})
that
\[
(k+b-1)!(\beta-b+n-k-1)!= (\beta+n-1)! \int_0^1 x^{k+b-1} (1-x)^{\beta-b+n-k-1}\,\mathrm{d} x.
\]
Then
\begin{align}
\lqn{\sum_{k=0}^{n\wedge(N-1)} (k+b-1)!(\beta-b+n-k-1)! \binom{n}{k}}
&
= (\beta+n-1)! \int_0^1 \left[\,\sum_{k=0}^n \binom{n}{k}\!
\left(\frac{x}{1-x}\right)^{\!k}\right] x^{b-1} (1-x)^{\beta-b+n-1}\,\mathrm{d} x
\nonumber\\
&
= (\beta+n-1)! \int_0^1 x^{b-1} (1-x)^{\beta-b-1}\,\mathrm{d} x
= (b-1)! (\beta-b-1)! \,\frac{(\beta+n-1)!}{(\beta-1)!}.
\label{sum3ter}
\end{align}
Putting (\ref{sum3ter}) into (\ref{moment-beta1}) yields (\ref{moment-beta2}).
\end{proof}
%
By choosing $\beta=b$ in Theorem~\ref{th-moment-beta}, we derive that
\[
\mathbb{E}\!\left[\left(S_b^+-b\right)_n\right]
=\frac{1}{(b-1)!} \sum_{k=n}^{n\wedge(N-1)} (-1)^k \frac{(k+b-1)!}{(k-n)!} \binom{n}{k}.
\]
We immediately obtain the following particular result which will
be used in Theorem~\ref{theo-exp-f}.
%
\begin{corollary}\label{cor1}
The factorial pseudo-moments of $(S_b^+-b)$ are given by
\[
\mathbb{E}\!\left[\left(S_b^+-b\right)_n\right]
=\begin{cases}
(-b)_n & \text{if $0\le n\le N-1$,}
\\
0  & \text{if $n\ge N$.}
\end{cases}
\]
\end{corollary}
%
The above identity can be rewritten, if $0\le n\le N-1$, as
\begin{align}\label{moment_comb}
\mathbb{E}\!\left[\binom{S_b^+-b}{n}\right]=(-1)^n\binom{n+b-1}{b-1}.
\end{align}
Moreover, since $S_b^+\in\{b,b+1,\dots,b+N-1\}$, it is clear that
$\big(S_b^+-b\big)\big(S_b^+-b-1\big)\cdots\big(S_b^+-b-N+1\big)=0$
which immediately entails that $\left(S_b^+-b\right)_n=0$ for any $n\ge N$;
then $\mathbb{E}\!\left[\left(S_b^+-b\right)_n\right]=0$ for $n\ge N$
as stated in Corollary~\ref{cor1}.

By choosing $\beta=0$ in Theorem~\ref{th-moment-beta}, we plainly extract that
$\mathbb{E}\!\left[\left(S_b^+\right)_n\right]=0$ for $n\in\{1,\dots,N-1\}$.
Moreover, as previously, $\left(S_b^+\right)_n=0$ for any $n\ge b+N$;
then $\mathbb{E}\!\left[\left(S_b^+\right)_n\right]=0$ for $n\ge b+N$.
Actually, we can compute the factorial pseudo-moments of $S_b^+$,
$\mathbb{E}\!\left[\left(S_b^+\right)_n\right]$, for $n\in\{N,N+1,\dots,b+N\}$.
The formula of Theorem~\ref{th-moment-beta} seems to be untractable,
so we provide another way for evaluating them.
%
\begin{theorem}\label{theorem-moments}
The factorial pseudo-moments of $S_b^+$ are given by
\[
\mathbb{E}\!\left[\left(S_b^+\right)_n\right]
=\begin{cases}
\displaystyle (-1)^{N-1}\binom{n-1}{N-1} (b+N-1)_n & \text{if $N\le n\le b+N-1$,}
\\[1ex]
0 & \text{if $1\le n\le N-1$ or $n\ge b+N$.}
\end{cases}
\]
Moreover, for $n\in\{1,\dots,N-1\}$, the pseudo-moment of $S_b^+$ of order $n$ vanishes:
\[
\mathbb{E}\big[\big(S_b^+\big)^n\big]=0
\]
and
\[
\mathbb{E}\Big[\big(S_b^+\big)^N\Big]=(-1)^{N-1}\frac{(b+N-1)!}{(b-1)!}=-(-b)_N.
\]
\end{theorem}
%
\begin{proof}
We focus on the case where $N\le n\le b+N-1$. We have that
\begin{align*}
\mathbb{E}\!\left[\left(S_b^+\right)_n\right]
&
=\sum_{\ell=0}^{N-1} \mathbb{P}\{S_b^+=\ell+b\} (\ell+b)_n
\\
&
=b\binom{b+N-1}{b} \sum_{\ell=0}^{N-1} (-1)^{\ell}(\ell+b-1)_{n-1}
\binom{N-1}{\ell}.
\end{align*}
The intermediate sum lying in the last displayed equality can be
evaluated as follows: by observing that
$(\ell+b-1)_{n-1}=(\mathrm{d}^{n-1}/\mathrm{d} x^{n-1})(x^{\ell+b-1})|_{x=1}$ 
and appealing to Leibniz rule, we obtain that
\begin{align*}
\lqn{\sum_{\ell=0}^{N-1} (-1)^{\ell}(\ell+b-1)_{n-1}\binom{N-1}{\ell}}
&
=\left.\frac{\mathrm{d}^{n-1}}{\mathrm{d} x^{n-1}}
\left(\sum_{\ell=0}^{N-1} (-1)^{\ell} \binom{N-1}{\ell}
x^{\ell+b-1}\right)\right|_{x=1}
=\left.\frac{\mathrm{d}^{n-1}}{\mathrm{d} x^{n-1}}\big((1-x)^{N-1} x^{b-1}\big)\right|_{x=1}
\\
&
=\sum_{j=N-1}^{n-1} \binom{n-1}{j} \left.\frac{\mathrm{d}^j}{\mathrm{d} x^j}
\left((1-x)^{N-1}\right)\right|_{x=1}
\left.\frac{\mathrm{d}^{n-1-j}}{\mathrm{d} x^{n-1-j}}\big(x^{b-1}\big)\right|_{x=1}
\\
&
=(-1)^{N-1}(N-1)! \binom{n-1}{N-1} (b-1)_{n-N}
=(-1)^{N-1}\frac{(b-1)!(n-1)!}{(n-N)!(b+N-n-1)!}.
\end{align*}
Consequently,
\begin{align*}
\mathbb{E}\!\left[\left(S_b^+\right)_n\right]
&
=(-1)^{N-1}b\binom{b+N-1}{b}\frac{(b-1)!(n-1)!}{(n-N)!(b+N-n-1)!}
\\
&
=(-1)^{N-1}\frac{(n-1)!(b+N-1)!}{(n-N)!(N-1)!(b+N-n-1)!}.
\end{align*}
This is the result announced in Theorem~\ref{theorem-moments} when $N\le n\le b+N-1$.

Next, concerning the pseudo-moments of $S_b^+$, we appeal to an elementary
argument of linear algebra: the family $(1,X_1,X_2,\dots,X_n)$
(recall that $X_k=X(X-1)\cdots(X-k+1)$) is a basis of the space of polynomials
of degree not greater than $n$. So, $X^n$ can be written as a linear
combination of $X_1,X_2,\dots,X_n$. Then $\mathbb{E}\!\left[\left(S_b^+\right)^n\right]$
can be written as a linear combination of the factorial pseudo-moments of
$S_b^+$ of order between $1$ and $n$. These latter vanish
for $n\in\{1,\dots, N-1\}$. As a result, $\mathbb{E}\!\left[\left(S_b^+\right)^n\right]=0$.

The same argument ensures the equalities
$\mathbb{E}\!\left[\left(S_b^+\right)_N\right]=\mathbb{E}\!\left[\left(S_b^+\right)^N
\right]$, which is equal to $(-1)^{N-1}(b+N-1)_N$, and
$\mathbb{E}\!\left[\left(S_b^+-b\right)_N\right]=\mathbb{E}\!\left[\left(S_b^+\right)^N\right]
+(-b)^N$ which vanishes. Each of them yields the value of
$\mathbb{E}\!\left[\left(S_b^+\right)^N\right]$.

The proof of Theorem~\ref{theorem-moments} is completed.
\end{proof}

\section{Link with the high-order finite-difference operator}

Set $\mathbf{\Delta}^{\!+}f(i)=f(i+1)-f(i)$ for any $i\in\mathbb{Z}$ and
$(\mathbf{\Delta}^{\!+})^j=\underbrace{\mathbf{\Delta}^{\!+}\circ\dots\circ
\mathbf{\Delta}^{\!+}}_{j\text{ times}}$
for any $j\in\mathbb{N}^*$. Set also $(\mathbf{\Delta}^{\!+})^0f=f$.
The quantity $(\mathbf{\Delta}^{\!+})^j$ is the iterated forward finite-difference operator given by
\[
(\mathbf{\Delta}^{\!+})^jf(i)=\sum_{k=0}^j (-1)^{j+k} \binom{j}{k} f(i+k).
\]
Conversely, $f(i+k)$ can be expressed by means of $(\mathbf{\Delta}^{\!+})^jf(i)$, $0\le j\le k$,
according as
\begin{align}\label{difference}
f(i+k)=\sum_{j=0}^k \binom{k}{j} \,(\mathbf{\Delta}^{\!+})^jf(i).
\end{align}
We have the following expression for any functional of the pseudo-random
variable $S_b^+$.
%
\begin{theorem}\label{theo-exp-f}
We have, for any function $f$ defined on $\{b,b+1,\dots,b+N-1\}$, that
\begin{equation}\label{exp-f-Sb}
\mathbb{E}\!\left[f(S_b^+)\right]
=\sum_{j=0}^{N-1} (-1)^j \binom{j+b-1}{b-1} \,(\mathbf{\Delta}^{\!+})^jf(b).
\end{equation}
\end{theorem}
%
\begin{proof}
By (\ref{difference}), we see that
\[
\mathbb{E}\!\left[f(S_b^+)\right]
=\mathbb{E}\!\left(\sum_{j=0}^{S_b^+-b} \binom{S_b^+-b}{j} (\mathbf{\Delta}^{\!+})^jf(b)\right)
=\sum_{j=0}^{N-1} \mathbb{E}\!\left[ \binom{S_b^+-b}{j}\right] (\mathbf{\Delta}^{\!+})^jf(b)
\]
which immediately yields~(\ref{exp-f-Sb}) thanks to~(\ref{moment_comb}).
\end{proof}
%
\begin{corollary}
The generating function of $S_b^+$ is given by
\begin{equation}\label{gene-Sb}
\mathbb{E}\!\left(\zeta^{S_b^+}\right)
=\zeta^b \sum_{j=0}^{N-1} \binom{j+b-1}{b-1} (1-\zeta)^j.
\end{equation}
\end{corollary}
%
\begin{proof}
Let us apply Theorem~\ref{theo-exp-f} to the function $f(i)=\zeta^i$ for which
we plainly have $(\mathbf{\Delta}^{\!+})^jf(i)=(-1)^j\zeta^i(1-\zeta)^j$.
This immediately yields~(\ref{gene-Sb}).
\end{proof}
%
\begin{remark}
A direct computation with~(\ref{pseudo-dist-Sb}) yields the alternative representation:
\[
\mathbb{E}\!\left(\zeta^{S_b^+}\right)=b\binom{b+N-1}{b} \zeta^b \int_0^1 x^{b-1}
(1-\zeta x)^{N-1}\,\mathrm{d} x.
\]
\end{remark}
%

Of special interest is the case when the starting point of the pseudo-random walk
is any point $x\in\mathbb{Z}$. By translating $b$ into $b-x$ and the function
$f$ into the shifted function $f(.+x)$ in formula~(\ref{exp-f-Sb}),
we get that
\[
\mathbb{E}_x\!\left[f(S_b^+)\right]\stackrel{\text{def}}{=}
\mathbb{E}\!\left[f(x+S_{b-x}^+)\right]
=\sum_{j=0}^{N-1} (-1)^j \binom{j+b-x-1}{j} \,(\mathbf{\Delta}^{\!+})^jf(b).
\]
Thus, we obtain the following result.
%
\begin{theorem}\label{theo-exp-fxb}
We have, for any function $f$ defined on $\{b,b+1,\dots,b+N-1\}$, that
\begin{equation}\label{formula-exp-fxb}
\mathbb{E}_x\!\left[f(S_b^+)\right]=\sum_{j=0}^{N-1} \mathbf{P}_{b,j}^+(x)
\,(\mathbf{\Delta}^{\!+})^jf(b)
\end{equation}
with $\mathbf{P}_{b,0}^+(x)=1$ and, for $j\in\{1,\dots,N-1\}$,
\[
\mathbf{P}_{b,j}^+(x)=\frac{1}{j!}\prod_{k=0}^{j-1} (x-b-k).
\]
The $\mathbf{P}_{b,j}^+$, $0\le j\le N-1$, are Newton interpolation polynomials.
They are of degree not greater than $(N-1)$ and characterized,
for any $k\in\{0,\dots,N-1\}$, by
\[
(\mathbf{\Delta}^{\!+})^k\mathbf{P}_{b,j}^+(b)=\delta_{jk}.
\]
\end{theorem}
%
\begin{proof}
Coming back to the proof of Theorem~\ref{theo-exp-f} and appealing
to Theorem~\ref{theo-dist-Sb}, we write that
\begin{align}
\mathbf{P}_{b,j}^+(x)
&
=\mathbb{E}_x\!\left[\!\binom{S_b^+-b}{j}\!\right]
=\mathbb{E}\!\left[\!\binom{S_{b-x}^++x-b}{j}\!\right]
=\sum_{m=j}^{N-1} \binom{m}{j}\mathbb{P}\{S_{b-x}^+=b-x+m\}
\nonumber\\
&
=\frac{1}{(N-1)!} \sum_{m=j}^{N-1} (-1)^m \binom{m}{j}\!
\binom{N-1}{m} K_m(x)
\label{formula-Pb+}
\end{align}
where, for any $m\in\{0,\dots,N-1\}$,
\[
K_m(x)=\left(\prod_{k=0}^{N-1} (b-x+k)\right)\!/(b-x+m)
=\prod_{0\le k\le N-1\atop k\neq m} (b-x+k).
\]
The expression $K_m(x)$ defines a polynomial of the variable $x$
of degree $(N-1)$, so $\mathbf{P}_{b,j}^+$ is a polynomial of degree not greater
that $(N-1)$. It is obvious that $K_m(b+\ell)=0$ for
$\ell\in\{0,\dots,N-1\}\backslash\{m\}$. On the other hand,
$K_{\ell}(b+\ell)=(-1)^{\ell}(N-1)!/\binom{N-1}{\ell}.$
By putting this into~(\ref{formula-Pb+}), we get that
\[
\mathbf{P}_{b,j}^+(b+\ell)=\binom{\ell}{j}\ind_{\{j\le\ell\}}.
\]
Next, we obtain, for any $k\in\{0,\dots,N-1\}$, that
\begin{align*}
(\mathbf{\Delta}^{\!+})^k\mathbf{P}_{b,j}^+(b)
&
=\sum_{\ell=0}^k (-1)^{k+\ell} \binom{k}{\ell}\mathbf{P}_{b,j}^+(b+\ell)
=\sum_{\ell=j}^k (-1)^{k+\ell} \binom{k}{\ell}\!\binom{\ell}{j}
\\
&
=(-1)^{j+k}\binom{k}{j}\sum_{\ell=0}^{k-j} (-1)^{\ell} \binom{k-j}{\ell}=
\delta_{jk}.
\end{align*}
The proof of Theorem~\ref{theo-exp-fxb} is finished.
\end{proof}
%
We complete this paragraph by stating a strong pseudo-Markov property related to
time $\sigma_b^+$.
%
\begin{theorem}
We have, for any function $f$ defined on $\mathbb{Z}$ and any $n\in\mathbb{N}$,
that
\begin{equation}\label{Markov-b}
\mathbb{E}_x\!\big[f\big(S_{\sigma_b^++n}\big)\big]=\sum_{j=0}^{N-1}
\mathbf{P}_{b,j}^+(x)\,(\mathbf{\Delta}^{\!+})^j \mathbb{E}_b[f(S_n)].
\end{equation}
In (\ref{Markov-b}), the operator $(\mathbf{\Delta}^{\!+})^j$ acts on the variable $b$.
\end{theorem}
%
\begin{proof}
We denote by $\mathbb{P}_{\!x}$ the pseudo-probability associated with
the pseudo-expectation $\mathbb{E}_x$. Actually, it represents the
pseudo-probability related to the pseudo-random walk started at point $x$ at time $0$.
We have, by independence of the $U_j$'s, that
\begin{align*}
\mathbb{E}_x\!\big[f\big(S_{\sigma_b^++n}\big)\big]
&
=\sum_{k,\ell\in\mathbb{N}:\atop b\le \ell\le b+N-1}
\mathbb{E}_x\!\big[
\ind_{\{\sigma_b^+=k,S_b^+=\ell\}} f\!\left(S_b^++U_{k+1}+\dots+U_{k+n}\right)\!\!\big]
\\
&
=\sum_{k,\ell\in\mathbb{N}:\atop b\le \ell\le b+N-1}
\mathbb{P}_{\!x}\{\sigma_b^+=k,S_b^+=\ell\} \,\mathbb{E}[f(\ell+U_1+\dots+U_n)]
\\
&
=\sum_{\ell=b}^{b+N-1} \mathbb{P}_{\!x}\{S_b^+=\ell\} \,\mathbb{E}_{\ell}[f(S_n)]
=\mathbb{E}_x\big[\mathbb{E}_{S_b^+}[f(S_n)]\big].
\end{align*}
Hence, by setting $g(x)=\mathbb{E}_x[f(S_n)]$, we have obtained that
\[
\mathbb{E}_x\!\big[f\big(S_{\sigma_b^++n}\big)\big]=\mathbb{E}_x\big[g(S_b^+)\big]
\]
which proves~(\ref{Markov-b}) thanks to~(\ref{formula-exp-fxb}).
\end{proof}
%
\begin{example}
Below, we display the form of (\ref{Markov-b}) for the particular values
$1,2,3$ of $N$.
\begin{itemize}
\item
For $N=1$, (\ref{Markov-b}) reads
\[
\mathbb{E}_x\!\big[f\big(S_{\sigma_b^++n}\big)\big]
=\mathbb{E}_b[f(S_n)]
\]
which is of course trivial! This is the strong Markov property for the
ordinary random walk.

\item
For $N=2$, (\ref{Markov-b}) reads
\begin{align*}
\mathbb{E}_x\!\big[f\big(S_{\sigma_b^++n}\big)\big]
&
=\mathbb{E}_b[f(S_n)]+(x-b)\,\mathbf{\Delta}^{\!+} \mathbb{E}_b[f(S_n)]
\\
&
=(b-x+1)\,\mathbb{E}_b[f(S_n)]+(x-b)\,\mathbb{E}_{b+1}[f(S_n)].
\end{align*}

\item
For $N=3$, (\ref{Markov-b}) reads
\begin{align*}
\mathbb{E}_x\!\big[f\big(S_{\sigma_b^++n}\big)\big]
&
=\mathbb{E}_b[f(S_n)]+(x-b)\,\mathbf{\Delta}^{\!+} \mathbb{E}_b[f(S_n)]
\\
&
\hphantom{=\;}
+\frac12\,(x-b)(x-b-1)\,(\mathbf{\Delta}^{\!+})^2\, \mathbb{E}_b[f(S_n)]
\\
&
=\frac12\,(x-b-1)(x-b-2)\,\mathbb{E}_b[f(S_n)]-(x-b)(x-b-2)\,\mathbb{E}_{b+1}[f(S_n)]
\\
&
\hphantom{=\;}
+\frac12\,(x-b)(x-b-1)\,\mathbb{E}_{b+2}[f(S_n)].
\end{align*}
\end{itemize}
\end{example}

\section{Joint pseudo-distribution of $\big(\tau_b^+,X_b^+\big)$}\label{section-tau-b}

Below, we give an ad hoc definition for the convergence of a family of exit times.
%
\begin{definition}\label{def2}
Let $((X_t^{\varepsilon})_{t\ge 0})_{\varepsilon>0}$ be a family of pseudo-processes which converges
towards a pseudo-process $(X_t)_{t\ge 0}$ when $\varepsilon\to 0^+$
in the sense of Definition~\ref{def1}.
Let $I$ be a subset of $\mathbb{R}$ and set $\tau_I^{\varepsilon}
=\inf\{t\ge 0:X_t^{\varepsilon}\notin I\}$,
$X_I^{\varepsilon}=X_{\tau_I^{\varepsilon}}^{\varepsilon}$ and
$\tau_I^{\vphantom{\varepsilon}}=\inf\{t\ge 0:X_t\notin I\}$,
$X_I^{\vphantom{\varepsilon}}=X_{\tau_I^{\vphantom{\varepsilon}}}$.

We say that
\[
\big(\tau_I^{\varepsilon},X_I^{\varepsilon}\big) \underset{\varepsilon\to 0^+}{\longrightarrow}
\big(\tau_I^{\vphantom{\varepsilon}},X_I^{\vphantom{\varepsilon}}\big)
\]
if and only if
\[
\forall \lambda>0,\,\forall \mu\in\mathbb{R},
\quad\mathbb{E}\Big(\mathrm{e}^{-\lambda\tau_I^{\varepsilon}
+\mathrm{i} \mu X_I^{\varepsilon}}\ind_{\{\tau_I^{\varepsilon}<+\infty\}}\Big)
\underset{\varepsilon\to 0^+}{\longrightarrow}\mathbb{E}
\Big(\mathrm{e}^{-\lambda\tau_I^{\vphantom{\varepsilon}}+\mathrm{i} \mu X_I^{\vphantom{\varepsilon}}}
\ind_{\{\tau_I^{\vphantom{\varepsilon}}<+\infty\}}\Big).
\]
We say that
\[
X_I^{\varepsilon} \underset{\varepsilon\to 0^+}{\longrightarrow} X_I^{\vphantom{\varepsilon}}
\]
if and only if
\[
\forall \mu\in\mathbb{R},\quad\mathbb{E} \Big(\mathrm{e}^{\mathrm{i}
\mu X_I^{\varepsilon}}\ind_{\{\tau_I^{\varepsilon}<+\infty\}}\Big)
\underset{\varepsilon\to 0^+}{\longrightarrow}
\mathbb{E}\Big(\mathrm{e}^{\mathrm{i} \mu X_I^{\vphantom{\varepsilon}}}
\ind_{\{\tau_I^{\vphantom{\varepsilon}}<+\infty\}}\Big).
\]
\end{definition}
%
As in Section~\ref{section-limit}, we choose for the family
$((X_t^{\varepsilon})_{t\ge 0})_{\varepsilon>0}$
the pseudo-processes defined, for any $\varepsilon>0$, by
\[
X_t^{\varepsilon}=\varepsilon S_{\lfloor t/\varepsilon^{2N}\rfloor}, \quad t\ge 0,
\]
and for the pseudo-process $(X_t)_{t\ge 0}$ the pseudo-Brownian motion.
For $I$, we choose the interval $(-\infty,b)$ so that $\tau_I^{\varepsilon}
=\tau_b^{\varepsilon+}$, $X_I^{\varepsilon}=X_b^{\varepsilon+}$
and $\tau_I^{\vphantom{\varepsilon}}=\tau_b^+$, $X_I^{\vphantom{\varepsilon}}=X_b^+$.
Set $b_{\varepsilon}=\lceil b/\varepsilon\rceil$
where $\lceil\,.\,\rceil$ is the usual ceiling function.
We have $\tau_b^{\varepsilon+}=\varepsilon^{2N}\sigma_{b_{\varepsilon}}^+$
and $X_b^{\varepsilon+}=\varepsilon S_{b_{\varepsilon}}^+$.
Recall the setting $\varphi_j=-\mathrm{i} \,\mathrm{e}^{\mathrm{i}
\frac{2j-1}{2N} \pi}$, $1\le j\le N$.
%
\begin{theorem}\label{theo-conv}
Assume that $c\le 1/2^{2N-1}$. The following convergence holds:
\[
\big(\tau_b^{\varepsilon+},X_b^{\varepsilon+}\big) \underset{\varepsilon\to 0^+}{\longrightarrow}
\big(\tau_b^+,X_b^+\big)
\]
where, for any $\lambda>0$ and any $\mu\in\mathbb{R}$,
\[
\mathbb{E}\Big(\mathrm{e}^{-\lambda\tau_b^++\mathrm{i} \mu X_b^+}
\ind_{\{\tau_b<+\infty\}}\Big)
=\mathrm{e}^{\mathrm{i} \mu b} \sum_{k=1}^N \prod_{1\le j\le N\atop j\neq k}
\frac{\varphi_j}{\varphi_j-\varphi_k}
\prod_{1\le j\le N\atop j\neq k}
\bigg(1-\frac{\mathrm{i}\,\overline{\varphi_j}\,\mu}{\sqrt[2N\!\!]{\lambda/c}}\bigg)\,
\mathrm{e}^{-\varphi_k\!\!\sqrt[2N\!\!]{\lambda/c}\,b}.
\]
\end{theorem}
%
\begin{proof}
We already pointed out that the assumption $c\le 1/2^{2N-1}$ entails
that $M_{\infty}=1$. Therefore, (\ref{double-gene-sigma}) holds for
$z=\mathrm{e}^{-\lambda\varepsilon^{2N}}<1/M_{\infty}=1$, i.e., for $\lambda>0$.
So, by~(\ref{double-gene-sigma}), we have, for $\lambda>0$, that
\begin{align*}
\mathbb{E}\Big(\mathrm{e}^{-\lambda\tau_b^{\varepsilon+}+\mathrm{i} \mu X_b^{\varepsilon+}}
\ind_{\{\tau_b^{\varepsilon+}<+\infty\}}\Big)
&
=\mathbb{E}\Big(\mathrm{e}^{-\lambda\varepsilon^{2N}\sigma_{b_{\varepsilon}}^+
+\mathrm{i} \mu\varepsilon S_{b_{\varepsilon}}^+}\ind_{\{\sigma_{b_{\varepsilon}}^+<+\infty\}}\Big)
\\
&
=\mathrm{e}^{\mathrm{i} \mu\varepsilon b_{\varepsilon}} \sum_{k=1}^{N}
\mathbf{L}_k\Big(\mathrm{e}^{-\lambda\varepsilon^{2N}},\mathrm{e}^{\mathrm{i} \mu\varepsilon}\Big)
u_k\Big(\mathrm{e}^{-\lambda\varepsilon^{2N}}\Big)^{b_{\varepsilon}}.
\end{align*}
Recall that we previously set $u_j(\lambda,\varepsilon)
=u_j\big(\mathrm{e}^{-\lambda\varepsilon^{2N}}\big)$.
Thanks to asymptotics~(\ref{asymptotic-v}), we get that
\[
u_k(\lambda,\varepsilon)-u_j(\lambda,\varepsilon)\underset{\varepsilon\to 0^+}{\sim}
(\varphi_j-\varphi_k)\!\!\sqrt[2N\!\!]{\lambda/c}\,\varepsilon,\quad
1-u_j(\lambda,\varepsilon)\mathrm{e}^{\mathrm{i} \mu\varepsilon}
\underset{\varepsilon\to 0^+}{\sim} (\varphi_j\!\!\sqrt[2N\!\!]{\lambda/c}-\mathrm{i} \mu)
\,\varepsilon.
\]
Thus,
\[
\lim_{\varepsilon\to 0^+} \mathbf{L}_k\Big(\mathrm{e}^{-\lambda\varepsilon^{2N}},
\mathrm{e}^{\mathrm{i} \mu\varepsilon}\Big)=
\prod_{1\le j\le N\atop j\neq k}\frac{\varphi_j\!\!\sqrt[2N\!\!]{\lambda/c}-\mathrm{i} \mu}
{(\varphi_j-\varphi_k)\!\!\sqrt[2N\!\!]{\lambda/c}}
=\prod_{1\le j\le N\atop j\neq k}\frac{\varphi_j}{\varphi_j-\varphi_k}
\prod_{1\le j\le N\atop j\neq k}
\bigg(1-\frac{\mathrm{i}\,\overline{\varphi_j}\,\mu}{\sqrt[2N\!\!]{\lambda/c}}\bigg).
\]
Finally, we can easily conclude with the help of the elementary limits
\[
\lim_{\varepsilon\to 0^+} \mathrm{e}^{\mathrm{i} \mu\varepsilon b_{\varepsilon}}
= \mathrm{e}^{\mathrm{i} \mu b},\quad
\lim_{\varepsilon\to 0^+} u_k(\lambda,\varepsilon)^{b_{\varepsilon}}
=\mathrm{e}^{-\varphi_k\!\!\sqrt[2N\!\!]{\lambda/c}\,b}.
\]
\end{proof}
%
\begin{theorem}\label{theo-conv-bis}
The following convergence holds:
\[
X_b^{\varepsilon+} \underset{\varepsilon\to 0^+}{\longrightarrow} X_b^+
\]
where, for any $\mu\in\mathbb{R}$,
\[
\mathbb{E}\Big(\mathrm{e}^{\mathrm{i} \mu X_b^+}\ind_{\{\tau_b^+<+\infty\}}\Big)
=\mathrm{e}^{\mathrm{i} \mu b}\sum_{j=0}^{N-1} \frac{(-\mathrm{i} \mu b)^j}{j!}.
\]
This is the Fourier transform of the pseudo-random variable $X_b^+$.
Moreover,
\[
\mathbb{P}\{\tau_b^+<+\infty\}=1.
\]
\end{theorem}
%
\begin{proof}
By (\ref{exp-f-Sb}), we have that
\begin{align*}
\mathbb{E}\Big(\mathrm{e}^{\mathrm{i} \mu X_b^{\varepsilon+}}
\ind_{\{\tau_b^{\varepsilon+}<+\infty\}}\Big)
&
=\mathbb{E}\Big(\mathrm{e}^{\mathrm{i} \mu\varepsilon S_{b_{\varepsilon}}^+}
\ind_{\{\sigma_{b_{\varepsilon}}^+<+\infty\}}\Big)
\\
&
=\mathrm{e}^{\mathrm{i} \mu\varepsilon b_{\varepsilon}} \sum_{k=0}^{N-1}
\binom{j+b_{\varepsilon}-1}{b_{\varepsilon}-1}
\left(1-\mathrm{e}^{\mathrm{i} \mu\varepsilon}\right)^j.
\end{align*}
We can easily conclude by using the elementary asymptotics
\[
\lim_{\varepsilon\to 0^+} \mathrm{e}^{\mathrm{i} \mu\varepsilon b_{\varepsilon}}
= \mathrm{e}^{\mathrm{i} \mu b},\quad
\binom{j+b_{\varepsilon}-1}{b_{\varepsilon}-1}\underset{\varepsilon\to 0^+}{\sim}
\frac{b^{\,j}}{j!\,\varepsilon^j},\quad
\left(1-\mathrm{e}^{\mathrm{i} \mu\varepsilon}\right)^j
\underset{\varepsilon\to 0^+}{\sim}(-\mathrm{i} \mu\varepsilon)^j.
\]
\end{proof}
%
\begin{corollary}
The pseudo-distribution of $X_b^+$ is given by
\[
\mathbb{P}\{X_b^+\in \mathrm{d} z\}/\mathrm{d} z=\sum_{j=0}^{N-1}
\frac{b^{\,j}}{j!} \,\delta_b^{(j)}(z).
\]
\end{corollary}
%
This formula should be understood as follows: for any $(N-1)$-times
differentiable function $f$, by omitting the condition $\tau_b^+<+\infty$,
\[
\mathbb{E}\!\left[f\left(X_b^+\right)\right]
=\sum_{j=0}^{N-1} (-1)^j\,\frac{b^{\,j}}{j!} \,f^{(j)}(b).
\]
We retrieve a result of~\cite{la2} and, in the case $N=2$, a pioneering
result of~\cite{nish2}.

\newpage
\begin{center}
\textbf{\Large Part III --- First exit time from a bounded interval}
\end{center}
\vspace{\baselineskip}

\section{On the pseudo-distribution of $(\sigma_{ab},S_{ab})$}

Let $a,b$ be two integers such that $a<0<b$ and let
$\mathcal{E}=\{a,a-1,\dots,a-N+1\}\cup\{b,b+1,\dots,b+N-1\}$.
In this section, we explicitly compute the
generating function of $(\sigma_{ab},S_{ab})$. Set, for $\ell\in \mathcal{E}$,
\[
H_{ab,\ell}(z)=\mathbb{E}\!\left(z^{\sigma_{ab}}\ind_{\{S_{ab}=\ell,\sigma_{ab}<+\infty\}}\right)
=\sum_{k\in\mathbb{N}} \mathbb{P}\{\sigma_{ab}=k,S_{ab}=\ell\}z^k.
\]
We are able to provide an explicit expression of $H_{ab,\ell}(z)$.
As in Section~\ref{section-sigma-b}, due to~(\ref{bound-Fbis}),
we have the following \textit{a priori} estimate:
$|\mathbb{P}\{\sigma_{ab}=k,S_{ab}=\ell\}|\le
|\mathbb{P}\{S_1\in (a,b),\dots,S_{k-1}\in (a,b),S_k=\ell\}|\le M_1^k$.
As a byproduct, the power series defining $H_{ab,\ell}(z)$ absolutely converges
for $|z|<1/M_1$.

\subsection{Joint pseudo-distribution of $(\sigma_{ab},S_{ab})$}

%
\begin{theorem}\label{theo-sigma-ab}
The pseudo-distribution of $(\sigma_{ab},S_{ab})$ is characterized by the identity,
valid for any $z\in(0,1)$,
\begin{equation}\label{pseudo-dist-sigma-ab}
\mathbb{E}\!\left(z^{\sigma_{ab}}\ind_{\{S_{ab}=\ell,\sigma_{ab}<+\infty\}}\right)
=\frac{W_{\ell}(u_1(z),\dots,u_{2N}(z))}{W(u_1(z),\dots,u_{2N}(z))}
\end{equation}
where
\[
W(u_1,\dots,u_{2N}) =\begin{vmatrix}
1      & u_1    & \dots & u_1^{N-1}    & u_1^{b-a+N-1}    & \dots & u_1^{b-a+2N-2}   \\
\vdots & \vdots &       & \vdots       & \vdots           &       & \vdots           \\[1ex]
1      & u_{2N} & \dots & u_{2N}^{N-1} & u_{2N}^{b-a+N-1} & \dots & u_{2N}^{b-a+2N-2}
\end{vmatrix}
\]
and, if $a-N+1\le\ell\le a$, $W_{\ell}(u_1,\dots,u_{2N})$ is the determinant
\[
\left|\!\!\begin{array}{ccccccccccc}
1      & u_1     & \dots & u_1^{\ell+N-a-2}    & u_1^{N-a-1}    & u_1^{\ell+N-a}    & \dots  & u_1^{N-1}    & u_1^{b-a+N-1}    & \dots & u_1^{b-a+2N-2}    \\
\vdots & \vdots  &       & \vdots              & \vdots         & \vdots            &        & \vdots       & \vdots           &       & \vdots            \\[1ex]
1      & u_{2N}  & \dots & u_{2N}^{\ell+N-a-2} & u_{2N}^{N-a-1} & u_{2N}^{\ell+N-a} & \dots  & u_{2N}^{N-1} & u_{2N}^{b-a+N-1} & \dots & u_{2N}^{b-a+2N-2} \\
\end{array}\!\!\right|\!,
\]
if $b\le\ell\le b+N-1$, $W_{\ell}(u_1,\dots,u_{2N})$ is the determinant
\[
\left|\!\!\begin{array}{ccccccccccc}
1      & u_1     & \dots & u_1^{N-1}    & u_1^{b-a+N-1}     & \dots & u_1^{\ell+N-a-2}    & u_1^{N-a-1}    & u_1^{\ell+N-a}    & \dots  & u_1^{b-a+2N-2}    \\
\vdots & \vdots  &       & \vdots       & \vdots            &       & \vdots              & \vdots         & \vdots            &        & \vdots            \\[1ex]
1      & u_{2N}  & \dots & u_{2N}^{N-1} & u_{2N}^{b-a+N-1}  & \dots & u_{2N}^{\ell+N-a-2} & u_{2N}^{N-a-1} & u_{2N}^{\ell+N-a} & \dots  & u_{2N}^{b-a+2N-2} \\
\end{array}\!\!\right|\!.
\]
\end{theorem}
%
\begin{proof}
Pick an integer $k$ such that $k\le a$ or $k\ge b$. If $S_n=k$, then an exit of
the interval $(a,b)$ occurs before time $n$: $\sigma_{ab}\le n$. This remark and the independence of the
increments of the pseudo-random walk entail that
\begin{align}
\mathbb{P}\{S_n=k\}
&
=\mathbb{P}\{S_n=k,\sigma_{ab}\le n\}=\sum_{j=0}^n\sum_{\ell\in \mathcal{E}} \mathbb{P}\{S_n=k,\sigma_{ab}=j,S_{ab}=\ell\}
\nonumber\\
&
=\sum_{j=0}^n\sum_{\ell\in \mathcal{E}} \mathbb{P}\{\sigma_{ab}=j,S_{ab}=\ell\}
\mathbb{P}\{S_{n-j}=k-\ell\}.
\label{convol2}
\end{align}
Thanks to the absolute convergence of the series defining $G_k(z)$ and
$H_{ab,\ell}(z)$ for $z\in(0,1)$ and $|z|<1/M_1$ respectively,
we can apply the generating function to
equality~(\ref{convol2}). We get, for $z\in(0,1/M_1)$, that
\[
G_k(z)=\sum_{\ell\in \mathcal{E}} G_{k-\ell}(z) H_{ab,\ell}(z).
\]
Using expression~(\ref{expression-Gk}) of $G_k$, namely
$G_k(z)=\sum_{j=1}^N \alpha_j(z)\,u_j(z)^{|k|}$, we get,
for $k\ge b+N-1$ (recall that $v_j(z)=1/u_j(z)$), that
\begin{equation}\label{system2-equations-inter1}
\sum_{j=1}^N \alpha_j(z)\,u_j(z)^{k}
\left(\sum_{\ell\in \mathcal{E}} H_{ab,\ell}(z)v_j(z)^{\ell}-1\right)=0,
\end{equation}
and, for $k\le a-N+1$, that
\begin{equation}\label{system2-equations-inter2}
\sum_{j=1}^N \alpha_j(z)\,v_j(z)^{k}
\left(\sum_{\ell\in \mathcal{E}} H_{ab,\ell}(z)u_j(z)^{\ell}-1\right)=0.
\end{equation}
When limiting the range of $k$ to the set $\{b+N,b+N+1,\dots,b+2N-1\}$
in (\ref{system2-equations-inter1}) and to the set
$\{a-2N+1,a-2N+2,\dots,a-N\}$ in (\ref{system2-equations-inter2}),
we see that (\ref{system2-equations-inter1}) and
(\ref{system2-equations-inter2}) are homogeneous
Vandermonde systems whose solutions are trivial, that is, the terms within
parentheses in~(\ref{system2-equations-inter1}) and
(\ref{system2-equations-inter2}) vanish. Thus, we get the two systems below:
\[
\sum_{\ell\in \mathcal{E}} H_{ab,\ell}(z)u_j(z)^{\ell}=1, \quad 1\le j\le N.
\]
and
\[
\sum_{\ell\in \mathcal{E}} H_{ab,\ell}(z)v_j(z)^{\ell}=1, \quad 1\le j\le N.
\]
It will be convenient to relabel the $u_j(z)$'s and $v_j(z)$'s,
$1\le j\le N$, as $u_j(z)=v_{j+N}(z)$ and $v_j(z)=u_{j+N}(z)$;
note that $v_j(z)=1/u_j(z)$ for any $j\in\{1,\dots,2N\}$
and $\{u_1(z),\dots,u_{2N}(z)\}$ $=\{v_1(z),\dots,v_{2N}(z)\}$.
By using the relabeling $u_j(z),v_j(z)$, $1\le j\le 2N$,
we obtain the two equivalent following systems of $2N$ equations
and $2N$ unknowns, $u_1(z),\dots,u_{2N}(z)$ for the first one,
$v_1(z),\dots,v_{2N}(z)$ for the second one:
\begin{equation}\label{system-equations1}
\sum_{\ell\in \mathcal{E}} H_{ab,\ell}(z)u_j(z)^{\ell}=1, \quad 1\le j\le 2N,
\end{equation}
and
\begin{equation}\label{system-equations2}
\sum_{\ell\in \mathcal{E}} H_{ab,\ell}(z) v_j(z)^{\ell}=1, \quad 1\le j\le 2N.
\end{equation}
Systems (\ref{system-equations1}) and (\ref{system-equations2}) are ``lacunary''
Vandermonde systems (some powers of $u_j(z)$ are missing).
For instance, let us rewrite system~(\ref{system-equations1})
as
\begin{equation}\label{system-equations3}
\sum_{\ell\in \mathcal{E}} H_{ab,\ell}(z)u_j(z)^{\ell+N-a-1}=u_j(z)^{N-a-1},
\quad 1\le j\le 2N.
\end{equation}
Cramer's formulae immediately yield~(\ref{pseudo-dist-sigma-ab}) at least
for $z\in(0,1/M_1)$. By analyticity of the $u_j$'s on $(0,1)$,
it is easily seen that~(\ref{pseudo-dist-sigma-ab}) holds true
for $z\in(0,1)$. Systems (\ref{system-equations1}) and (\ref{system-equations2})
will be used in Lemma~\ref{lemma-syst-Hbis}.
\end{proof}

A method for computing the determinants exhibited in Theorem~\ref{theo-sigma-ab}
and solving system~(\ref{system-equations3}) is proposed in
Appendix~\ref{appendix-determinants}.
In particular, we can deduce from Proposition~\ref{proposition-cramer-lacunary}
an alternative representation of $\mathbb{E}\!\left(z^{\sigma_{ab}}
\ind_{\{S_{ab}=\ell,\sigma_{ab}<+\infty\}}\right)$ which can be seen as the
analogous of (\ref{joint-dist}). Set
$s_0(z)=1$ and, for $k,\ell\in\{1,\dots,2N\}$,
\[
s_{\ell}(z)=\sum_{1\le i_1<\dots<i_{\ell}\le 2N} u_{i_1}(z)\cdots u_{i_{\ell}}(z),
\quad
p_k(z)=\prod_{1\le i\le 2N\atop i\neq k} [u_k(z)-u_i(z)],
\]
$s_{k,0}(z)=1$, $s_{k,m}(z)=0$
for any integer $m$ such that $m\le -1$ or $m\ge 2N$ and, for $m\in\{1,\dots,2N-1\}$,
\[
s_{k,m}(z)=\sum_{1\le i_1<\dots<i_m\le 2N\atop
i_1,\dots,i_m\neq k} u_{i_1}(z)\cdots u_{i_m}(z).
\]
Set also
\[
\tilde{W}(z)=\begin{vmatrix}
s_N(z)         & s_{N-1}(z)     & \dots & s_{N-b+a+2}(z) \\
s_{N+1}(z)     & s_N(z)         & \dots & s_{N-b+a+1}(z) \\
\vdots         & \vdots         &       & \vdots         \\
s_{N+b-a-2}(z) & s_{N+b-a-3}(z) & \dots & s_N(z)
\end{vmatrix}\!,
\]
\[
\tilde{W}_{k\ell}(z)=\begin{vmatrix}
s_{k,N}(z)          & s_{k,N-1}(z)         & \dots & s_{k,N-b+a+1}(z)    \\
s_{k,N+1}(z)        & s_{k,N}(z)           & \dots & s_{k,N-b+a}(z)      \\
\vdots              & \vdots               &       & \vdots              \\
s_{k,N+b-a-2}(z)    & s_{k,N+b-a-3}(z)     & \dots & s_{k,N-1}(z)        \\
s_{k,N+b-\ell-1}(z) & s_{k,N+b-\ell-2}(z) & \dots & s_{k,N+a-\ell}(z)
\end{vmatrix}\!.
\]
Then, applying Proposition~\ref{proposition-cramer-lacunary} with the choices
$p=r=N$ and $q=b-a-1$ leads, for any $\ell\in\mathcal{E}$, to
\begin{equation}\label{pseudo-dist-sigma-ab-bis}
\mathbb{E}\!\left(z^{\sigma_{ab}}\ind_{\{S_{ab}=\ell,\sigma_{ab}<+\infty\}}\right)
=\frac{(-1)^{\ell+N-a-1}}{\tilde{W}(z)} \sum_{k=1}^{2N}
\frac{\tilde{W}_{k\ell}(z)}{p_k(z)}\,u_k(z)^{N-a-1}.
\end{equation}

The double generating function defined by
\[
\mathbb{E}\!\left(z^{\sigma_{ab}}\zeta^{S_{ab}}\ind_{\{\sigma_{ab}<+\infty\}}\right)
=\sum_{\ell\in \mathcal{E}} \mathbb{E}\!\left(z^{\sigma_{ab}}\ind_{\{S_{ab}
=\ell,\sigma_{ab}<+\infty\}}\right)\zeta^{\ell}
\]
admits an interesting representation by means of interpolation polynomials that
we display in the following theorem.
%
\begin{theorem}\label{theo-sigma-ab2}
The double generating function of $(\sigma_{ab},S_{ab})$ is given,
for $z\in(0,1/M_1)$ and $\zeta\in\mathbb{C}$, by
\begin{equation}\label{double-gene-sigma-ab}
\mathbb{E}\!\left(z^{\sigma_{ab}}\zeta^{S_{ab}}\ind_{\{\sigma_{ab}<+\infty\}}\right)
=\sum_{k=1}^{2N} \mathbf{\tilde{L}}_k(z,\zeta) \left(v_k(z)\zeta\right)^{a-N+1}
\end{equation}
where
\[
\mathbf{\tilde{L}}_k(z,\zeta)=P_k(z,\zeta)\prod_{1\le j\le 2N\atop j\neq k}
\frac{\zeta-u_j(z)}{u_k(z)-u_j(z)},\quad k\in\{1,\dots,2N\}
\]
are interpolation polynomials with respect to the variable $\zeta$
satisfying $\mathbf{\tilde{L}}_k(z,u_j(z))=\delta_{jk}$ and $P_k(z,\zeta)$
are some polynomials with respect to the variable $\zeta$ of degree $(b-a-1)$.
\end{theorem}
%
\begin{proof}
By (\ref{pseudo-dist-sigma-ab}), we have that
\[
\mathbb{E}\!\left(z^{\sigma_{ab}}\zeta^{S_{ab}}\ind_{\{S_{ab}\le a,\sigma_{ab}<+\infty\}}\right)
=\sum_{\ell=a-N+1}^a H_{ab,\ell}(z) \zeta^{\ell}
=\sum_{\ell=a-N+1}^a \frac{W_{\ell}(u_1(z),\dots,u_{2N}(z))}{W(u_1(z),\dots,u_{2N}(z))} \,\zeta^{\ell}.
\]
In order to simplify the text, we omit the variable $z$.
We expand the determinant $W_{\ell}(u_1,\dots,u_{2N})$, $a-N+1\le\ell\le a$ with respect
to its $(\ell+N-a)$th column:
\[
W_{\ell}(u_1,\dots,u_{2N})=\sum_{k=1}^{2N} u_k^{N-a-1} W_{k\ell}(u_1,\dots,u_{2N})
\]
where $W_{k\ell}(u_1,\dots,u_{2N})$, $1\le k\le 2N$, is the determinant
\[
\left|\!\!\begin{array}{ccccccccccc}
1      & u_1     & \dots & u_1^{\ell+N-a-2}     & 0      & u_1^{\ell+N-a}     & \dots  & u_1^{N-1}     & u_1^{b-a+N-1}     & \dots & u_1^{b-a+2N-2}     \\
\vdots & \vdots  &       & \vdots               & \vdots & \vdots             &        & \vdots        & \vdots            &       & \vdots             \\[1ex]
1      & u_{k-1} & \dots & u_{k-1}^{\ell+N-a-2} & 0      & u_{k-1}^{\ell+N-a} & \dots  & u_{k-1}^{N-1} & u_{k-1}^{b-a+N-1} & \dots & u_{k-1}^{b-a+2N-2} \\[1ex]
1      & u_k     & \dots & u_k^{\ell+N-a-2}     & 1      & u_k^{\ell+N-a}     & \dots  & u_k^{N-1}     & u_k^{b-a+N-1}     & \dots & u_k^{b-a+2N-2}     \\[1ex]
1      & u_{k+1} & \dots & u_{k+1}^{\ell+N-a-2} & 0      & u_{k+1}^{\ell+N-a} & \dots  & u_{k+1}^{N-1} & u_{k+1}^{b-a+N-1} & \dots & u_{k+1}^{b-a+2N-2} \\
\vdots & \vdots  &       & \vdots               & \vdots & \vdots             &        & \vdots        & \vdots            &       & \vdots             \\[1ex]
1      & u_{2N}  & \dots & u_{2N}^{\ell+N-a-2}  & 0      & u_{2N}^{\ell+N-a}  & \dots  & u_{2N}^{N-1}  & u_{2N}^{b-a+N-1}  & \dots & u_{2N}^{b-a+2N-2}  \\
\end{array}\!\!\right|
\]
which plainly coincides with
\[
\left|\!\!\begin{array}{c@{\hspace{.9em}}c@{\hspace{.9em}}c@{\hspace{.9em}}c@{\hspace{.9em}}c@{\hspace{.9em}}c@{\hspace{.9em}}c@{\hspace{.9em}}c@{\hspace{.9em}}c@{\hspace{.9em}}c@{\hspace{.9em}}c}
1      & u_1     & \dots & u_1^{\ell+N-a-2}     & u_1^{\ell+N-a-1}     & u_1^{\ell+N-a}     & \dots  & u_1^{N-1}     & u_1^{b-a+N-1}     & \dots & u_1^{b-a+2N-2}     \\
\vdots & \vdots  &       & \vdots               & \vdots               & \vdots             &        & \vdots        & \vdots            &       & \vdots             \\[1ex]
1      & u_{k-1} & \dots & u_{k-1}^{\ell+N-a-2} & u_{k-1}^{\ell+N-a-1} & u_{k-1}^{\ell+N-a} & \dots  & u_{k-1}^{N-1} & u_{k-1}^{b-a+N-1} & \dots & u_{k-1}^{b-a+2N-2} \\[1ex]
0      & 0       & \dots & 0                    & 1                    & 0                  & \dots  & 0             & 0                 & \dots & 0                  \\[1ex]
1      & u_{k+1} & \dots & u_{k+1}^{\ell+N-a-2} & u_{k+1}^{\ell+N-a-1} & u_{k+1}^{\ell+N-a} & \dots  & u_{k+1}^{N-1} & u_{k+1}^{b-a+N-1} & \dots & u_{k+1}^{b-a+2N-2} \\
\vdots & \vdots  &       & \vdots               & \vdots               & \vdots             &        & \vdots        & \vdots            &       & \vdots             \\[1ex]
1      & u_{2N}  & \dots & u_{2N}^{\ell+N-a-2}  & u_{2N}^{\ell+N-a-1}  & u_{2N}^{\ell+N-a}  & \dots  & u_{2N}^{N-1}  & u_{2N}^{b-a+N-1}  & \dots & u_{2N}^{b-a+2N-2}  \\
\end{array}\!\!\right|\!.
\]
Therefore, we obtain that
\begin{align*}
\sum_{\ell=a-N+1}^a W_{\ell}(u_1,\dots,u_{2N}) \zeta^{\ell}
&
=\sum_{\ell=a-N+1}^a\left(\sum_{k=1}^{2N} u_k^{N-a-1} W_{k\ell}(u_1,\dots,u_{2N}) \right)\zeta^{\ell}
\\
&
=\zeta^{a-N+1}\sum_{k=1}^{2N}\left(\sum_{\ell=a-N+1}^a
W_{k\ell}(u_1,\dots,u_{2N})\zeta^{\ell+N-a-1} \right) v_k^{a-N+1}.
\end{align*}
Next, we can see that the foregoing sum within parentheses is the
expansion of the following determinant with respect to its $k$th row
(by putting back the variable $z$):
\[
D_k^-(z,\zeta) =\begin{vmatrix}
1      & u_1(z)     & \dots & u_1^{N-1}(z)     & u_1^{b-a+N-1}(z)     & \dots & u_1^{b-a+2N-2}(z)    \\
\vdots & \vdots     &       & \vdots           & \vdots               &       & \vdots               \\[1ex]
1      & u_{k-1}(z) & \dots & u_{k-1}^{N-1}(z) & u_{k-1}^{b-a+N-1}(z) & \dots & u_{k-1}^{b-a+2N-2}(z)\\[1ex]
1      & \zeta      & \dots & \zeta^{N-1}      & 0                    & \dots & 0                    \\[1ex]
1      & u_{k+1}(z) & \dots & u_{k+1}^{N-1}(z) & u_{k+1}^{b-a+N-1}(z) & \dots & u_{k+1}^{b-a+2N-2}(z)\\
\vdots & \vdots     &       & \vdots           & \vdots               &       & \vdots               \\[1ex]
1      & u_{2N}(z)  & \dots & u_{2N}^{N-1}(z)  & u_{2N}^{b-a+N-1}(z)  & \dots & u_{2N}^{b-a+2N-2}(z) \\
\end{vmatrix}\!,
\]
As a result, by setting
\begin{align*}
D(z)&=W(u_1(z),\dots,u_{2N}(z))
\\
&=\begin{vmatrix}
1      & u_1(z)    & \dots & u_1^{N-1}(z)    & u_1^{b-a+N-1}(z)    & \dots & u_1^{b-a+2N-2}(z)   \\
\vdots & \vdots    &       & \vdots          & \vdots              &       & \vdots              \\[1ex]
1      & u_{2N}(z) & \dots & u_{2N}^{N-1}(z) & u_{2N}^{b-a+N-1}(z) & \dots & u_{2N}^{b-a+2N-2}(z)
\end{vmatrix}\!,
\end{align*}
we obtain that
\begin{equation}\label{db-gene1}
\mathbb{E}\!\left(z^{\sigma_{ab}}\zeta^{S_{ab}}\ind_{\{S_{ab}\le a,\sigma_{ab}<+\infty\}}\right)
=\frac{1}{D(z)}\sum_{k=1}^N D_k^-(z,\zeta)(v_k(z)\zeta)^{a-N+1}.
\end{equation}
Similarly, we could check that
\begin{equation}\label{db-gene2}
\mathbb{E}\!\left(z^{\sigma_{ab}}\zeta^{S_{ab}}\ind_{\{S_{ab}\ge b,\sigma_{ab}<+\infty\}}\right)
=\frac{1}{D(z)}\sum_{k=1}^N D_k^+(z,\zeta)(v_k(z)\zeta)^{a-N+1}
\end{equation}
where $D_k^+(z,\zeta)$ is the determinant
\[
D_k^+(z,\zeta) =\begin{vmatrix}
1      & u_1(z)     & \dots & u_1^{N-1}(z)     & u_1^{b-a+N-1}(z)     & \dots & u_1^{b-a+2N-2}(z)    \\
\vdots & \vdots     &       & \vdots           & \vdots               &       & \vdots               \\[1ex]
1      & u_{k-1}(z) & \dots & u_{k-1}^{N-1}(z) & u_{k-1}^{b-a+N-1}(z) & \dots & u_{k-1}^{b-a+2N-2}(z)\\[1ex]
0      & 0          & \dots & 0                & \zeta^{b-a+ N-1}     & \dots & \zeta^{b-a+ 2N-2}    \\[1ex]
1      & u_{k+1}(z) & \dots & u_{k+1}^{N-1}(z) & u_{k+1}^{b-a+N-1}(z) & \dots & u_{k+1}^{b-a+2N-2}(z)\\
\vdots & \vdots     &       & \vdots           & \vdots               &       & \vdots               \\[1ex]
1      & u_{2N}(z)  & \dots & u_{2N}^{N-1}(z)  & u_{2N}^{b-a+N-1}(z)  & \dots & u_{2N}^{b-a+2N-2}(z) \\
\end{vmatrix}.
\]
By adding~(\ref{db-gene1}) and~(\ref{db-gene2}) and setting
\[
D_k(z,\zeta)=D_k^+(z,\zeta)+D_k^-(z,\zeta)
=W(u_1(z),\dots,u_{k-1}(z),\zeta,u_{k+1}(z),\dots,u_{2N}(z)),
\]
we obtain that
\begin{equation}\label{double-gene-sigmabis}
\mathbb{E}\!\left(z^{\sigma_{ab}}\zeta^{S_{ab}}\ind_{\{\sigma_{ab}<+\infty\}}\right)
=\frac{1}{D(z)}\sum_{k=1}^{2N} D_k(z,\zeta) \left(v_k(z)\zeta\right)^{a-N+1}.
\end{equation}
We observe that the polynomials $\mathbf{\tilde{L}}_k(z,\zeta)=D_k(z,\zeta)/D(z)$ with respect
to the variable $\zeta$ are of degree $b-a+2N-2$ and satisfy the equalities
$\mathbf{\tilde{L}}_k(z,u_j(z))=\delta_{jk}$ for all $j\in \mathcal{E}$ and $k\in\{1,\dots,2N\}$.
Hence they can be expressed by means of the elementary Lagrange polynomials as
displayed in Theorem~\ref{theo-sigma-ab2}.
\end{proof}
%
\begin{example}
For $N=2$, (\ref{double-gene-sigma-ab}) reads
\begin{align*}
\mathbb{E}\!\left(z^{\sigma_{ab}}\zeta^{S_{ab}}\ind_{\{\sigma_{ab}<+\infty\}}\right)
&
=\mathbf{\tilde{L}}_1(z,\zeta) \left(v_1(z)\zeta\right)^{a-1}
+\mathbf{\tilde{L}}_2(z,\zeta) \left(v_2(z)\zeta\right)^{a-1}
\\
&
\hphantom{=\;}+\mathbf{\tilde{L}}_3(z,\zeta) \left(v_3(z)\zeta\right)^{a-1}
+\mathbf{\tilde{L}}_4(z,\zeta) \left(v_4(z)\zeta\right)^{a-1}
\end{align*}
where
\[
\mathbf{\tilde{L}}_1(z,\zeta)=\frac{W(\zeta,u_2(z),u_3(z),u_4(z))}{W(u_1(z),u_2(z),u_3(z),u_4(z))},\quad
\mathbf{\tilde{L}}_2(z,\zeta)=\frac{W(u_1(z),\zeta,u_3(z),u_4(z))}{W(u_1(z),u_2(z),u_3(z),u_4(z))},
\]
\[
\mathbf{\tilde{L}}_3(z,\zeta)=\frac{W(u_1(z),u_2(z),\zeta,u_4(z))}{W(u_1(z),u_2(z),u_3(z),u_4(z))},\quad
\mathbf{\tilde{L}}_4(z,\zeta)=\frac{W(u_1(z),u_2(z),u_3(z),\zeta)}{W(u_1(z),u_2(z),u_3(z),u_4(z))}.
\]
All the polynomials $\mathbf{\tilde{L}}_k(z,\zeta)$, $1\le k\le 4$, have the form
$A_{k,a-1}(z)+A_{k,a}(z)\zeta+A_{k,b}(z)\zeta^{b-a+1}$ $+A_{k,b+1}(z)\zeta^{b-a+2}$.
\end{example}
%
\begin{remark}
By expanding the determinant $D_k(z,\zeta)$ with respect to its $k$th raw,
we obtain an expansion for the polynomial $\mathbf{\tilde{L}}(z,\zeta)$ as a linear combination
of $1,\zeta,\dots,\zeta^{N-1},$ $\zeta^{b-a+N-1},\zeta^{b-a+N},\dots,$ $\zeta^{b-a+2N-2}$,
that is, an expansion of the form
\[
\mathbf{\tilde{L}}_k(z,\zeta)=\sum_{\ell=0}^{N-1} A_{k,\ell+a-N+1}(z) \zeta^{\ell}
+\sum_{\ell=b-a+N-1}^{b-a+2N-2} A_{k,\ell+a-N+1}(z) \zeta^{\ell}
=\zeta^{N-a-1}\sum_{\ell\in \mathcal{E}} A_{k\ell}(z)\zeta^{\ell}.
\]
Hence,
\[
\mathbb{E}\!\left(z^{\sigma_{ab}}\zeta^{S_{ab}}\ind_{\{\sigma_{ab}<+\infty\}}\right)
=\sum_{k=1}^{2N}\left(\sum_{\ell\in \mathcal{E}} A_{k\ell}(z)\zeta^{\ell}\right)v_k(z)^{a-N+1}
=\sum_{\ell\in \mathcal{E}} \left(\sum_{k=1}^{2N} A_{k\ell}(z) v_k(z)^{a-N+1}\right) \zeta^{\ell}
\]
from which we extract, for any $\ell\in \mathcal{E}$, that
\[
\mathbb{E}\!\left(z^{\sigma_{ab}}\ind_{\{S_{ab}=\ell,\sigma_{ab}<+\infty\}}\right)
=\sum_{k=1}^{2N} A_{k\ell}(z) u_k(z)^{N-a-1}.
\]
Actually, the foregoing sum comes from the quotient $W_{\ell}(u_1(z),\dots,u_{2N}(z))
/W(u_1(z),\dots,$ $u_{2N}(z))$ given by~(\ref{pseudo-dist-sigma-ab})
by expanding the determinant $W_{\ell}(u_1(z),\dots,u_{2N}(z))$
with respect to the $(\ell+N-a)$th column or $(\ell+N-b+1)$th column
according as $a-N+1\le \ell\le a$ or $b\le \ell\le b+N-1$.
\end{remark}
%

\subsection{Pseudo-distribution of $S_{ab}$}

In order to derive the pseudo-distribution of $S_{ab}$ which is characterized
by the numbers $H_{ab,\ell}(1)$, $\ell\in \mathcal{E}$, we solve the systems obtained
by taking the limit in~(\ref{double-gene-sigmabis}) as $z\to 1^-$.
%
\begin{lemma}\label{lemma-syst-Hbis}
The following identities hold: for $N\le k\le 2N-1$,
\begin{align}
&\sum_{\ell=b}^{b+N-1} \binom{\ell+N-a-1}{k} H_{ab,\ell}(1)=\binom{N-a-1}{k},
\label{system-H(1)1}\\
&\sum_{\ell=a-N+1}^{a} \binom{b+N-1-\ell}{k} H_{ab,\ell}(1)=\binom{b+N-1}{k}.
\label{system-H(1)2}
\end{align}
\end{lemma}
%
\begin{proof}
By~(\ref{asymptotic-u}), we have the expansion $u_j(z)=1+\varepsilon_j(z)$
where $\varepsilon_j(z)\underset{z\to 1^-}{=}\mathcal{O}\!\left(\!\!\sqrt[2N\!]{1-z}\,\right)$
for any $j\in\{1,\dots,N\}$.
Actually such asymptotics holds true for any $j\in\{1,\dots,2N\}$ because
of the equality $u_j=1/u_{j-N}$ for $j\in\{N+1,\dots,2N\}$.
We put this into systems (\ref{system-equations1}) and
(\ref{system-equations2}). For doing this, it is convenient to rewrite these latter as
\begin{align*}
&\sum_{\ell\in \mathcal{E}} H_{ab,\ell}(z)u_j(z)^{\ell+N-a-1}=u_j(z)^{N-a-1}, \quad 1\le j\le 2N,
\\
&\sum_{\ell\in \mathcal{E}} H_{ab,\ell}(z)u_j(z)^{b+N-1-\ell}=u_j(z)^{b+N-1}, \quad 1\le j\le 2N.
\end{align*}
We obtain that
\begin{align}
&\sum_{\ell\in \mathcal{E}} (1+\varepsilon_j(z))^{\ell+N-a-1} H_{ab,\ell}(z)=(1+\varepsilon_j(z))^{N-a-1}, \quad 1\le j\le 2N,
\label{system-asymptotics1}\\
&\sum_{\ell\in \mathcal{E}} (1+\varepsilon_j(z))^{b+N-1-\ell} H_{ab,\ell}(z)=(1+\varepsilon_j(z))^{b+N-1}, \quad 1\le j\le 2N.
\label{system-asymptotics2}
\end{align}
System~(\ref{system-asymptotics1}) writes
\begin{align}
\lqn{\sum_{k=0}^{b-a+2N-2} \left(\sum_{\ell\in \mathcal{E}:\atop \ell\ge k+a-N+1} \binom{\ell+N-a-1}{k}
H_{ab,\ell}(z) \right) \varepsilon_j(z)^k}
&=\sum_{k=0}^{N-a-1} \binom{N-a-1}{k} \varepsilon_j(z)^k.
\label{eqHbis}
\end{align}
Set
\begin{align*}
M_k(z)&=\sum_{\ell\in \mathcal{E}:\atop \ell\ge k+a-N+1} \binom{\ell+N-a-1}{k} H_{ab,\ell}(z)
-\binom{N-a-1}{k},
\\
R_j(z) &=-\sum_{k=2N}^{b-a+2N-2} M_k(z)\varepsilon_j(z)^k.
\end{align*}
Then, equality~(\ref{eqHbis}) reads
\begin{align*}
\sum_{k=0}^{2N-1} M_k(z)\varepsilon_j(z)^k = R_j(z),\quad 1\le j\le 2N.
\end{align*}
This is a Vandermonde system which can be solved as in the proof of
Lemma~\ref{lemma-syst-H} upon changing $N$ into $2N$.
We can check that $\lim_{z\to 1^-} M_k(z)=0$, which entails that
\begin{align}
\sum_{\ell\in \mathcal{E}: \atop \ell\ge k+a-N+1} \binom{\ell+N-a-1}{k} H_{ab,\ell}(1)
=\binom{N-a-1}{k}, \quad 0\le k\le 2N-1.
\label{system-H(1)1bis}
\end{align}
Similarly, using~(\ref{system-asymptotics2}), we can prove that
\begin{align}
\sum_{\ell\in \mathcal{E}: \atop \ell\le b+N-1-k} \binom{b+N-\ell-1}{k} H_{ab,\ell}(1)
=\binom{N+b-1}{k}, \quad 0\le j\le 2N-1.
\label{system-H(1)2bis}
\end{align}
Actually, we shall only use (\ref{system-H(1)1bis}) and
(\ref{system-H(1)2bis}) restricted to $j\in\{N,\dots,2N-1\}$ which immediately
yields system~(\ref{system-H(1)1})-(\ref{system-H(1)2}).
\end{proof}

%
Now, we state one of the most important result of this work. We solve
the famous problem of the ``gambler's ruin'' in the context of the pseudo-random
walk.
%
\begin{theorem}\label{theo-dist-Sab}
The pseudo-distribution of $S_{ab}$ is given, for $\ell\in\{0,1,\dots,N-1\}$, by
\begin{align*}
\mathbb{P}\{S_{ab}=a-\ell,\sigma_{ab}<+\infty\}
=(-1)^{\ell}\frac{KN}{\ell-a}\binom{N-1}{\ell}/\binom{\ell+b-a+N-1}{N},
\\[1ex]
\mathbb{P}\{S_{ab}=b+\ell,\sigma_{ab}<+\infty\}
=(-1)^{\ell}\frac{KN}{\ell+b}\binom{N-1}{\ell}/\binom{\ell+b-a+N-1}{N}
\end{align*}
where
\[
K=\binom{N-a-1}{N}\!\binom{N+b-1}{N}=\frac{(-1)^N}{N!}\,a(a-1)\dots(a-N+1)
\,b(b+1)\dots(b+N-1).
\]
Moreover, $\mathbb{P}\{\sigma_{ab}<+\infty\}=1$ and
\begin{align*}
\mathbb{P}\{\sigma_b^+<\sigma_a^-\}&=KN^2
\int\!\!\!\int_{\mathcal{D}^+} u^{-a-1}(1-u)^{N-1} v^{b-1}(1-v)^{N-1} \,\mathrm{d} u\,\mathrm{d} v,
\\
\mathbb{P}\{\sigma_a^-<\sigma_b^+\}&=KN^2
\int\!\!\!\int_{\mathcal{D}^-} u^{-a-1}(1-u)^{N-1} v^{b-1}(1-v)^{N-1} \,\mathrm{d} u\,\mathrm{d} v,
\end{align*}
where
\begin{align*}
\mathcal{D}^+ =\{(u,v)\in\mathbb{R}^2:0\le v\le u\le 1\},\quad
\mathcal{D}^- =\{(u,v)\in\mathbb{R}^2:0\le u\le v\le 1\}.
\end{align*}
\end{theorem}
%
\begin{proof}
We have to solve system~(\ref{system-H(1)1})-(\ref{system-H(1)2}).
For (\ref{system-H(1)1}) for instance, the principal matrix and the right-hand
side matrix are
\[
\begin{bmatrix}\displaystyle\binom{\ell+b-a+N-1}{k+N}\end{bmatrix}_{0\le k,\ell\le N-1}\quad
\text{and}\quad\begin{bmatrix}\displaystyle\binom{N-a-1}{k+N}\end{bmatrix}_{0\le k\le N-1}
\]
and the matrix form of the solution is given by
\[
\begin{bmatrix}\displaystyle\binom{\ell+b-a+N-1}{k+N}\end{bmatrix}_{0\le k,\ell\le N-1}^{-1}
\begin{bmatrix}\displaystyle\binom{N-a-1}{k+N}\end{bmatrix}_{0\le k\le N-1}.
\]
The computation of this product being quite fastidious, we postponed
it to Appendix~\ref{appendix-matrices}. The result is given by
Theorem~\ref{theorem-matrix-inv}:
\[
\begin{bmatrix}\displaystyle(-1)^{\ell}\frac{KN}{\ell+b}\binom{N-1}{\ell}
/\binom{\ell+b-a+N-1}{N}\end{bmatrix}_{0\le \ell\le N-1}\!.
\]
The entries of this matrix provide the pseudo-probabilities
$\mathbb{P}\{S_{ab}=b+\ell,\sigma_{ab}<+\infty\}$, $0\le\ell\le N-1$, which are
exhibited in Theorem~\ref{theo-dist-Sab}. The analogous formula for
$\mathbb{P}\{S_{ab}=a-\ell,\sigma_{ab}<+\infty\}$ holds true in the same way.

Next, by observing that $\sigma_{ab}=\min(\sigma_a^-,\sigma_b^+)$ and that
$\{\sigma_{ab}<+\infty\}=\{\sigma_a^-<+\infty\}\cup\{\sigma_b^+<+\infty\}$,
we have that
\begin{align}
\mathbb{P}\{\sigma_b^+<\sigma_a^-\}
&
=\mathbb{P}\{\sigma_b^+<\sigma_a^-,\sigma_{ab}<+\infty\}=\mathbb{P}\{S_{ab}\ge b,\sigma_{ab}<+\infty\}
\\
&
=\sum_{\ell=0}^{N-1}\mathbb{P}\{S_{ab}=b+\ell,\sigma_{ab}<+\infty\}
\nonumber\\
&=K\sum_{\ell=0}^{N-1} (-1)^{\ell}\frac{N}{\ell+b}\binom{N-1}{\ell}
/\binom{\ell+b-a+N-1}{N}.
\label{proba-b-before-a}
\end{align}
Noticing that
$1/(\ell+b)=\int_0^1 y^{\ell+b-1}\,\mathrm{d} y$ and $1/\binom{\ell+b-a+N-1}{N}
=N\int_0^1 x^{\ell+b-a-1}(1-x)^{N-1}\,\mathrm{d} x$, we get that
\begin{align*}
\lqn{\sum_{\ell=0}^{N-1} (-1)^{\ell}\frac{N}{\ell+b}\binom{N-1}{\ell}
/\binom{\ell+b-a+N-1}{N}}
&=N^2\int_0^1\!\!\!\int_0^1 \left(\sum_{\ell=0}^{N-1}
(-1)^{\ell}\binom{N-1}{\ell}(xy)^{\ell}\right)x^{b-a-1}(1-x)^{N-1}y^{b-1}\,\mathrm{d} x\,\mathrm{d} y
\\
&=N^2\int_0^1\!\!\!\int_0^1 x^{b-a-1}(1-x)^{N-1}y^{b-1}(1-xy)^{N-1}\,\mathrm{d} x\,\mathrm{d} y.
\end{align*}
The computations can be pursued by performing the change of variables
$(u,v)=(x,xy)$ in the above integral:
\begin{align}
\lqn{\sum_{\ell=0}^{N-1} (-1)^{\ell}\frac{N}{\ell+b}\binom{N-1}{\ell}
/\binom{\ell+b-a+N-1}{N}}
&=N^2\int\!\!\!\int_{\mathcal{D}^+} u^{-a-1}(1-u)^{N-1}v^{b-1}(1-v)^{N-1}\,\mathrm{d} u\,\mathrm{d} v.
\label{sum}
\end{align}
Putting~(\ref{sum}) into~(\ref{proba-b-before-a}) yields the expression
of $\mathbb{P}\{\sigma_b^+<\sigma_a^-\}$ displayed in Theorem~\ref{theo-dist-Sab}.
The similar expression for $\mathbb{P}\{\sigma_a^-<\sigma_b^+\}$ holds true. Finally,
\begin{align*}
\mathbb{P}\{\sigma_{ab}<+\infty\} &=\mathbb{P}\{\sigma_a^-<\sigma_b^+,\sigma_{ab}<+\infty\}
+\mathbb{P}\{\sigma_b^+<\sigma_a^-,\sigma_{ab}<+\infty\}
\\
&=\mathbb{P}\{\sigma_a^-<\sigma_b^+\}+\mathbb{P}\{\sigma_b^+<\sigma_a^-\}
\\
&=KN^2\int_0^1\!\!\!\int_0^1 u^{-a-1}(1-u)^{N-1} v^{b-1}(1-v)^{N-1} \,\mathrm{d} u\,\mathrm{d} v.
\end{align*}
The foregoing integral is quite elementary:
\begin{align*}
\lqn{\int_0^1\!\!\!\int_0^1 u^{-a-1}(1-u)^{N-1} v^{b-1}(1-v)^{N-1} \,\mathrm{d} u\,\mathrm{d} v}
&=\int_0^1 u^{-a-1}(1-u)^{N-1}\,\mathrm{d} u \int_0^1v^{b-1}(1-v)^{N-1} \,\mathrm{d} v
\\
&=B(N,-a)B(N,b)=\frac{1}{N^2\binom{N-a-1}{N}\!\binom{b+N-1}{N}}=\frac{1}{KN^2}
\end{align*}
which entails that $\mathbb{P}\{\sigma_{ab}<+\infty\}=1$.
\end{proof}
%
In the sequel, when considering $S_{ab}$, we shall omit the condition $\sigma_{ab}<+\infty$.
%
\begin{example}
Let us have a look on the particular values $1,2,3$ of $N$.
\begin{itemize}
\item
\textsl{Case $N=1$.} In this case $S_{ab}\in\{a,b\}$ and
\[
\mathbb{P}\{S_{ab}=a\}=\mathbb{P}\{\sigma_a^-<\sigma_b^+\}=\frac{b}{b-a},
\quad\mathbb{P}\{S_{ab}=b\}=\mathbb{P}\{\sigma_b^+<\sigma_a^-\}=-\frac{a}{b-a}.
\]
We retrieve one of the most well-know and important result for the ordinary random walk:
this is the famous problem of the gambler's ruin!
\item
\textsl{Case $N=2$.} In this case $S_{ab}\in\{a-1,a,b,b+1\}$ and
\[
\mathbb{P}\{S_{ab}=a-1\}=\frac{ab(b+1)}{(b-a+1)(b-a+2)},\quad \mathbb{P}\{S_{ab}=a\}
=-\frac{(a-1)b(b+1)}{(b-a)(b-a+1)},
\]
\[
\mathbb{P}\{S_{ab}=b\}=\frac{a(a-1)(b+1)}{(b-a)(b-a+1)},\quad \mathbb{P}\{S_{ab}=b+1\}
=-\frac{a(a-1)b}{(b-a+1)(b-a+2)},
\]
\begin{align*}
\mathbb{P}\{\sigma_a^-<\sigma_b^+\}&=\frac{b(b+1)(b-3a+2)}{(b-a)(b-a+1)(b-a+2)},
\\
\mathbb{P}\{\sigma_b^+<\sigma_a^-\}&=\frac{a(a-1)(3b-a+2)}{(b-a)(b-a+1)(b-a+2)}.
\end{align*}
\item
\textsl{Case $N=3$.} In this case $S_{ab}\in\{a-2,a-1,a,b,b+1,b+2\}$ and
\begin{align*}
\mathbb{P}\{S_{ab}=a-2\}&=\frac{a(a-1)b(b+1)(b+2)}{(b-a+2)(b-a+3)(b-a+4)},\\
\mathbb{P}\{S_{ab}=a-1\}&=-\frac{a(a-2)b(b+1)(b+2)}{(b-a+1)(b-a+2)(b-a+3)},\\
\mathbb{P}\{S_{ab}=a\}  &=\frac{(a-1)(a-2)b(b+1)(b+2)}{(b-a)(b-a+1)(b-a+2)},\\
\mathbb{P}\{S_{ab}=b\}  &=-\frac{a(a-1)(a-2)(b+1)(b+2)}{(b-a)(b-a+1)(b-a+2)},\\
\mathbb{P}\{S_{ab}=b+1\}&=\frac{a(a-1)(a-2)b(b+2)}{(b-a+1)(b-a+2)(b-a+3)},\\
\mathbb{P}\{S_{ab}=b+2\}&=-\frac{a(a-1)(a-2)b(b+1)}{(b-a+2)(b-a+3)(b-a+4)},
\end{align*}
\begin{align*}
\mathbb{P}\{\sigma_a^-<\sigma_b^+\}&=\frac{b(b+1)(b+2)(10a^2-5ab+b^2-25a+7b+12)}{
(b-a)(b-a+1)(b-a+2)(b-a+3)(b-a+4)},
\\
\mathbb{P}\{\sigma_b^+<\sigma_a^-\}&=-\frac{a(a-1)(a-2)(a^2-5ab+10b^2-7a+25b+12)}{
(b-a)(b-a+1)(b-a+2)(b-a+3)(b-a+4)}.
\end{align*}
\end{itemize}

\end{example}
%
\subsection{Pseudo-moments of $S_{ab}$}

Let us recall the notation we previously introduced in
Section~\ref{subsection-pseudo-moments-Sb}:
$(i)_n=i(i-1)(i-2)\cdots (i-n+1)$ for any $i\in\mathbb{Z}$ and any $n\in\mathbb{N}^*$ and
$(i)_0=1$, as well as the conventions $1/i!=0$ for any negative integer $i$ and
$\sum_{k=i}^j=0$ if $i>j$.

In this section, we compute several functionals related to the pseudo-moments
of $S_{ab}$. Namely, we provide formulae for
$\mathbb{E}\!\left[S_{ab}(S_{ab}-b)_{n-1}\right]$ (Theorem~\ref{theo-moments-Sab}),
$\mathbb{E}\!\left[(S_{ab})^n\right]$ (Corollary~\ref{theo-moments-Sab-bis}) and
$\mathbb{E}\!\left[(S_{ab}-b)_n\right]$ (Theorem~\ref{theo-moments-Sab-ter}).
This schedule may seem surprising; actually, we have been able to carry
out the calculations by following this chronology.

Putting the identities
$1/\binom{\ell+b-a+N-1}{N}=N\int_0^1 x^{\ell+b-a-1}(1-x)^{N-1}\,\mathrm{d} x$
and $1/(\ell+b)=\int_0^1 y^{\ell+b-1}\,\mathrm{d} y$ into the equality
\[
\mathbb{E}\!\left[f(S_{ab})\ind_{\{S_{ab}\ge b\}}\right]
=KN \sum_{\ell=0}^{N-1} (-1)^{\ell}\frac{f(\ell+b)}{\ell+b}
\binom{N-1}{\ell}/\binom{\ell+b-a+N-1}{N},
\]
we immediately get the following integral representations
for $\mathbb{E}\big[f(S_{ab})\ind_{\{S_{ab}\ge b\}}\big]$,
and the analogous ones hold true for
$\mathbb{E}\big[f(S_{ab})\ind_{\{S_{ab}\le a\}}\big]$.
%
\begin{theorem}\label{th-moment-betabis}
For any function $f$ defined on $\mathcal{E}$,
\begin{align}
\lqn{\mathbb{E}\!\left[f(S_{ab})\ind_{\{S_{ab}\ge b\}}\right]}
&
=KN^2\int_0^1\left[\sum_{\ell=0}^{N-1} (-1)^{\ell}\binom{N-1}{\ell}
\frac{f(\ell+b)}{\ell+b}\,x^{\ell}\right] x^{b-a-1}(1-x)^{N-1}\,\mathrm{d} x
\label{moment-beta1bis}\\
&
=KN^2\int_0^1\!\!\!\int_0^1\left[\sum_{\ell=0}^{N-1} (-1)^{\ell}\binom{N-1}{\ell}
f(\ell+b)(xy)^{\ell}\right] x^{b-a-1}(1-x)^{N-1}y^{b-1}\,\mathrm{d} x\,\mathrm{d} y,
\label{moment-beta1bisbis}\\
\lqn{\mathbb{E}\!\left[f(S_{ab})\ind_{\{S_{ab}\le a\}}\right]}
&
=KN^2\int_0^1\left[\sum_{\ell=0}^{N-1} (-1)^{\ell}\binom{N-1}{\ell}
\,\frac{f(a-\ell)}{\ell-a}\,x^{\ell}\right] x^{b-a-1}(1-x)^{N-1}\,\mathrm{d} x
\label{moment-beta2bis}\\
&
=KN^2\int_0^1\!\!\!\int_0^1\left[\sum_{\ell=0}^{N-1} (-1)^{\ell}\binom{N-1}{\ell}
f(a-\ell)(xy)^{\ell}\right] x^{b-a-1}(1-x)^{N-1}y^{-a-1}\,\mathrm{d} x\,\mathrm{d} y.
\label{moment-beta2bisbis}
\end{align}
\end{theorem}
%
In view of~(\ref{moment-beta1bis}) and~(\ref{moment-beta2bis}) and in order
to compute the pseudo-moments of $S_{ab}$, it is convenient to introduce
the function $f_n$ defined by $f_n(i)=i(i-b)_{n-1}$ for any integers $i$
and $n$ such that $n\ge 1$. In particular, $f_1(i)=i$.
We immediately see that, by choosing $f=f_1$ in Theorem~\ref{th-moment-betabis},
quantities (\ref{moment-beta1bis}) and (\ref{moment-beta2bis}) are opposite.
As a by product, $\mathbb{E}\!\left[S_{ab}\right]=0$. More generally,
we have the results below.
%
\begin{theorem}\label{theo-moments-Sab}
For any positive integer $n$,
\begin{align}
\lqn{\mathbb{E}\!\left[S_{ab}(S_{ab}-b)_{n-1}\ind_{\{S_{ab}\ge b\}}\right]}
&=\begin{cases}
\displaystyle(-1)^{n-1} \frac{KN}{2N-n}\,\frac{N!}{(N-n)!}\,
/\binom{2N+b-a-2}{2N-n} & \text{if $1\le n\le N$},
\\
0  & \text{if $n\ge N+1$},
\end{cases}
\label{moment-beta1ter}
\end{align}
\begin{align}
\lqn{\mathbb{E}\!\left[S_{ab}(S_{ab}-b)_{n-1}\ind_{\{S_{ab}\le a\}}\right]}
&=\begin{cases}
\displaystyle(-1)^n \frac{KN}{2N-n}\,\frac{N!}{(N-n)!}\,/\binom{2N+b-a-2}{2N-n} &
\text{if $1\le n\le N$},
\\
0  & \text{if $N+1\le n\le 2N-1$},
\\
\displaystyle (-1)^{N+n-1}KN N!(n-N-1)!\binom{n+b-a-2}{n-2N}
& \text{if $n\ge 2N$}.
\end{cases}
\label{moment-beta2ter}
\end{align}
In particular, for any $n\in\{1,\dots,2N-1\}$,
\[
\mathbb{E}\!\left[S_{ab}(S_{ab}-b)_{n-1}\right]=0.
\]
\end{theorem}
%
\begin{proof}
By (\ref{moment-beta1bis}), we get that
\[
\mathbb{E}\!\left[f_n(S_{ab})\ind_{\{S_{ab}\ge b\}}\right]
=KN^2\int_0^1\left[\sum_{\ell=0}^{N-1} (-1)^{\ell}\binom{N-1}{\ell}
(\ell)_{n-1}\,x^{\ell+b-a-1}\right] (1-x)^{N-1}\,\mathrm{d} x.
\]
By noticing that $(\ell)_{n-1}=0$ for $\ell\in\{0,\dots,n-2\}$ and
$\binom{N-1}{\ell}(\ell)_{n-1}=\binom{N-n}{\ell-n+1}(N-1)_{n-1}$ for $\ell\ge n-1$,
we obtain that
\begin{align*}
\sum_{\ell=0}^{N-1} (-1)^{\ell}\binom{N-1}{\ell}(\ell)_{n-1}\,x^{\ell}
&
=(N-1)_{n-1} \ind_{\{1\le n\le N\}}\sum_{\ell=n-1}^{N-1}
(-1)^{\ell}\binom{N-n}{\ell-n+1}\,x^{\ell}
\\
&
=(N-1)_{n-1} \ind_{\{1\le n\le N\}}\, x^{n-1}\sum_{\ell=0}^{N-n}
(-1)^{\ell+n-1}\binom{N-n}{\ell}\,x^{\ell}
\\
&=\begin{cases}
\displaystyle(-1)^{n-1}\frac{(N-1)!}{(N-n)!} \,x^{n-1}(1-x)^{N-n}
&\text{if $1\le n\le N$},
\\
0 &\text{if $n\ge N+1$}.
\end{cases}
\end{align*}
Hence, if $1\le n\le N$,
\begin{align*}
\mathbb{E}\!\left[f_n(S_{ab})\ind_{\{S_{ab}\ge b\}}\right]
&=(-1)^{n-1}KN^2\frac{(N-1)!}{(N-n)!}\int_0^1  x^{n+b-a-2}(1-x)^{2N-n-1}\,\mathrm{d} x
\\
&=(-1)^{n-1}\frac{KNN!}{(N-n)!} \times\frac{(n+b-a-2)!(2N-n-1)!}{(2N+b-a-2)!}
\end{align*}
and we arrive at (\ref{moment-beta1ter}). Moreover, if $n\ge N+1$
and $S_{ab}\ge b$, we have $S_{ab}\in\{b,b+1,$ $\dots,b+N-1\}$.
Then, it is clear that $f_n(S_{ab})=0$ and (\ref{moment-beta1ter})
still holds in this case.

On the other hand, by Theorem~\ref{theo-dist-Sab}, we get that
\begin{align*}
\mathbb{E}\!\left[f_n(S_{ab})\ind_{\{S_{ab}\le a\}}\right]
&=\sum_{\ell=0}^{N-1}\mathbb{P}\{S_{ab}=a-\ell\} f_n(a-\ell)
\\
&=KN\sum_{\ell=0}^{N-1}(-1)^{\ell-1}(a-b-\ell)_{n-1}\binom{N-1}{\ell}
/\binom{\ell+b-a+N-1}{N}
\\
&=(-1)^nKNN!\sum_{\ell=0}^{N-1} (-1)^{\ell}\binom{N-1}{\ell}
\frac{(\ell+b-a+n-2)!}{(\ell+b-a+N-1)!}.
\end{align*}
For $n\le N$, we can write that
\begin{align*}
\frac{(\ell+b-a+n-2)!}{(\ell+b-a+N-1)!}
&=\frac{1}{(N-n)!}\,\frac{1}{(\ell+b-a+N-1)\binom{\ell+b-a+N-2}{N-n}}
\\
&=\frac{1}{(N-n)!}\int_0^1 x^{\ell+n+b-a-2}(1-x)^{N-n}\,\mathrm{d} x.
\end{align*}
Then,
\begin{align*}
\mathbb{E}\!\left[f_n(S_{ab})\ind_{\{S_{ab}\le a\}}\right]
&=(-1)^n\frac{KNN!}{(N-n)!} \int_0^1\sum_{\ell=0}^{N-1} (-1)^{\ell}\binom{N-1}{\ell}
x^{\ell+n+b-a-2}(1-x)^{N-n}\,\mathrm{d} x
\\
&=(-1)^n\frac{KNN!}{(N-n)!} \int_0^1 x^{n+b-a-2}(1-x)^{2N-n-1}\,\mathrm{d} x
\\
&=(-1)^n\frac{KNN!}{(N-n)!}\times\frac{(n+b-a-2)!(2N-n-1)!}{(2N+b-a-2)!}
\end{align*}
which proves (\ref{moment-beta2ter}). For $n\ge N+1$, we write instead that
\[
\frac{(\ell+b-a+n-2)!}{(\ell+b-a+N-1)!}
=\left.\frac{\mathrm{d}^{n-N-1}}{\mathrm{d} x^{n-N-1}}\big(x^{\ell+n+b-a-2}\big)\right|_{x=1}.
\]
Therefore,
\begin{align*}
\mathbb{E}\!\left[f_n(S_{ab})\ind_{\{S_{ab}\le a\}}\right]
&=(-1)^n KNN! \left.\frac{\mathrm{d}^{n-N-1}}{\mathrm{d} x^{n-N-1}}\left(\sum_{\ell=0}^{N-1}
(-1)^{\ell}\binom{N-1}{\ell} x^{\ell+n+b-a-2}\right)\right|_{x=1}
\\
&=(-1)^n KNN! \left.\frac{\mathrm{d}^{n-N-1}}{\mathrm{d} x^{n-N-1}} \left(x^{n+b-a-2}
(1-x)^{N-1}\right)\right|_{x=1}.
\end{align*}
If $N+1\le n\le 2N-1$, the above derivative vanishes since $1$ is a root of
multiplicity $N-1$ of the polynomial $x^{\ell+n+b-a-2}(1-x)^{N-1}$.
Finally, for $n\ge 2N$, we appeal to Leibniz rule for evaluating the derivative
of interest:
\begin{align*}
\lqn{\left.\frac{\mathrm{d}^{n-N-1}}{\mathrm{d} x^{n-N-1}} \left(x^{n+b-a-2}
(1-x)^{N-1}\right)\right|_{x=1}}
&
=\sum_{k=0}^{n-N-1} (-1)^k\binom{n-N-1}{k}\frac{(n+b-a-2)!}{(k+N+b-a-1)!}
\,(N-1)_k \,\delta_{k,N-1}
\\
&
=(-1)^{N-1}(N-1)!\,\frac{(n+b-a-2)!}{(2N+b-a-2)!}\binom{n-N-1}{N-1}
\\
&
=(-1)^{N-1}(n-N-1)!\binom{n+b-a-2}{n-2N}.
\end{align*}
This proves (\ref{moment-beta2ter}) in this case.
\end{proof}
%
\begin{corollary}\label{theo-moments-Sab-bis}
For $n\in\{1,\dots,2N-1\}$, the pseudo-moment of $S_{ab}$ of order $n$ vanishes:
\[
\mathbb{E}\!\left[(S_{ab})^n\right]=0.
\]
Moreover,
\begin{equation}\label{moment-2N}
\mathbb{E}\!\left[(S_{ab})^{2N}\right]=-a(a-1)\cdots(a-N+1)b(b+1)\cdots (b+N-1).
\end{equation}
\end{corollary}
%
\begin{proof}
As in the proof of Theorem~\ref{theorem-moments}, we appeal to the following
argument: the polynomial $X^n$ is a linear combination of
$f_1(X),\dots,f_n(X)$. Then $\mathbb{E}\!\left[(S_{ab})^n\right]$ can be
written as a linear combination of $\mathbb{E}[f_1(S_{ab})],\dots,\mathbb{E}[f_n(S_{ab})]$
which vanish when $1\le n\le 2N-1$. Thus,
$\mathbb{E}\!\left[(S_{ab})^n\right]=0$. The same argument entails that
\[
\mathbb{E}\!\left[(S_{ab})^{2N}\right]=\mathbb{E}[f_{2N}(S_{ab})]=(-1)^{N-1} KN!^2
\]
which proves~(\ref{moment-2N}).
\end{proof}
%
In the theorem below, we provide an integral representation for
certain factorial pseudo-moments of $(S_{ab}-a)$ and $(S_{ab}-b)$
which will be used in next section.
%
\begin{theorem}\label{theo-moments-Sab-ter}
For any integer $n\in\{0,\dots,N-1\}$,
\begin{align}
\mathbb{E}\!\left[(S_{ab}-b)_n\ind_{\{S_{ab}\ge b\}}\right]
&
=(-1)^n KN^2 (N-1)_n
\nonumber\\
&
\hphantom{=\;}\times
\int\!\!\!\int_{\mathcal{D}^+} u^{-a-1}(1-u)^{N-1} v^{n+b-1}(1-v)^{N-n-1}\,\mathrm{d} u\,\mathrm{d} v,
\label{integral-rep1}\\
\mathbb{E}\!\left[(a-S_{ab})_n\ind_{\{S_{ab}\le a\}}\right]
&
=(-1)^n KN^2 (N-1)_n
\nonumber\\
&
\hphantom{=\;}\times
\int\!\!\!\int_{\mathcal{D}^-} u^{n-a-1}(1-u)^{N-n-1} v^{b-1}(1-v)^{N-1}\,\mathrm{d} u\,\mathrm{d} v.
\label{integral-rep2}
\end{align}
\end{theorem}
%
The above identities can be rewritten as
\begin{align}
\mathbb{E}\!\left[\binom{S_{ab}-b}{n}\ind_{\{S_{ab}\ge b\}}\right]
&
=(-1)^n KN^2 \binom{N-1}{n}
\nonumber\\
&
\hphantom{=\;}\times
\int\!\!\!\int_{\mathcal{D}^+} u^{-a-1}(1-u)^{N-1} v^{n+b-1}(1-v)^{N-n-1}\,\mathrm{d} u\,\mathrm{d} v,
\label{moment-comb-ab-bis1}\\
\mathbb{E}\!\left[\binom{a-S_{ab}}{n}\ind_{\{S_{ab}\le a\}}\right]
&
=(-1)^n KN^2 \binom{N-1}{n}
\nonumber\\
&
\hphantom{=\;}\times
\int\!\!\!\int_{\mathcal{D}^-} u^{n-a-1}(1-u)^{N-n-1} v^{b-1}(1-v)^{N-1}\,\mathrm{d} u\,\mathrm{d} v.
\label{moment-comb-ab-bis2}
\end{align}
%
\begin{proof}
By~(\ref{moment-beta1bisbis}), we have that
\begin{align*}
\lqn{\mathbb{E}\!\left[\left(S_{ab}-b\right)_n\ind_{\{S_{ab}\ge b\}}\right]}
=KN^2\int_0^1\!\!\!\int_0^1\left[\sum_{\ell=0}^{N-1} (-1)^{\ell}\binom{N-1}{\ell}
(\ell)_n(xy)^{\ell}\right] x^{b-a-1}(1-x)^{N-1}y^{b-1}\,\mathrm{d} x\,\mathrm{d} y.
\end{align*}
The sum lying in the above integral can be easily calculated:
\begin{align*}
\sum_{\ell=0}^{N-1} (-1)^{\ell}\binom{N-1}{\ell}(\ell)_n(xy)^{\ell}
&
=(N-1)_n\sum_{\ell=n}^{N-1} (-1)^{\ell}\binom{N-n-1}{\ell-n}(xy)^{\ell}
\\
&
=(-1)^n(N-1)_n(xy)^{n}(1-xy)^{N-n-1}.
\end{align*}
Hence
\begin{align*}
\lqn{\mathbb{E}\!\left[\left(S_{ab}-b\right)_n\ind_{\{S_{ab}\ge b\}}\right]}
=(-1)^n(N-1)_n KN^2\int_0^1\!\!\!\int_0^1 x^{n+b-a-1}(1-x)^{N-1}y^{n+b-1}
(1-xy)^{N-n-1}\,\mathrm{d} x\,\mathrm{d} y.
\end{align*}
Performing the change of variables $(u,v)=(x,xy)$ in the foregoing
integral immediately yields~(\ref{integral-rep1}).
Formula~(\ref{integral-rep2}) can be deduced from~(\ref{moment-beta2bisbis})
exactly in the same way.
\end{proof}

\section{Link with high-order finite-difference equations}

Set $\mathbf{\Delta}^{\!+}f(i)=f(i+1)-f(i)$ and
$\mathbf{\Delta}^{\!-}f(i)=f(i)-f(i-1)$ for any $i\in\mathbb{Z}$ and
$(\mathbf{\Delta}^{\!+})^j=\underbrace{\mathbf{\Delta}^{\!+}\circ\dots\circ
\mathbf{\Delta}^{\!+}}_{j\text{ times}}$
and $(\mathbf{\Delta}^{\!-})^j=\underbrace{\mathbf{\Delta}^{\!-}\circ\dots
\circ\mathbf{\Delta}^{\!-}}_{j\text{ times}}$
for any $j\in\mathbb{N}^*$. Set also $(\mathbf{\Delta}^{\!+})^0f
=(\mathbf{\Delta}^{\!-})^0f=f$.
The quantities $(\mathbf{\Delta}^{\!+})^j$ and $(\mathbf{\Delta}^{\!-})^j$
are the iterated forward and backward
finite-difference operators given by
\[
(\mathbf{\Delta}^{\!+})^jf(i)=\sum_{k=0}^j (-1)^{j+k} \binom{j}{k} f(i+k),\quad
(\mathbf{\Delta}^{\!-})^jf(i)=\sum_{k=0}^j (-1)^{k} \binom{j}{k} f(i-k).
\]
Conversely, $f(i+k)$ and $f(i-k)$ can be expressed by means of $(\mathbf{\Delta}^{\!+})^jf(i),
(\mathbf{\Delta}^{\!-})^jf(i)$, $0\le j\le k$,
according as
\begin{align}\label{differencebis}
f(i+k)=\sum_{j=0}^k \binom{k}{j} (\mathbf{\Delta}^{\!+})^jf(i),\quad
f(i-k)=\sum_{j=0}^k (-1)^j\binom{k}{j} (\mathbf{\Delta}^{\!-})^jf(i).
\end{align}
We have the following expression for any functional of the pseudo-random
variable $S_{ab}$.
%
\begin{theorem}\label{theo-exp-fSab}
We have, for any function $f$ defined on $\mathcal{E}$, that
\begin{equation}\label{exp-f-Sab}
\mathbb{E}[f(S_{ab})]= \sum_{j=0}^{N-1} I_{ab,j}^- \,(\mathbf{\Delta}^{\!-})^jf(a)
+\sum_{j=0}^{N-1} I_{ab,j}^+ \,(\mathbf{\Delta}^{\!+})^jf(b)
\end{equation}
with
\begin{align*}
I_{ab,j}^-&=KN^2 \binom{N-1}{j} \int\!\!\!\int_{\mathcal{D}^-}
u^{j-a-1}(1-u)^{N-j-1} v^{b-1}(1-v)^{N-1}\,\mathrm{d} u\,\mathrm{d} v,
\\
I_{ab,j}^+&=(-1)^j KN^2 \binom{N-1}{j}\int\!\!\!\int_{\mathcal{D}^+}
u^{-a-1}(1-u)^{N-1} v^{j+b-1}(1-v)^{N-j-1}\,\mathrm{d} u\,\mathrm{d} v.
\end{align*}
\end{theorem}
%
\begin{proof}
By (\ref{differencebis}), we see that
\begin{align}
\mathbb{E}[f(S_{ab})]
&
=\mathbb{E}\!\left(\ind_{\{S_{ab}\le a\}}\sum_{j=0}^{a-S_{ab}}
(-1)^j\binom{a-S_{ab}}{j} (\mathbf{\Delta}^{\!-})^jf(a)\right)
\nonumber\\
&
\hphantom{=\;}+\mathbb{E}\!\left(\ind_{\{S_{ab}\ge b\}}\sum_{j=0}^{S_{ab}-b}
\binom{S_{ab}-b}{j} (\mathbf{\Delta}^{\!+})^jf(b)\right)
\nonumber\\
&
=\sum_{j=0}^{N-1} (-1)^j\mathbb{E}\!\left[\ind_{\{S_{ab}\le a\}}
\binom{a-S_{ab}}{j}\!\right] (\mathbf{\Delta}^{\!-})^jf(a)
\nonumber\\
&
\hphantom{=\;}+\sum_{j=0}^{N-1} \mathbb{E}\!\left[\ind_{\{S_{ab}\ge b\}}
\binom{S_{ab}-b}{j}\!\right] (\mathbf{\Delta}^{\!+})^jf(b)
\label{exp-f-Sab-comb}
\end{align}
which immediately yields~(\ref{exp-f-Sab}) thanks to~(\ref{moment-comb-ab-bis1})
and (\ref{moment-comb-ab-bis2}).
\end{proof}
%
\begin{corollary}
The generating function of $S_{ab}$ is given by
\begin{equation}\label{gene-Sab}
\mathbb{E}\!\left(\zeta^{S_{ab}}\right)
=\zeta^a \sum_{j=0}^{N-1} I_{ab,j}^- (1-1/\zeta)^j+
\zeta^b \sum_{j=0}^{N-1} I_{ab,j}^+ (\zeta-1)^j.
\end{equation}
\end{corollary}
%
\begin{proof}
Let us apply Theorem~\ref{theo-exp-fSab} to the function $f(i)=\zeta^i$ for which
we plainly have $(\mathbf{\Delta}^{\!+})^jf(i)=\zeta^i(\zeta-1)^j$ and
$(\mathbf{\Delta}^{\!-})^jf(i)=\zeta^i(1-1/\zeta)^j$. This immediately yields~(\ref{gene-Sab}).
\end{proof}
%

Of special interest is the case when the starting point of the pseudo-random walk
is some point $x\in\mathbb{Z}$. By translating $a,b$ into $a-x,b-x$ and the function
$f$ into the shifted function $f(.+x)$, we have that
\[
\mathbb{E}_x[f(S_{ab})]=\mathbb{E}[f(x+S_{a-x,b-x})]
=\sum_{j=0}^{N-1} I_{a-x,b-x,j}^- \,(\mathbf{\Delta}^{\!-})^jf(a)
+\sum_{j=0}^{N-1} I_{a-x,b-x,j}^+ \,(\mathbf{\Delta}^{\!+})^jf(b).
\]
More precisely, we have the following result.
%
\begin{theorem}\label{theo-exp-fxab}
We have, for any function $f$ defined on $\mathcal{E}$, that
\begin{equation}\label{exp-f(x)-Sab}
\mathbb{E}_x[f(S_{ab})]= \sum_{j=0}^{N-1} \mathbf{P}_{ab,j}^-(x)
\,(\mathbf{\Delta}^{\!-})^jf(a)
+\sum_{j=0}^{N-1} \mathbf{P}_{ab,j}^+(x) \,(\mathbf{\Delta}^{\!+})^jf(b)
\end{equation}
where $\mathbf{P}_{ab,j}^-$ and $\mathbf{P}_{ab,j}^+$, $0\le j\le N-1$,
are polynomials of degree not
greater than $2N-1$ characterized, for any $k\in\{0,\dots,N-1\}$, by
\begin{equation}\label{bound-cond}
(\mathbf{\Delta}^{\!-})^k\mathbf{P}_{ab,j}^-(a)=(\mathbf{\Delta}^{\!+})^k
\mathbf{P}_{ab,j}^+(b)=\delta_{jk},\quad
(\mathbf{\Delta}^{\!-})^k\mathbf{P}_{ab,j}^+(a)=(\mathbf{\Delta}^{\!+})^k
\mathbf{P}_{ab,j}^-(b)=0.
\end{equation}
\end{theorem}
%
\begin{proof}
By setting $\mathbf{P}_{ab,j}^-(x)=I_{a-x,b-x,j}^-$ and $\mathbf{P}_{ab,j}^+(x)=I_{a-x,b-x,j}^+$,
(\ref{exp-f-Sab-comb}) immediately yields (\ref{exp-f(x)-Sab}).
By observing that $\mathbf{P}_{ab,j}^-(x)=(-1)^j\mathbf{P}_{ab,j}^+(a+b-x)$, it is enough
to work with, e.g., $\mathbf{P}_{ab,j}^+$.
Coming back to the proof of Theorem~\ref{theo-exp-fSab} and appealing
to Theorem~\ref{theo-dist-Sab}, we write that
\begin{align}
\mathbf{P}_{ab,j}^+(x)
&
=\mathbb{E}_x\!\left[\ind_{\{S_{ab}\ge b\}}\binom{S_{ab}-b}{j}\!\right]
=\sum_{m=j}^{N-1} \binom{m}{j}\mathbb{P}_x\{S_{ab}=b+m\}
\nonumber\\
&
=\frac{1}{(N-1)!N!} \sum_{m=j}^{N-1} (-1)^m\left[\binom{m}{j}\!
\binom{N-1}{m}/\binom{m+b-a+N-1}{N}\!\right] \tilde{K}_m(x)
\label{formula-P+}
\end{align}
where, for any $m\in\{0,\dots,N-1\}$,
\[
\tilde{K}_m(x)=\left(\prod_{k=0}^{N-1} [(x-a+k)(b-x+k)]\right)\!/(b-x+m)
=\prod_{k=0}^{N-1} (x-a+k)\prod_{0\le k\le N-1\atop k\neq m} (b-x+k).
\]
The expression $\tilde{K}_m(x)$ defines a polynomial of the variable $x$
of degree $2N-1$, so $\mathbf{P}_{ab,j}^+$ is a polynomial of degree not greater
that $2N-1$.

It is obvious that $\tilde{K}_m(a-\ell)=0$ for $\ell,m\in\{0,\dots,N-1\}$.
Then, $\mathbf{P}_{ab,j}^+(a-\ell)=0$ which implies that
$(\mathbf{\Delta}^{\!-})^k\mathbf{P}_{ab,j}^+(a)=0$
for any $k\in\{0,\dots,N-1\}$.
Now, let us evaluate $\mathbf{P}_{ab,j}^+(b+\ell)$ for $\ell\in\{0,\dots,N-1\}$.
We plainly have that $\tilde{K}_m(b+\ell)=0$ for $\ell\neq m$ and that
\[
\tilde{K}_{\ell}(b+\ell)=(-1)^{\ell}(N-1)!N!\binom{\ell+b-a+N-1}{N}/\binom{N-1}{\ell}.
\]
By putting this into~(\ref{formula-P+}), we get that
\[
\mathbf{P}_{ab,j}^+(b+\ell)=\binom{\ell}{j}\ind_{\{j\le\ell\}}.
\]
Next, we obtain, for any $k\in\{0,\dots,N-1\}$, that
\begin{align*}
(\mathbf{\Delta}^{\!+})^k\mathbf{P}_{ab,j}^+(b)
&
=\sum_{\ell=0}^k (-1)^{k+\ell} \binom{k}{\ell}\mathbf{P}_{ab,j}^+(b+\ell)
=\sum_{\ell=j}^k (-1)^{k+\ell} \binom{k}{\ell}\!\binom{\ell}{j}
\\
&
=(-1)^{j+k}\binom{k}{j}\sum_{\ell=0}^{k-j} (-1)^{\ell} \binom{k-j}{\ell}=
\delta_{jk}.
\end{align*}
The proof of Theorem~\ref{theo-exp-fxab} is finished.
\end{proof}
%
\begin{example}
In the case where $N=2$, (\ref{exp-f(x)-Sab}) writes as
\[
\mathbb{E}_x[f(S_{ab})]= \mathbf{P}_{ab,0}^-(x) f(a)+\mathbf{P}_{ab,1}^-(x)\mathbf{\Delta}^{\!-}f(a)
+\mathbf{P}_{ab,0}^+(x) f(b)+\mathbf{P}_{ab,1}^+(x)\mathbf{\Delta}^{\!+}f(b)
\]
with
\[
\mathbf{P}_{ab,0}^-(x)=\frac{(x-b)(x-b-1)(2x-3a+b+2)}{(b-a)(b-a+1)(b-a+2)},\quad
\mathbf{P}_{ab,1}^-(x)=\frac{(x-a)(x-b)(x-b-1)}{(b-a+1)(b-a+2)},
\]
\[
\mathbf{P}_{ab,0}^+(x)=-\frac{(x-a)(x-a+1)(2x+a-3b-2)}{(b-a)(b-a+1)(b-a+2)},\quad
\mathbf{P}_{ab,1}^+(x)=\frac{(x-a)(x-a+1)(x-b)}{(b-a+1)(b-a+2)}.
\]
\end{example}
%
Below, we state a strong pseudo-Markov property related to time $\sigma_{ab}$.
%
\begin{theorem}
We have, for any function $f$ defined on $\mathbb{Z}$ and any $n\in\mathbb{N}$,
that
\begin{equation}\label{Markov-ab}
\mathbb{E}_x\!\big[f\big(S_{\sigma_{ab}+n}\big)\big]=
\sum_{j=0}^{N-1}\mathbf{P}_{ab,j}^-(x)\,(\mathbf{\Delta}^{\!-})^j \mathbb{E}_a[f(S_n)]
+\sum_{j=0}^{N-1}\mathbf{P}_{ab,j}^+(x)\,(\mathbf{\Delta}^{\!+})^j \mathbb{E}_b[f(S_n)].
\end{equation}
In (\ref{Markov-ab}), the operators $(\mathbf{\Delta}^{\!-})^j$ and
$(\mathbf{\Delta}^{\!+})^j$ act on the variables $a$ and $b$.
\end{theorem}
%
\begin{proof}
Formula~(\ref{Markov-ab}) can be proved exactly in the same way as~(\ref{Markov-b}):
by setting $g(x)=\mathbb{E}_x[f(S_n)]$, we have that
$\mathbb{E}_x\!\big[f\big(S_{\sigma_{ab}+n}\big)\big]=\mathbb{E}_x\big[g(S_{ab})\big].$
This proves~(\ref{Markov-ab}) thanks to~(\ref{exp-f(x)-Sab}).
\end{proof}
%
\begin{example}
Below, we display the form of (\ref{Markov-ab}) for the particular values
$1,2$ of $N$.
\begin{itemize}
\item
For $N=1$, (\ref{Markov-b}) reads as
\[
\mathbb{E}_x\!\big[f\big(S_{\sigma_{ab}+n}\big)\big]
=\frac{b-x}{b-a}\,\mathbb{E}_a[f(S_n)]+\frac{x-a}{b-a}\,\mathbb{E}_b[f(S_n)]
\]
which is of course well-known! This is the strong Markov property for the
ordinary random walk.

\item
For $N=2$, (\ref{Markov-b}) reads as
\begin{align*}
\mathbb{E}_x\!\big[f\big(S_{\sigma_{ab}+n}\big)\big]
&
=\frac{(x-b)(x-b-1)(2x-3a+b+2)}{(b-a)(b-a+1)(b-a+2)} \,\mathbb{E}_a[f(S_n)]
\\
&
\hphantom{=\;}
+\frac{(x-a)(x-b)(x-b-1)}{(b-a+1)(b-a+2)} \,\mathbf{\Delta}^{\!-} \mathbb{E}_a[f(S_n)]
\\
&
\hphantom{=\;}
-\frac{(x-a)(x-a+1)(2x+a-3b-2)}{(b-a)(b-a+1)(b-a+2)} \,\mathbb{E}_b[f(S_n)]
\\
&
\hphantom{=\;}
+\frac{(x-a)(x-a+1)(x-b)}{(b-a+1)(b-a+2)} \,\mathbf{\Delta}^{\!+} \mathbb{E}_b[f(S_n)]
\\
&
=\frac{(x-a+1)(x-b)(x-b-1)}{(b-a)(b-a+1)} \,\mathbb{E}_a[f(S_n)]
\\
&
\hphantom{=\;}
-\frac{(x-a)(x-b)(x-b-1)}{(b-a+1)(b-a+2)} \,\mathbb{E}_{a-1}[f(S_n)]
\\
&
\hphantom{=\;}
-\frac{(x-a)(x-a+1)(x-b-1)}{(b-a)(b-a+1)} \,\mathbb{E}_b[f(S_n)]
\\
&
\hphantom{=\;}
+\frac{(x-a)(x-a+1)(x-b)}{(b-a+1)(b-a+2)} \,\mathbb{E}_{b+1}[f(S_n)]
\end{align*}

\end{itemize}
\end{example}

Now, we consider the discrete Laplacian $\mathbf{\Delta}
=\mathbf{\Delta}^{\!+}\circ\mathbf{\Delta}^{\!-}
=\mathbf{\Delta}^{\!-}\circ\mathbf{\Delta}^{\!+}$.
It is explicitly defined by $\mathbf{\Delta} f(i)=f(i+1)-2f(i)+f(i-1)$.
Let us introduce the iterated Laplacian $\mathbf{\Delta}^{\!N}
=(\mathbf{\Delta}^{\!+})^N\circ(\mathbf{\Delta}^{\!-})^N
=(\mathbf{\Delta}^{\!-})^N\circ(\mathbf{\Delta}^{\!+})^N$.
We compute $\mathbf{\Delta}^{\!N}\!f(i)$ for any function $f$
and any $i\in\mathbb{Z}$:
\begin{align*}
\mathbf{\Delta}^{\!N}\!f(i)
&=(\mathbf{\Delta}^{\!-})^N\!\left(\vphantom{\sum_{k=0}^N}\right.\!
\sum_{j=0}^N (-1)^{j+N} \binom{N}{j} f(.+j)\!\!\left.\vphantom{\sum_{k=0}^N}\right)\!(i)
\\
&=\sum_{j=0}^N (-1)^{j+N} \binom{N}{j} \!\!\left(\sum_{k=0}^N (-1)^k
\binom{N}{k} f(i+j-k)\right)
\\
&=(-1)^N\sum_{0\le j,k\le N} (-1)^{j-k} \binom{N}{j} \!\binom{N}{k}f(i+j-k)
\\
&=(-1)^N\sum_{\ell=-N}^N (-1)^{\ell} \!\left(\vphantom{\sum_{k=0}^N}\right.\!
\sum_{k=(-\ell)\vee 0}^{(N-\ell)\wedge N}\binom{N}{k}\! \binom{N}{k+\ell}
\!\!\left.\vphantom{\sum_{k=0}^N}\right)\! f(i+\ell).
\end{align*}
By using the elementary identity
$\sum_{k=(\ell-q)\vee 0}^{p\wedge \ell} \binom{p}{k}\! \binom{q}{\ell-k}
=\binom{p+q}{\ell}$, we get that
\[
\sum_{k=(-\ell)\vee 0}^{(N-\ell)\wedge N}\binom{N}{k}\! \binom{N}{k+\ell}
=\sum_{k=(-\ell)\vee 0}^{(N-\ell)\wedge N}\binom{N}{k}\! \binom{N}{N-k-\ell}
=\binom{2N}{\ell+N}.
\]
As a result, we obtain the expression of $\mathbf{\Delta}^{\!N}\!f(i)$ announced in
the introduction, namely
\[
\mathbf{\Delta}^{\!N}\!f(i)=\sum_{\ell=-N}^N (-1)^{\ell+N} \binom{2N}{\ell+N} f(i+\ell).
\]
%
\begin{example}
Fix a nonnegative integer $j$ and put $f_j(i)=(i)_j$ for any $i\in\mathbb{Z}$.
It is plain that, if $k\le j$,
$(\mathbf{\Delta}^{\!+})^k f_j(i)= (j)_k \,f_{j-k}(i),\,
(\mathbf{\Delta}^{\!-})^k f_j(i)= (j)_k \,f_{j-k}(i-1)
$
and if $k>j$, $(\mathbf{\Delta}^{\!+})^k f_j(i)=(\mathbf{\Delta}^{\!-})^k f_j(i)=0$.
Therefore, if $2k\le j$,
$
\mathbf{\Delta}^{\!k} f_j(i)=(j)_k(j-k)_k \,f_{j-2k}(i-1)
$
and if $2k>j$, $\mathbf{\Delta}^{\!k} f_j(i)=0$.
By using a linear algebra argument, we deduce that
$\mathbf{\Delta}^{\!N}\!P=0$ for any polynomial $P$ of degree not greater that $2N-1$.
As a byproduct,
\begin{equation}\label{laplacian-P}
\mathbf{\Delta}^{\!N}\!P_{ab,j}^+=\mathbf{\Delta}^{\!N}\!P_{ab,j}^-=0.
\end{equation}
\end{example}
%
Now, the main link between time $\sigma_{ab}$ and finite-difference
equations is the following one.
%
\begin{theorem}
Let $\varphi$ be a function defined on $\mathcal{E}$.
The function $\Phi$ defined on $\mathbb{Z}$ by $\Phi(x)
=\mathbb{E}_x[\varphi(S_{ab})]$ is the
solution to the discrete Lauricella's problem
\begin{equation}\label{syst-lauri}
\begin{cases}
\mathbf{\Delta}^{\!N}\Phi(x)=0\text{ for } x\in\mathbb{Z},
\\
(\mathbf{\Delta}^{\!-})^k\Phi(a)=(\mathbf{\Delta}^{\!-})^k
\varphi(a)\text{ for } j\in\{0,\dots,N-1\},
\\
(\mathbf{\Delta}^{\!+})^k\Phi(b)=(\mathbf{\Delta}^{\!+})^k
\varphi(b)\text{ for } j\in\{0,\dots,N-1\}.
\end{cases}
\end{equation}
\end{theorem}
%
\begin{proof}
By (\ref{exp-f(x)-Sab}) we write that
\[
\Phi(x)= \sum_{j=0}^{N-1} \mathbf{P}_{ab,j}^-(x) \,(\mathbf{\Delta}^{\!-})^j\varphi(a)
+\sum_{j=0}^{N-1} \mathbf{P}_{ab,j}^+(x) \,(\mathbf{\Delta}^{\!+})^j\varphi(b).
\]
With this representation at hand, identities~(\ref{bound-cond}) and
(\ref{laplacian-P}) immediately yield equations~(\ref{syst-lauri}).
\end{proof}

\section{Joint pseudo-distribution of $\big(\tau_{ab},X_{ab}\big)$}

As in Section~\ref{section-tau-b}, we choose for the family
$((X_t^{\varepsilon})_{t\ge 0})_{\varepsilon>0}$
the pseudo-processes defined, for any $\varepsilon>0$, by
\[
X_t^{\varepsilon}=\varepsilon S_{\lfloor t/\varepsilon^{2N}\rfloor}, \quad t\ge 0,
\]
and for the pseudo-process $(X_t)_{t\ge 0}$ the pseudo-Brownian motion.
In Definition~\ref{def2}, we choose for $I$ the interval $(a,b)$; then
$\tau_I^{\varepsilon}=\tau_{ab}^{\varepsilon}$, $X_I^{\varepsilon}=X_{ab}^{\varepsilon}$
and $\tau_I^{\vphantom{\varepsilon}}=\tau_{ab}$, $X_I^{\vphantom{\varepsilon}}=X_{ab}$.
Set $a_\varepsilon=\lfloor a/\varepsilon\rfloor$ and $b_{\varepsilon}=\lceil b/\varepsilon\rceil$ where
$\lfloor\,.\,\rfloor$ and $\lceil\,.\,\rceil$ respectively stand for
the usual floor and ceiling functions.
We have $\tau_{ab}^{\varepsilon}=\varepsilon^{2N}\sigma_{a_\varepsilon,b_{\varepsilon}}$ and
$X_{ab}^{\varepsilon}=\varepsilon S_{a_\varepsilon,b_{\varepsilon}}$.
%
\begin{theorem}\label{theo-conv2}
The following convergence holds:
\[
\big(\tau_{ab}^{\varepsilon},X_{ab}^{\varepsilon}\big) \underset{\varepsilon\to 0^+}{\longrightarrow}
\big(\tau_{ab},X_{ab}\big)
\]
where, for any $\lambda>0$ and any $\mu\in\mathbb{R}$,
\begin{align*}
\mathbb{E}\Big(\mathrm{e}^{-\lambda\tau_{ab}+\mathrm{i} \mu X_{ab}}\ind_{\{\tau_{ab}<+\infty\}}\Big)
&=\mathrm{e}^{\mathrm{i} \mu a} \sum_{k=1}^N \frac{\mathbf{D}_k^-(\lambda,\mu)}{\mathbf{D}(\lambda)}
+\mathrm{e}^{\mathrm{i} \mu b} \sum_{k=1}^N \frac{\mathbf{D}_k^+(\lambda,\mu)}{\mathbf{D}(\lambda)}.
\end{align*}
In the foregoing formula, $\mathbf{D}(\lambda)$, $\mathbf{D}_k^-(\lambda,\mu)$,
$\mathbf{D}_k^+(\lambda,\mu)$ are the respective determinants
\[
\left|\begin{array}{@{\hspace{0em}}c@{\hspace{.9em}}c@{\hspace{.9em}}c@{\hspace{.9em}}c@{\hspace{.9em}}c@{\hspace{.9em}}c@{\hspace{.9em}}c@{\hspace{.9em}}c@{\hspace{0em}}}
1      & \varphi_1    & \dots & \varphi_1^{N-1}    &\mathrm{e}^{-\varphi_1\!\!\sqrt[2N\!\!]{\lambda/c}\,(b-a)}    & \mathrm{e}^{-\varphi_1\!\!\sqrt[2N\!\!]{\lambda/c}\,(b-a)}\,\varphi_1       & \dots & \mathrm{e}^{-\varphi_1\!\!\sqrt[2N\!\!]{\lambda/c}\,(b-a)}\,\varphi_1^{N-1}      \\
\vdots & \vdots       &       & \vdots             &\vdots                                                        & \vdots                                                                      &       & \vdots                                                                           \\[1ex]
1      & \varphi_{2N} & \dots & \varphi_{2N}^{N-1} &\mathrm{e}^{-\varphi_{2N}\!\!\sqrt[2N\!\!]{\lambda/c}\,(b-a)} & \mathrm{e}^{-\varphi_{2N}\!\!\sqrt[2N\!\!]{\lambda/c}\,(b-a)}\,\varphi_{2N} & \dots & \mathrm{e}^{-\varphi_{2N}\!\!\sqrt[2N\!\!]{\lambda/c}\,(b-a)}\,\varphi_{2N}^{N-1}
\end{array}\right|\!,
\]

\vspace{.3\baselineskip}
\[
\left|\begin{array}{@{\hspace{0em}}c@{\hspace{.6em}}c@{\hspace{.6em}}c@{\hspace{.6em}}c@{\hspace{.6em}}c@{\hspace{.6em}}c@{\hspace{.6em}}c@{\hspace{.6em}}c@{\hspace{0em}}}
1          & \varphi_1        & \dots & \varphi_1^{N-1}        & \mathrm{e}^{-\varphi_1\!\!\sqrt[2N\!\!]{\lambda/c}\,(b-a)}     & \mathrm{e}^{-\varphi_1\!\!\sqrt[2N\!\!]{\lambda/c}\,(b-a)}\,\varphi_1         & \dots & \mathrm{e}^{-\varphi_1\!\!\sqrt[2N\!\!]{\lambda/c}\,(b-a)}\,\varphi_1^{N-1}         \\
\vdots     & \vdots           &       & \vdots                 & \vdots                                                         & \vdots                                                                        &       & \vdots                                                                              \\[1ex]
1          & \varphi_{k-1}    & \dots & \varphi_{k-1}^{N-1}    & \mathrm{e}^{-\varphi_{k-1}\!\!\sqrt[2N\!\!]{\lambda/c}\,(b-a)} & \mathrm{e}^{-\varphi_{k-1}\!\!\sqrt[2N\!\!]{\lambda/c}\,(b-a)}\,\varphi_{k-1} & \dots & \mathrm{e}^{-\varphi_{k-1}\!\!\sqrt[2N\!\!]{\lambda/c}\,(b-a)}\,\varphi_{k-1}^{N-1} \\[.5ex]
1          & \delta           & \dots & \delta^{N-1}           & 0                                                              & 0                                                                             & \dots & 0                                                                                   \\
1          & \varphi_{k+1}    & \dots & \varphi_{k+1}^{N-1}    & \mathrm{e}^{-\varphi_{k+1}\!\!\sqrt[2N\!\!]{\lambda/c}\,(b-a)} & \mathrm{e}^{-\varphi_{k+1}\!\!\sqrt[2N\!\!]{\lambda/c}\,(b-a)}\,\varphi_{k+1} & \dots & \mathrm{e}^{-\varphi_{k+1}\!\!\sqrt[2N\!\!]{\lambda/c}\,(b-a)}\,\varphi_{k+1}^{N-1} \\
\vdots     & \vdots           &       & \vdots                 & \vdots                                                         & \vdots                                                                        &       & \vdots                                                                              \\[1ex]
1          & \varphi_{2N}     & \dots & \varphi_{2N}^{N-1}     & \mathrm{e}^{-\varphi_{2N}\!\!\sqrt[2N\!\!]{\lambda/c}\,(b-a)}  & \mathrm{e}^{-\varphi_{2N}\!\!\sqrt[2N\!\!]{\lambda/c}\,(b-a)}\,\varphi_{2N}   & \dots & \mathrm{e}^{-\varphi_{2N}\!\!\sqrt[2N\!\!]{\lambda/c}\,(b-a)}\,\varphi_{2N}^{N-1}
\end{array}\right|\!,
\]

\vspace{.3\baselineskip}
\[
\left|\begin{array}{@{\hspace{0em}}c@{\hspace{.6em}}c@{\hspace{.6em}}c@{\hspace{.6em}}c@{\hspace{.6em}}c@{\hspace{.6em}}c@{\hspace{.6em}}c@{\hspace{.6em}}c@{\hspace{0em}}}
1          & \varphi_1        & \dots & \varphi_1^{N-1}        & \mathrm{e}^{-\varphi_1\!\!\sqrt[2N\!\!]{\lambda/c}\,(b-a)}     & \mathrm{e}^{-\varphi_1\!\!\sqrt[2N\!\!]{\lambda/c}\,(b-a)}\,\varphi_1         & \dots & \mathrm{e}^{-\varphi_1\!\!\sqrt[2N\!\!]{\lambda/c}\,(b-a)}\,\varphi_1^{N-1}         \\
\vdots     & \vdots           &       & \vdots                 & \vdots                                                         & \vdots                                                                        &       & \vdots                                                                              \\[1ex]
1          & \varphi_{k-1}    & \dots & \varphi_{k-1}^{N-1}    & \mathrm{e}^{-\varphi_{k-1}\!\!\sqrt[2N\!\!]{\lambda/c}\,(b-a)} & \mathrm{e}^{-\varphi_{k-1}\!\!\sqrt[2N\!\!]{\lambda/c}\,(b-a)}\,\varphi_{k-1} & \dots & \mathrm{e}^{-\varphi_{k-1}\!\!\sqrt[2N\!\!]{\lambda/c}\,(b-a)}\,\varphi_{k-1}^{N-1} \\[.5ex]
0          & 0                & \dots & 0                      & 1                                                              & \delta                                                                        & \dots & \delta^{N-1}                                                                        \\
1          & \varphi_{k+1}    & \dots & \varphi_{k+1}^{N-1}    & \mathrm{e}^{-\varphi_{k+1}\!\!\sqrt[2N\!\!]{\lambda/c}\,(b-a)} & \mathrm{e}^{-\varphi_{k+1}\!\!\sqrt[2N\!\!]{\lambda/c}\,(b-a)}\,\varphi_{k+1} & \dots & \mathrm{e}^{-\varphi_{k+1}\!\!\sqrt[2N\!\!]{\lambda/c}\,(b-a)}\,\varphi_{k+1}^{N-1} \\
\vdots     & \vdots           &       & \vdots                 & \vdots                                                         & \vdots                                                                        &       & \vdots                                                                              \\[1ex]
1          & \varphi_{2N}     & \dots & \varphi_{2N}^{N-1}     & \mathrm{e}^{-\varphi_{2N}\!\!\sqrt[2N\!\!]{\lambda/c}\,(b-a)}  & \mathrm{e}^{-\varphi_{2N}\!\!\sqrt[2N\!\!]{\lambda/c}\,(b-a)}\,\varphi_{2N}   & \dots & \mathrm{e}^{-\varphi_{2N}\!\!\sqrt[2N\!\!]{\lambda/c}\,(b-a)}\,\varphi_{2N}^{N-1}
\end{array}\right|\!.
\]
In the two last determinants, we have put $\delta=-\frac{\mathrm{i} \mu}{\sqrt[2N\!\!]{\lambda/c}}$.
\end{theorem}
%
In Theorem~\ref{theo-conv2}, we obtain the joint pseudo-distribution of
$(\tau_{ab},X_{ab})$ characterized by its Laplace-Fourier transform. This is a
new result for pseudo-Brownian motion that we shall develop in a forthcoming
paper.
%
\begin{proof}
By Definition~\ref{def2} and Theorem~\ref{theo-sigma-ab2}, we have that
\begin{align}
\mathbb{E}\Big(\mathrm{e}^{-\lambda\tau_{ab}+\mathrm{i} \mu X_{ab}}\ind_{\{\tau_{ab}<+\infty\}}\Big)
&
=\lim_{\varepsilon\to 0^+} \mathbb{E}\Big(\mathrm{e}^{-\lambda\tau_{ab}^{\varepsilon}
+\mathrm{i} \mu X_{ab}^{\varepsilon}}
\ind_{\{\tau_{ab}^{\varepsilon}<+\infty\}}\Big)
\nonumber\\
&
=\lim_{\varepsilon\to 0^+} \mathbb{E}\Big(\mathrm{e}^{-\lambda\varepsilon^{2N}
\sigma_{a_{\varepsilon},b_{\varepsilon}}+\mathrm{i} \mu \varepsilon S_{a_{\varepsilon},b_{\varepsilon}}}
\ind_{\{\sigma_{a_{\varepsilon},b_{\varepsilon}}<+\infty\}}\Big)
\nonumber\\
&
=\lim_{\varepsilon\to 0^+} \sum_{k=1}^{2N} \mathbf{\tilde{L}}_k
\Big(\mathrm{e}^{-\lambda\varepsilon^{2N}},\mathrm{e}^{\mathrm{i} \mu\varepsilon}\Big)
\,(\mathrm{e}^{\mathrm{i} \mu\varepsilon}v_k(\lambda,\varepsilon))^{a_{\varepsilon}-N}
\nonumber\\
&
=\mathrm{e}^{\mathrm{i} \mu a}\sum_{k=1}^{2N} \lim_{\varepsilon\to 0^+}
\mathbf{\tilde{L}}_k\Big(\mathrm{e}^{-\lambda\varepsilon^{2N}},
\mathrm{e}^{\mathrm{i} \mu\varepsilon}\Big)
v_k(\lambda,\varepsilon)^{a_{\varepsilon}}.
\label{limit-FT-ter}
\end{align}
Recall that $\mathbf{\tilde{L}}_k(z,\zeta)=D_k(z,\zeta)/D(z)$
and that the quantities $D$ and $D_k$ are expressed by means of the determinant
\begin{align*}
\lqn{W(u_1,\dots,u_{2N})}
&=\begin{vmatrix}
1      & u_1    & u_1^2    & \dots & u_1^{N-1}    & u_1^{b-a+N-1}    & u_1^{b-a+N}    & u_1^{b-a+N+1}    & \dots & u_1^{b-a+2N-2}   \\
\vdots & \vdots & \vdots   &       & \vdots       & \vdots           & \vdots         & \vdots           &       & \vdots           \\[1ex]
1      & u_{2N} & u_{2N}^2 & \dots & u_{2N}^{N-1} & u_{2N}^{b-a+N-1} & u_{2N}^{b-a+N} & u_{2N}^{b-a+N+1} & \dots & u_{2N}^{b-a+2N-2}
\end{vmatrix}\!.
\end{align*}
By replacing the columns labeled as $C_j$, $1\le j\le 2N$, by the linear combinations
$C_j'=\sum_{k=0}^j(-1)^{j+k}$ $\binom{j}{k}C_k$
if $1\le j\le N$, and $C_j'=\sum_{k=N}^j (-1)^{j+k}\binom{j}{k}C_k$ if $N+1\le j\le 2N$,
the foregoing determinant remains invariant and can be rewritten as
\begin{align*}
&\left|\begin{matrix}
1      & u_1-1    & (u_1-1)^2    & \dots & (u_1-1)^{N-1}    \\
\vdots & \vdots   & \vdots       &       & \vdots           \\[1ex]
1      & u_{2N}-1 & (u_{2N}-1)^2 & \dots & (u_{2N}-1)^{N-1}
\end{matrix}\right.
\\[2ex]
&
\hphantom{\;}\left.\begin{matrix}
& u_1^{b-a+N-1}    & u_1^{b-a+N-1}(u-1)    & u_1^{b-a+N-1}(u-1)^2    & \dots & u_1^{b-a+2N-2}(u_1-1)^{N-1}   \\
& \vdots           & \vdots                & \vdots                  &       & \vdots                        \\[1ex]
& u_{2N}^{b-a+N-1} & u_{2N}^{b-a+N-1}(u-1) & u_{2N}^{b-a+N-1}(u-1)^2 & \dots & u_{2N}^{b-a+2N-2}(u_1-1)^{N-1}
\end{matrix}\right|\!.
\end{align*}
Then, by replacing the $u_j$'s by $u_j(\lambda,\varepsilon)$, $b-a$
by $b_{\varepsilon}-a_{\varepsilon}$ and by using the asymptotics
$(u_j(\lambda,\varepsilon)-1)^k\underset{\varepsilon\to 0^+}{\sim}
\left(-\varphi_j\!\!\sqrt[2N\!\!]{\lambda/c}\,\right)^{\!k}\varepsilon^k$
and $u_j(\lambda,\varepsilon)^{b_{\varepsilon}-a_{\varepsilon}+N-1}
\underset{\varepsilon\to 0^+}{\longrightarrow}
\mathrm{e}^{-\varphi_j\!\!\sqrt[2N\!\!]{\lambda/c}\,(b-a)}$
coming from~(\ref{asymptotic-v}), we get that
\begin{align}
\lqn{D\Big(\mathrm{e}^{-\lambda\varepsilon^{2N}}\Big)=W\Big(u_1(\lambda,\varepsilon),
\dots,u_{2N}(\lambda,\varepsilon)\Big)}
&
\underset{\varepsilon\to 0^+}{\sim}\left|\begin{matrix}
1      & \left(-\varphi_1\!\!\sqrt[2N\!\!]{\lambda/c}\,\right)\varepsilon    & \dots & \left(-\varphi_1\!\!\sqrt[2N\!\!]{\lambda/c}\,\right)^{\!N-1}\varepsilon^{N-1}    & \mathrm{e}^{-\varphi_1\!\!\sqrt[2N\!\!]{\lambda/c}\,(b-a)}    \\
\vdots & \vdots                                                           &       & \vdots                                                                       & \vdots                                                        \\[1ex]
1      & \left(-\varphi_{2N}\!\!\sqrt[2N\!\!]{\lambda/c}\,\right)\varepsilon & \dots & \left(-\varphi_{2N}\!\!\sqrt[2N\!\!]{\lambda/c}\,\right)^{\!N-1}\varepsilon^{N-1} & \mathrm{e}^{-\varphi_{2N}\!\!\sqrt[2N\!\!]{\lambda/c}\,(b-a)}
\end{matrix}\right.
\nonumber\\[2ex]
&
\hphantom{=\;}\left.\begin{matrix}
& \left(-\varphi_1\!\!\sqrt[2N\!\!]{\lambda/c}\,\right)\mathrm{e}^{-\varphi_1\!\!\sqrt[2N\!\!]{\lambda/c}\,(b-a)}\,\varepsilon       & \dots & \left(-\varphi_1\!\!\sqrt[2N\!\!]{\lambda/c}\,\right)^{\!N-1}\mathrm{e}^{-\varphi_1\!\!\sqrt[2N\!\!]{\lambda/c}\,(b-a)}\,\varepsilon^{N-1}    \\
& \vdots                                                                                                                           &       & \vdots                                                                                                                                 \\[1ex]
& \left(-\varphi_{2N}\!\!\sqrt[2N\!\!]{\lambda/c}\,\right) \mathrm{e}^{-\varphi_{2N}\!\!\sqrt[2N\!\!]{\lambda/c}\,(b-a)}\,\varepsilon & \dots & \left(-\varphi_{2N}\!\!\sqrt[2N\!\!]{\lambda/c}\,\right)^{\!N-1}\mathrm{e}^{-\varphi_{2N}\!\!\sqrt[2N\!\!]{\lambda/c}\,(b-a)}\,\varepsilon^{N-1}
\end{matrix}\right|
\nonumber\\[2ex]
&
=(\lambda/c)^{(N-1)/2}\varepsilon^{N(N-1)} \mathbf{D}(\lambda).
\label{asymptotic-inter1}
\end{align}
Similarly, by using the elementary asymptotics $\mathrm{e}^{\mathrm{i} \mu\varepsilon}-
1\underset{\varepsilon\to 0^+}{\sim} \mathrm{i}\mu\varepsilon$
and $\mathrm{e}^{\mathrm{i} \mu\varepsilon(b_{\varepsilon}-a_{\varepsilon}+N-1)}
\underset{\varepsilon\to 0^+}{\longrightarrow} \mathrm{e}^{\mathrm{i} \mu(b-a)}$, we obtain that
\begin{align}
\mathrm{e}^{\mathrm{i} \mu a} D_k\Big(\mathrm{e}^{-\lambda\varepsilon^{2N}},
\mathrm{e}^{\mathrm{i} \mu\varepsilon}\Big)
&
\;\;\,=\;\;\,\mathrm{e}^{\mathrm{i} \mu a} W\Big(u_1(\lambda,\varepsilon),
\dots,u_{k-1}(\lambda,\varepsilon),
\mathrm{e}^{\mathrm{i} \mu\varepsilon},u_{k+1}(\lambda,\varepsilon),
\dots,u_{2N}(\lambda,\varepsilon)\Big)
\nonumber\\
&
\underset{\varepsilon\to 0^+}{\sim} (\lambda/c)^{(N-1)/2}\varepsilon^{N(N-1)}
\mathbf{D}_k(\lambda,\mu)
\label{asymptotic-inter2}
\end{align}
where $\mathbf{D}_k(\lambda,\mu)$ denotes the determinant
\[
\left|\begin{array}{@{\hspace{0em}}c@{\hspace{.2em}}c@{\hspace{.2em}}c@{\hspace{.2em}}c@{\hspace{0em}}c@{\hspace{.4em}}c@{\hspace{.2em}}c@{\hspace{.2em}}c@{\hspace{0em}}}
1          & \varphi_1               & \dots & \varphi_1^{N-1}               & \mathrm{e}^{-\varphi_1\!\!\sqrt[2N\!\!]{\lambda/c}\,(b-a)}     & \mathrm{e}^{-\varphi_1\!\!\sqrt[2N\!\!]{\lambda/c}\,(b-a)}\,\varphi_1         & \dots & \mathrm{e}^{-\varphi_1\!\!\sqrt[2N\!\!]{\lambda/c}\,(b-a)}\,\varphi_1^{N-1}         \\
\vdots     & \vdots                  &       & \vdots                        & \vdots                                                         & \vdots                                                                        &       & \vdots                                                                              \\[1ex]
1          & \varphi_{k-1}           & \dots & \varphi_{k-1}^{N-1}           & \mathrm{e}^{-\varphi_{k-1}\!\!\sqrt[2N\!\!]{\lambda/c}\,(b-a)} & \mathrm{e}^{-\varphi_{k-1}\!\!\sqrt[2N\!\!]{\lambda/c}\,(b-a)}\,\varphi_{k-1} & \dots & \mathrm{e}^{-\varphi_{k-1}\!\!\sqrt[2N\!\!]{\lambda/c}\,(b-a)}\,\varphi_{k-1}^{N-1} \\[.5ex]
\mathrm{e}^{\mathrm{i} \mu a}        & \mathrm{e}^{\mathrm{i} \mu a}\delta   & \dots & \mathrm{e}^{\mathrm{i} \mu a}\delta^{N-1}   & \mathrm{e}^{\mathrm{i} \mu b}                                  & \mathrm{e}^{\mathrm{i} \mu b}\delta                                      & \dots & \mathrm{e}^{\mathrm{i} \mu b}\delta^{N-1}                                      \\
1          & \varphi_{k+1}           & \dots & \varphi_{k+1}^{N-1}           & \mathrm{e}^{-\varphi_{k+1}\!\!\sqrt[2N\!\!]{\lambda/c}\,(b-a)} & \mathrm{e}^{-\varphi_{k+1}\!\!\sqrt[2N\!\!]{\lambda/c}\,(b-a)}\,\varphi_{k+1} & \dots & \mathrm{e}^{-\varphi_{k+1}\!\!\sqrt[2N\!\!]{\lambda/c}\,(b-a)}\,\varphi_{k+1}^{N-1} \\
\vdots     & \vdots                  &       & \vdots                        & \vdots                                                         & \vdots                                                                        &       & \vdots                                                                              \\[1ex]
1          & \varphi_{2N}            & \dots & \varphi_{2N}^{N-1}            & \mathrm{e}^{-\varphi_{2N}\!\!\sqrt[2N\!\!]{\lambda/c}\,(b-a)}  & \mathrm{e}^{-\varphi_{2N}\!\!\sqrt[2N\!\!]{\lambda/c}\,(b-a)}\,\varphi_{2N}   & \dots & \mathrm{e}^{-\varphi_{2N}\!\!\sqrt[2N\!\!]{\lambda/c}\,(b-a)}\,\varphi_{2N}^{N-1}
\end{array}\right|\!.
\]
By putting~(\ref{asymptotic-inter1}) and~(\ref{asymptotic-inter2}) into
(\ref{limit-FT-ter}), we derive that
\begin{align*}
\mathbb{E}\Big(\mathrm{e}^{-\lambda\tau_{ab}+\mathrm{i} \mu X_{ab}}\ind_{\{\tau_{ab}<+\infty\}}\Big)
&=\sum_{k=1}^{2N} \frac{\mathbf{D}_k(\lambda,\mu)}{\mathbf{D}(\lambda)}.
\end{align*}
It is plain that $\mathbf{D}_k(\lambda,\mu)=\mathrm{e}^{\mathrm{i} \mu a}
\mathbf{D}_k^-(\lambda,\mu)+\mathrm{e}^{\mathrm{i} \mu b}\mathbf{D}_k^-(\lambda,\mu)$
which finishes the proof of Theorem~\ref{theo-conv2}.
\end{proof}
%
\begin{theorem}\label{theo-dist-Xab}
The following convergence holds:
\[
X_{ab}^{\varepsilon} \underset{\varepsilon\to 0^+}{\longrightarrow} X_{ab}
\]
where, for any $\mu\in\mathbb{R}$,
\[
\mathbb{E}\Big(\mathrm{e}^{\mathrm{i} \mu X_{ab}}\ind_{\{\tau_{ab}<+\infty\}}\Big)
=\mathrm{e}^{\mathrm{i} \mu a} \sum_{j=0}^{N-1} \mathbf{I}_{ab,j}^- (\mathrm{i} \mu)^j
+\mathrm{e}^{\mathrm{i} \mu b} \sum_{j=0}^{N-1} \mathbf{I}_{ab,j}^+ (\mathrm{i} \mu)^j
\]
with
\begin{align*}
\mathbf{I}_{ab,j}^-
&=\left(\frac{b}{b-a}\right)^{\!\!N}\frac{a^j}{j!}
\sum_{k=0}^{N-j-1}\binom{k+N-1}{k} \!\left(\frac{-a}{b-a}\right)^{\!\!k},
\\
\mathbf{I}_{ab,j}^+
&=\left(\frac{-a}{b-a}\right)^{\!\!N} \frac{(-b)^{\,j}}{j!}
\sum_{k=0}^{N-j-1}\binom{k+N-1}{k} \!\left(\frac{b}{b-a}\right)^{\!\!k}.
\end{align*}
Moreover,
\[
\mathbb{P}\{\tau_{ab}<+\infty\}=1.
\]
\end{theorem}
%
\begin{proof}
By definition~\ref{def2} and by~(\ref{gene-Sab}),
\begin{align}
\lqn{\mathbb{E}\Big(\mathrm{e}^{\mathrm{i} \mu X_{ab}}\ind_{\{\tau_{ab}<+\infty\}}\Big)}
&
=\lim_{\varepsilon\to 0^+} \mathbb{E}\Big(\mathrm{e}^{\mathrm{i} \mu X_{ab}^{\varepsilon}}
\ind_{\{\tau_{ab}^{\varepsilon}<+\infty\}}\Big)
=\lim_{\varepsilon\to 0^+} \mathbb{E}\Big(\mathrm{e}^{\mathrm{i}
\mu \varepsilon S_{a_{\varepsilon},b_{\varepsilon}}}
\ind_{\{\sigma_{a_{\varepsilon},b_{\varepsilon}}<+\infty\}}\Big)
\nonumber
\\
&
=\lim_{\varepsilon\to 0^+} \left(\mathrm{e}^{\mathrm{i} \mu \varepsilon
a_{\varepsilon}} \sum_{j=0}^{N-1}
I_{a_{\varepsilon},b_{\varepsilon},j}^- (1-\mathrm{e}^{-\mathrm{i} \mu \varepsilon})^j
+\mathrm{e}^{\mathrm{i} \mu \varepsilon b_{\varepsilon}} \sum_{j=0}^{N-1}
I_{a_{\varepsilon},b_{\varepsilon},j}^+ (\mathrm{e}^{\mathrm{i} \mu \varepsilon}-1)^j\right)
\nonumber\\
&
\mathrm{e}^{\mathrm{i} \mu a} \lim_{\varepsilon\to 0^+} \sum_{j=0}^{N-1}
I_{a_{\varepsilon},b_{\varepsilon},j}^- (1-\mathrm{e}^{-\mathrm{i} \mu \varepsilon})^j
+
\mathrm{e}^{\mathrm{i} \mu b} \lim_{\varepsilon\to 0^+} \sum_{j=0}^{N-1}
I_{a_{\varepsilon},b_{\varepsilon},j}^+ (\mathrm{e}^{\mathrm{i} \mu \varepsilon}-1)^j.
\label{comb-lin}
\end{align}
Concerning, e.g., the quantity $I_{a_{\varepsilon},b_{\varepsilon},j}^+$, we have that
\begin{align}
I_{a_{\varepsilon},b_{\varepsilon},j}^+
&= \frac{(-1)^j}{(N-1)!^2} \binom{N-1}{j}\prod_{k=0}^{N-1} [(k-a_{\varepsilon})(k+b_{\varepsilon})]
\nonumber\\
&
\hphantom{=\;}\times\int\!\!\!\int_{\mathcal{D}^+}
u^{-a_{\varepsilon}-1}(1-u)^{N-1} v^{j+b_{\varepsilon}-1}(1-v)^{N-j-1}\,\mathrm{d} u\,\mathrm{d} v.
\label{I+-inter}
\end{align}
By performing the change of variables $(u,v)=(x,xy)$ in the above integral
and by expanding $(1-xy)^{N-j-1}$ as
\begin{align*}
(1-xy)^{N-j-1}&=[(1-x)+x(1-y)]^{N-j-1}
\\
&=\sum_{k=0}^{N-j-1} \binom{N-j-1}{k}(1-x)^k x^{N-j-k-1}(1-y)^{N-j-k-1},
\end{align*}
we get that
\begin{align}
\lqn{\int\!\!\!\int_{\mathcal{D}^+} u^{-a_{\varepsilon}-1}(1-u)^{N-1}
v^{j+b_{\varepsilon}-1}(1-v)^{N-j-1} \,\mathrm{d} u\,\mathrm{d} v}
&
=\int_0^1\!\!\!\int_0^1 x^{j+b_{\varepsilon}-a_{\varepsilon}-1}(1-x)^{N-1}
y^{j+b_{\varepsilon}-1} (1-xy)^{N-j-1}\,\mathrm{d} x\,\mathrm{d} y
\nonumber\\
&
=\sum_{k=0}^{N-j-1} \binom{N-j-1}{k}
\int_0^1 x^{b_{\varepsilon}-a_{\varepsilon}+N-k-2}(1-x)^{k+N-1}\,\mathrm{d}x
\int_0^1 y^{j+b_{\varepsilon}-1}(1-y)^{N-j-k-1}\,\mathrm{d}y
\nonumber\\
&
=\sum_{k=0}^{N-j-1} \binom{N-j-1}{k} \,\frac{(b_{\varepsilon}-a_{\varepsilon}
+N-k-2)!(k+N-1)!}{(b_{\varepsilon}-a_{\varepsilon}+2N-2)!}
\,\frac{(j+b_{\varepsilon}-1)!(N-j-k-1)!}{(b_{\varepsilon}+N-k-1)!}
\nonumber\\
&
=\frac{(N-1)!(N-j-1)!(j+b_{\varepsilon}-1)!}{(b_{\varepsilon}-a_{\varepsilon}+2N-2)!}
\sum_{k=0}^{N-j-1} \binom{k+N-1}{k} \,\frac{(b_{\varepsilon}-
a_{\varepsilon}+N-k-2)!}{(b_{\varepsilon}+N-k-1)!}.
\label{inter-sum}
\end{align}
By putting the asymptotics
\[
\frac{(j+b_{\varepsilon}-1)!}{(b_{\varepsilon}+N-k-1)!}=\frac{1}{(b_{\varepsilon}+N-k-1)_{N-j-k}}
\underset{\varepsilon\to 0^+}{\sim} \left(\frac{\varepsilon}{b}\right)^{\!N-j-k},
\]
\[
\frac{(b_{\varepsilon}-a_{\varepsilon}+N-k-2)!}{(b_{\varepsilon}-a_{\varepsilon}+2N-2)!}
=\frac{1}{(b_{\varepsilon}-a_{\varepsilon}+2N-2)_{N+k}}
\underset{\varepsilon\to 0^+}{\sim}\left(\frac{\varepsilon}{b-a}\right)^{\!\!N+k}
\]
into~(\ref{inter-sum}), we obtain that
\begin{align*}
\lqn{\int\!\!\!\int_{\mathcal{D}^+} u^{-a_{\varepsilon}-1}(1-u)^{N-1}
v^{j+b_{\varepsilon}-1}(1-v)^{N-j-1}\,\mathrm{d} u\,\mathrm{d} v}
&
\underset{\varepsilon\to 0^+}{\sim} \frac{(N-1)!(N-j-1)!}{b^{\,N-j}(b-a)^N}\, \varepsilon^{2N-j}
\sum_{k=0}^{N-j-1}\binom{k+N-1}{k} \!\left(\frac{b}{b-a}\right)^{\!\!k}.
\end{align*}
Next, using the asymptotics
\[
\prod_{k=0}^{N-1} [(k-a_{\varepsilon})(k+b_{\varepsilon})]
\underset{\varepsilon\to 0^+}{\sim}  (-1)^N\frac{(ab)^N}{\varepsilon^{2N}},
\]
expression~(\ref{I+-inter}) admits the following asymptotics:
\begin{align*}
I_{a_{\varepsilon},b_{\varepsilon},j}^+
&\underset{\varepsilon\to 0^+}{\sim} \frac{(-1)^{j+N}b^{\,j}}{j!\,\varepsilon^j}
\left(\frac{a}{b-a}\right)^{\!\!N}
\sum_{k=0}^{N-j-1}\binom{k+N-1}{k} \!\left(\frac{b}{b-a}\right)^{\!\!k}.
\end{align*}
Then, we see that the second limit lying in~(\ref{comb-lin}) tends to
\[
\left(\frac{-a}{b-a}\right)^{\!\!N}
\sum_{j=0}^{N-1} \frac{(-\mathrm{i} \mu b)^j}{j!}
\left[\sum_{k=0}^{N-j-1}\binom{k+N-1}{k} \!\left(\frac{b}{b-a}\right)^{\!\!k}\right]\!.
\]
In the same way, it may be seen that the first term of the sum lying in~(\ref{comb-lin})
tends to
\[
\left(\frac{b}{b-a}\right)^{\!\!N}
\sum_{j=0}^{N-1} \frac{(\mathrm{i} \mu a)^j}{j!}
\left[\sum_{k=0}^{N-j-1}\binom{k+N-1}{k} \!\left(\frac{-a}{b-a}\right)^{\!\!k}\right]\!.
\]

Finally, let us have a look on the pseudo-probability $\mathbb{P}\{\tau_{ab}<+\infty\}$.
We have that
\begin{align*}
\mathbf{I}_{ab,0}^-
&=\left(\frac{b}{b-a}\right)^{\!\!N} \sum_{k=0}^{N-1}\binom{k+N-1}{k}
\!\left(\frac{-a}{b-a}\right)^{\!\!k}
\\
&
=\frac{b^N}{(b-a)^{2N-1}} \sum_{k=0}^{N-1}\binom{k+N-1}{k} (-a)^k(b-a)^{N-1-k}
\\
&
=\frac{b^N}{(b-a)^{2N-1}} \sum_{0\le k\le N-1\atop 0\le\ell\le N-1-k}
\binom{k+N-1}{k} \!\binom{N-1-k}{\ell} (-a)^{k+\ell}b^{N-1-k-\ell}
\\
&
=\frac{1}{(b-a)^{2N-1}} \sum_{m=0}^{N-1} \left[\,\sum_{k=0}^m
\binom{k+N-1}{k} \!\binom{N-1-k}{m-k}\!\right] (-a)^m b^{2N-1-m}.
\end{align*}
By using the elementary identity
$\sum_{k=0}^n \binom{k+p}{k} \!\binom{n+q-k}{n-k}=\binom{n+p+q+1}{n}$
which comes from the equality $(1+x)^{-p}(1+x)^{-q}=(1+x)^{-p-q}$
together with the expansion, e.g., for $p$, $(1+x)^{-p}=\sum_{k=0}^{\infty}
(-1)^k\binom{k+p-1}{k} x^k$, we get that
\[
\sum_{k=0}^m \binom{k+N-1}{k} \!\binom{N-1-k}{m-k}=\binom{2N-1}{m}.
\]
As a byproduct,
\[
\mathbf{I}_{ab,0}^-=\frac{1}{(b-a)^{2N-1}} \sum_{m=0}^{N-1}\binom{2N-1}{m}
(-a)^m b^{2N-1-m}.
\]
Similarly,
\[
\mathbf{I}_{ab,0}^+=\frac{1}{(b-a)^{2N-1}} \sum_{m=N}^{2N-1}\binom{2N-1}{m}
(-a)^m b^{2N-1-m}
\]
and we deduce that
\[
\mathbb{P}\{\tau_{ab}<+\infty\}=\mathbf{I}_{ab,0}^-+\mathbf{I}_{ab,0}^+
= \frac{1}{(b-a)^{2N-1}} \sum_{m=0}^{2N-1}\binom{2N-1}{m} (-a)^m b^{2N-1-m}=1.
\]
The proof of Theorem~\ref{theo-dist-Xab} is finished.
\end{proof}
%
\begin{corollary}
The pseudo-density of $X_{ab}$ is given by
\[
\mathbb{P}\{X_{ab}\in \mathrm{d} z\}/\mathrm{d} z=\sum_{j=0}^{N-1}
(-1)^j\mathbf{I}_{ab,j}^-\,\delta_a^{(j)}(z)
+\sum_{j=0}^{N-1} (-1)^j\mathbf{I}_{ab,j}^+\,\delta_b^{(j)}(z).
\]
In particular,
\[
\mathbb{P}\{\tau_a^-<\tau_b^+\}=\mathbf{I}_{ab,0}^-,\quad \mathbb{P}
\{\tau_b^+<\tau_a^-\}=\mathbf{I}_{ab,0}^+.
\]
\end{corollary}
%
This result has been announced in~\cite{la5} without any proof. We shall develop
it in a forthcoming paper.

\newpage\appendix
\renewcommand{\thesection}{\textbf{\Large \Alph{section}}}
\begin{center}
\section{\textbf{\Large Appendix}}
\end{center}
\renewcommand{\thesection}{\Alph{section}}

\subsection{Lacunary Vandermonde systems}\label{appendix-determinants}

Let us introduce the ``lacunary'' Vandermonde determinant (of type $(p+r)\!\times\!(p+r)$):
\[
W_{pqr}(u_1,\dots,u_{p+r})
=\begin{vmatrix}
1      & u_1     & \dots & u_1^{p-1}     & u_1^{p+q}     & \dots & u_1^{p+q+r-1}    \\
\vdots & \vdots  &       & \vdots        & \vdots        &       & \vdots           \\[1ex]
1      & u_{p+r} & \dots & u_{p+r}^{p-1} & u_{p+r}^{p+q} & \dots & u_{p+r}^{p+q+r-1}
\end{vmatrix}\!.
\]
We put $s_0:=s_0(u_1,\dots,u_{p+r})=1$ and, for $\ell\in\{1,\dots,p+r\}$,
\[
s_{\ell}:=s_{\ell}(u_1,\dots,u_{p+r})
=\sum_{1\le i_1<\dots<i_{\ell}\le p+r} u_{i_1}\cdots u_{i_{\ell}}.
\]
We say ``lacunary'' because it comes from a genuine Vandermonde determinant
where the powers from $p$ to $(p+q-1)$ are missing. More precisely,
the determinant $W_{pqr}(u_1,\dots,$ $u_{p+r})$ is extracted from the classical
Vandermonde determinant $V_{p+q+r}(u_1,\dots,u_{p+q+r})$ of type $(p+q+r)\!\times\!(p+q+r)$
by removing the $q$ last rows and the $(p+1)$st, $(p+2)$nd, $\dots$, $(p+q)$th columns.
We decompose $V_{p+q+r}(u_1,\dots,u_{p+q+r})$ into blocks as follows:
\[
\left|\begin{array}{@{\hspace{0em}}cccc;{.3ex/.5ex}ccc;{.3ex/.5ex}ccc@{\hspace{0em}}}
1      & u_1       & \dots & u_1^{p-1}       & u_1^p       & \dots & u_1^{p+q-1}       & u_1^{p+q}       & \dots & u_1^{p+q+r-1}       \\
\vdots & \vdots    &       & \vdots          & \vdots      &       & \vdots            & \vdots          &       & \vdots              \\[1ex]
1      & u_{p+r}   & \dots & u_{p+r}^{p-1}   & u_{p+r}^p   & \dots & u_{p+r}^{p+q-1}   & u_{p+r}^{p+q}   & \dots & u_{p+r}^{p+q+r-1}   \\[1ex]
\hdashline[.3ex/.5ex]
&&&&&&&&&\\[-2ex]
1      & u_{p+r+1} & \dots & u_{p+r+1}^{p-1} & u_{p+r+1}^p & \dots & u_{p+r+1}^{p+q-1} & u_{p+r+1}^{p+q} & \dots & u_{p+r+1}^{p+q+r-1} \\
\vdots & \vdots    &       & \vdots          & \vdots      &       & \vdots            & \vdots          &       & \vdots              \\[1ex]
1      & u_{p+q+r} & \dots & u_{p+q+r}^{p-1} & u_{p+q+r}^p & \dots & u_{p+q+r}^{p+q-1} & u_{p+q+r}^{p+q} & \dots & u_{p+q+r}^{p+q+r-1}
\end{array}\right|\!.
\]
By moving the $r$ last columns before the $q$ previous ones, this determinant
can be rewritten as
\[
(-1)^{qr}\left|\begin{array}{@{\hspace{0em}}ccccccc;{.3ex/.5ex}ccc@{\hspace{0em}}}
1      & u_1       & \dots & u_1^{p-1}       & u_1^{p+q}       & \dots & u_1^{p+q+r-1}       & u_1^p       & \dots & u_1^{p+q-1}       \\
\vdots & \vdots    &       & \vdots          & \vdots          &       & \vdots              & \vdots      &       & \vdots            \\[1ex]
1      & u_{p+r}   & \dots & u_{p+r}^{p-1}   & u_{p+r}^{p+q}   & \dots & u_{p+r}^{p+q+r-1}   & u_{p+r}^p   & \dots & u_{p+r}^{p+q-1}   \\[1ex]
\hdashline[.3ex/.5ex]
&&&&&&&&&\\[-2ex]
1      & u_{p+r+1} & \dots & u_{p+r+1}^{p-1} & u_{p+r+1}^{p+q} & \dots & u_{p+r+1}^{p+q+r-1} & u_{p+r+1}^p & \dots & u_{p+r+1}^{p+q-1} \\
\vdots & \vdots    &       & \vdots          & \vdots          &       & \vdots              & \vdots      &       & \vdots            \\[1ex]
1      & u_{p+q+r} & \dots & u_{p+q+r}^{p-1} & u_{p+q+r}^{p+q} & \dots & u_{p+q+r}^{p+q+r-1} & u_{p+q+r}^p & \dots & u_{p+q+r}^{p+q-1}
\end{array}\right|\!.
\]
By appealing to an expansion by blocks of a determinant, it may be seen that
$W_{pqr}(u_1,$ $\dots,u_{p+r})$ is the cofactor of the ``south-east'' block of the
above determinant. Since the product of, e.g., the diagonal terms of this last
block is $u_{p+r+1}^p u_{p+r+2}^{p+1} \cdots u_{p+q+r}^{p+q-1}$, the determinant
$W_{pqr}(u_1,\dots,u_{p+r})$ is also the coefficient of
$u_{p+r+1}^p u_{p+r+2}^{p+1} \cdots u_{p+q+r}^{p+q-1}$ in
$V_{p+q+r}(u_1,\dots,u_{p+q+r})$. Now, let us expand $V_{p+q+r}(u_1,\dots,u_{p+q+r})$:
\[
V_{p+q+r}(u_1,\dots,u_{p+q+r})
=\prod_{1\le i<j\le p+q+r} (u_j-u_i)=\textstyle{\displaystyle\prod}_1
\,\textstyle{\displaystyle\prod}_2 \,\textstyle{\displaystyle\prod}_3
\]
with
\begin{align*}
\textstyle{\displaystyle\prod}_1 &=\prod_{1\le i<j\le p+r} (u_j-u_i),
\\
\textstyle{\displaystyle\prod}_2 &=\prod_{1\le i\le p+r\atop p+r+1\le j\le p+q+r} (u_j-u_i),
\\
\textstyle{\displaystyle\prod}_3 &=\prod_{p+r+1\le i<j\le p+q+r} (u_j-u_i).
\end{align*}
We have that
\begin{align*}
\textstyle{\displaystyle\prod}_2
&
=\prod_{j=p+r+1}^{p+q+r} [(u_j-u_1)\cdots(u_j-u_{p+r})]
=\prod_{j=p+r+1}^{p+q+r}\sum_{k=0}^{p+r} (-1)^{p+r-k} s_{p+r-k}u_j^k
\\
&
=\sum_{0\le k_1,\dots,k_q\le p+r} (-1)^{q(p+r)-k_1-\dots-k_q}
s_{p+r-k_1}\cdots s_{p+r-k_q} u_{p+r+1}^{k_1}\cdots u_{p+q+r}^{k_q}
\end{align*}
and
\begin{align*}
\textstyle{\displaystyle\prod}_3
&
=V_{q}(u_{p+r+1},\dots,u_{p+q+r})
=\sum_{\varsigma\in\mathfrak{S}_q} \epsilon(\varsigma) u_{p+r+1}^{\varsigma(1)-1}
\cdots u_{p+q+r}^{\varsigma(q)-1}
\\
&
=\sum_{0\le\ell_1,\dots,\ell_q\le q-1\atop\ell_1,\dots,\ell_q \text{ all different}}
\epsilon(\ell_1\dots,\ell_q)u_{p+r+1}^{\ell_1}\cdots u_{p+q+r}^{\ell_q}.
\end{align*}
The symbol $\mathfrak{S}_q$ in the above sum denotes the set of the permutations of
the numbers $1,2,\dots,q$, $\epsilon(\varsigma)$ is the signature of the permutation
$\varsigma$ and $\epsilon(\ell_1\dots,\ell_q)\in\{-1,+1\}$
is the signature of the permutation mapping $1,\dots,q$ into
$\ell_1\dots,\ell_q$. The product $\prod_2\prod_3$ is given by
\[
\displaystyle (-1)^{q(p+1)}\sum_{\begin{array}{c}\mbox{}\\[-3.5ex]
\scriptscriptstyle 0\le k_1,\dots,k_q\le p+r\\[-1.2ex]
\scriptscriptstyle 0\le\ell_1,\dots,\ell_q\le q-1\\[-1.2ex]
\scriptscriptstyle\ell_1,\dots,\ell_q \text{ all different}\end{array}}
(-1)^{k_1+\dots +k_q}\epsilon(\ell_1\dots,\ell_q)
s_{p+r-k_1}\cdots s_{p+r-k_q} u_{p+r+1}^{k_1+\ell_1}\cdots u_{p+q+r}^{k_q+\ell_q}.
\]
For obtaining the coefficient of
$u_{p+r+1}^p u_{p+r+2}^{p+1} \cdots u_{p+q+r}^{p+q-1}$,
we only keep in the foregoing sum the indices $k_1,\dots,k_q,\ell_1,\dots,\ell_q$
such that $k_1+\ell_1=p$, $k_2+\ell_2=p+1$, $\dots$, $k_q+\ell_q=p+q-1$
and $0\le k_1,\dots,k_q\le p+r$, $0\le \ell_1,\dots,\ell_q\le q-1$,
the indices $\ell_1,\dots,\ell_q$ being all distinct.
This gives that
\[
\textstyle{\displaystyle\prod}_2\,\textstyle{\displaystyle\prod}_3
=\displaystyle\sum_{\begin{array}{c}\mbox{}\\[-3.5ex]
\scriptscriptstyle\text{for }i\in\{1,\dots,q\}:\\[-1.2ex]
\scriptscriptstyle(i-r-1)\vee 0\le\ell_i\le (p+i-1)\wedge(q-1)\\[-1.2ex]
\scriptscriptstyle\ell_1,\dots,\ell_q \text{ all different}\end{array}}
\epsilon(\ell_1\dots,\ell_q) s_{\ell_1+r}s_{\ell_2+r-1}\cdots s_{\ell_q+r-q+1}.
\]
Finally, we can observe that the foregoing sum is nothing but the expansion
of the determinant
\[
\mathcal{S}_{pqr}=\begin{vmatrix}
s_r       & s_{r-1}   & \dots & s_{r-q+1} \\
s_{r+1}   & s_r       & \dots & s_{r-q+2} \\
\vdots    & \vdots    &       & \vdots    \\
s_{r+q-1} & s_{r+q-2} & \dots & s_r
\end{vmatrix}\!.
\]
As a result, we obtain the result below.
%
\begin{proposition}\label{proposition-W}
The determinant $W_{pqr}(u_1,\dots,u_{p+r})$ admits the following expression:
\begin{align*}
\lqn{W_{pqr}(u_1,\dots,u_{p+r})= \prod_{1\le i<j\le p+r} (u_j-u_i)}
&\times \begin{vmatrix}
s_r(u_1,\dots,u_{p+r})       & s_{r-1}(u_1,\dots,u_{p+r})   & \dots & s_{r-q+1}(u_1,\dots,u_{p+r}) \\
s_{r+1}(u_1,\dots,u_{p+r})   & s_r(u_1,\dots,u_{p+r})       & \dots & s_{r-q+2}(u_1,\dots,u_{p+r}) \\
\vdots                       & \vdots                       &       & \vdots    \\
s_{r+q-1}(u_1,\dots,u_{p+r}) & s_{r+q-2}(u_1,\dots,u_{p+r}) & \dots & s_r(u_1,\dots,u_{p+r})
\end{vmatrix}\!.
\end{align*}
\end{proposition}
%

Let now $W_{pqr,\ell}\!\left(\!\!\!\begin{array}{c} u_1,\dots,u_{p+r}\\[-.2ex]
\displaystyle\alpha_1,\dots,\alpha_{p+r}\end{array}\!\!\!\right)$
be the determinant deduced from $W_{pqr}(u_1,\dots,u_{p+r})$
by replacing one column by a general column
$\alpha_1,\dots,\alpha_{p+r}$, that is, the determinant given,
if $0\le\ell\le p-1$, by
\[
\begin{vmatrix}
1      & u_1     & \dots & u_1^{\ell-1}     & \alpha_1     & u_1^{\ell+1}     & \dots & u_1^{p-1}     & u_1^{p+q}     & \dots \hspace{1em} u_1^{p+q+r-1}     \\
\vdots & \vdots  &       & \vdots           & \vdots       & \vdots           &       & \vdots        & \vdots        & \hphantom{\dots} \hspace{1em}\vdots  \\[1ex]
1      & u_{p+r} & \dots & u_{p+r}^{\ell-1} & \alpha_{p+r} & u_{p+r}^{\ell+1} & \dots & u_{p+r}^{p-1} & u_{p+r}^{p+q} & \dots \hspace{1em} u_{p+r}^{p+q+r-1}
\end{vmatrix}\!,
\]
and, if $p+q\le\ell\le p+q+r-1$, by
\[
\begin{vmatrix}
1      & u_1     & \dots & u_1^{p-1}     & u_1^{p+q}     & \dots & u_1^{\ell-1}     & \alpha_1     & u_1^{\ell+1}     & \dots \hspace{1em} u_1^{p+q+r-1}     \\
\vdots & \vdots  &       & \vdots        & \vdots        &       & \vdots           & \vdots       & \vdots           & \hphantom{\dots} \hspace{1em}\vdots  \\[1ex]
1      & u_{p+r} & \dots & u_{p+r}^{p-1} & u_{p+r}^{p+q} & \dots & u_{p+r}^{\ell-1} & \alpha_{p+r} & u_{p+r}^{\ell+1} & \dots \hspace{1em} u_{p+r}^{p+q+r-1}
\end{vmatrix}\!.
\]
We have that
\[
W_{pqr,\ell}\!\left(\!\!\!\begin{array}{c} u_1,\dots,u_{p+r}\\[-.2ex]
\displaystyle\alpha_1,\dots,\alpha_{p+r}\end{array}\!\!\!\right)\!
=\sum_{k=1}^{p+r} \alpha_k W_{pqr,k\ell}(u_1,\dots,u_{k-1},u_{k+1},\dots,u_{p+r})
\]
where $W_{pqr,k\ell}(u_1,\dots,u_{k-1},u_{k+1},\dots,u_{p+r})$ is the
determinant given, if $0\le\ell\le p-1$, by
\[
\left|\begin{array}
{@{\hspace{0em}}cc@{\hspace{0.7em}}c@{\hspace{0.7em}}ccc@{\hspace{0.7em}}c@{\hspace{0.7em}}cc@{\hspace{0.7em}}c@{\hspace{0.2em}}c@{\hspace{0em}}}
1      & u_1     & \dots & u_1^{\ell-1}     & 0      & u_1^{\ell+1}     & \dots & u_1^{p-1}     & u_1^{p+q}     & \dots \hspace{1em} u_1^{p+q+r-1}     \\
\vdots & \vdots  &       & \vdots           & \vdots & \vdots           &       & \vdots        & \vdots        & \hphantom{\dots} \hspace{1em}\vdots  \\[.5ex]
1      & u_{k-1} & \dots & u_{k-1}^{\ell-1} & 0      & u_{k-1}^{\ell+1} & \dots & u_{k-1}^{p-1} & u_{k-1}^{p+q} & \dots \hspace{1em} u_{k-1}^{p+q+r-1} \\[.5ex]
1      & u_k     & \dots & u_k^{\ell-1}     & 1      & u_k^{\ell+1}     & \dots & u_k^{p-1}     & u_k^{p+q}     & \dots \hspace{1em} u_k^{p+q+r-1}     \\[.5ex]
1      & u_{k+1} & \dots & u_{k+1}^{\ell-1} & 0      & u_{k+1}^{\ell+1} & \dots & u_{k+1}^{p-1} & u_{k+1}^{p+q} & \dots \hspace{1em} u_{k+1}^{p+q+r-1} \\[.5ex]
\vdots & \vdots  &       & \vdots           & \vdots & \vdots           &       & \vdots        & \vdots        & \hphantom{\dots} \hspace{1em}\vdots  \\[.5ex]
1      & u_{p+r} & \dots & u_{p+r}^{\ell-1} & 0      & u_{p+r}^{\ell+1} & \dots & u_{p+r}^{p-1} & u_{p+r}^{p+q} & \dots \hspace{1em} u_{p+r}^{p+q+r-1}
\end{array}\right|\!,
\]
and, if $p+q\le\ell\le p+q+r-1$, by
\[
\left|\begin{array}
{@{\hspace{0em}}cc@{\hspace{0.7em}}c@{\hspace{0.7em}}ccc@{\hspace{0.7em}}c@{\hspace{0.7em}}cc@{\hspace{0.7em}}c@{\hspace{0.2em}}c@{\hspace{0em}}}
1      & u_1     & \dots & u_1^{p-1}     & u_1^{p+q}     & \dots & u_1^{\ell-1}     & 0      & u_1^{\ell+1}     &  \dots \hspace{1em} u_1^{p+q+r-1}     \\
\vdots & \vdots  &       & \vdots        & \vdots        &       & \vdots           & \vdots & \vdots           & \hphantom{\dots} \hspace{1em}\vdots  \\[.5ex]
1      & u_{k-1} & \dots & u_{k-1}^{p-1} & u_{k-1}^{p+q} & \dots & u_{k-1}^{\ell-1} & 0      & u_{k-1}^{\ell+1} &\dots \hspace{1em} u_{k-1}^{p+q+r-1} \\[.5ex]
1      & u_k     & \dots & u_k^{p-1}     & u_k^{p+q}     & \dots & u_k^{\ell-1}     & 1      & u_k^{\ell+1}     & \dots \hspace{1em} u_k^{p+q+r-1}     \\[.5ex]
1      & u_{k+1} & \dots & u_{k+1}^{p-1} & u_{k+1}^{p+q} & \dots & u_{k+1}^{\ell-1} & 0      & u_{k+1}^{\ell+1} & \dots \hspace{1em} u_{k+1}^{p+q+r-1} \\[.5ex]
\vdots & \vdots  &       & \vdots        & \vdots        &       & \vdots           & \vdots & \vdots           & \hphantom{\dots} \hspace{1em}\vdots  \\[.5ex]
1      & u_{p+r} & \dots & u_{p+r}^{p-1} & u_{p+r}^{p+q} & \dots & u_{p+r}^{\ell-1} & 0      & u_{p+r}^{\ell+1} &\dots \hspace{1em} u_{p+r}^{p+q+r-1}
\end{array}\right|\!.
\]
In fact, $W_{pqr,k\ell}(u_1,\dots,u_{k-1},u_{k+1},\dots,u_{p+r})$
is the coefficient of $u_k^{\ell}$ in $W_{pqr}(u_1,\dots,u_{p+r})$.
Let us introduce $s_{k,0}:=s_{k,0}(u_1,\dots,u_{k-1},u_{k+1},\dots,u_{p+r})=1$,
$s_{k,m}:=s_{k,m}(u_1,\dots,u_{k-1},$ $u_{k+1},\dots,u_{p+r})$ $=0$
for any integer $m$ such that $m\le -1$ or $m\ge p+r$ and, for $m\in\{1,\dots,p+r-1\}$,
\[
s_{k,m}:=s_{k,m}(u_1,\dots,u_{k-1},u_{k+1},\dots,u_{p+r})
=\sum_{1\le i_1<\dots<i_m\le p+r\atop
i_1,\dots,i_m\neq k} u_{i_1}\cdots u_{i_m}.
\]
We need to isolate $u_k$ in $\prod_{1\le i<j\le p+r} (u_j-u_i)$ and
$\mathcal{S}_{pqr}$. First, we write that
\begin{align}
\prod_{1\le i<j\le p+r} (u_j-u_i)
&
=(-1)^{p+r-k}\prod_{1\le i<j\le p+r\atop i,j\neq k} (u_j-u_i)\times
\prod_{1\le i\le p+r\atop i\neq k} (u_k-u_i)
\nonumber\\
&
=(-1)^{k-1} \prod_{1\le i<j\le p+r\atop i,j\neq k} (u_j-u_i)\times
\sum_{m=0}^{p+r-1} (-1)^m s_{k,p+r-m-1} \,u_k^m.
\label{term1}
\end{align}
Second, by isolating $u_k$ in $s_{m}$ according as
$s_{m}=s_{k,m}+u_k\,s_{k,m-1}$, we get that
\[
\mathcal{S}_{pqr}=\begin{vmatrix}
s_{k,r}+u_k\,s_{k,r-1}       & s_{k,r-1}+u_k\,s_{k,r-2}     & \dots & s_{k,r-q+1}+u_k\,s_{k,r-q}   \\
s_{k,r+1}+u_k\,s_{k,r}       & s_{k,r}+u_k\,s_{k,r-1}       & \dots & s_{k,r-q+2}+u_k\,s_{k,r-q+1} \\
\vdots                       & \vdots                       &       & \vdots                      \\
s_{k,r+q-1}+u_k\,s_{k,r+q-2} & s_{k,r+q-2}+u_k\,s_{k,r+q-3} & \dots & s_{k,r}+u_k\,s_{k,r-1}
\end{vmatrix}\!.
\]
By introducing vectors $\vec{V}_m$ with coordinates $(s_{k,m},s_{k,m+1},\dots,
s_{k,m+q-1})$ (written as a column), this determinant can be rewritten as
\[
\mathcal{S}_{pqr}=\det\!\big(\vec{V}_r+u_k\,\vec{V}_{r-1},\vec{V}_{r-1}+u_k\,\vec{V}_{r-2},
\dots,\vec{V}_{r-q+1}+u_k\,\vec{V}_{r-q}\big).
\]
By appealing to multilinearity, it is easy to see that
\begin{equation}\label{term2}
\mathcal{S}_{pqr}=\sum_{n=0}^q \det\!\big(\vec{V}_r,\vec{V}_{r-1},\dots,
\vec{V}_{r-q+n+1},\vec{V}_{r-q+n-1},\dots,\vec{V}_{r-q}\big) \,u_k^n.
\end{equation}
Now, let us multiply the sum lying in (\ref{term1}) by (\ref{term2}):
\begin{align}
\lqn{\Bigg(\sum_{m=0}^{p+r-1} (-1)^m s_{k,p+r-m-1} \,u_k^m\Bigg)\!
\Bigg(\sum_{n=0}^q \det\!\big(\vec{V}_r,\vec{V}_{r-1},\dots,
\vec{V}_{r-q+n+1},\vec{V}_{r-q+n-1},\dots,\vec{V}_{r-q}\big) \,u_k^n\Bigg)}
&
=\sum_{\ell=0}^{p+q+r-1}\!\!\Bigg(\sum_{m=0\vee(\ell-q)}^{(p+r-1)\wedge \ell}
\!\!\!\!\!\!(-1)^m s_{k,p+r-m-1}
\nonumber\\
&
\hphantom{=\;\;\;} \times\det\!\big(\vec{V}_r,\vec{V}_{r-1},\dots,
\vec{V}_{\ell+r-q-m+1},\vec{V}_{\ell+r-q-m-1},\dots,\vec{V}_{r-q}\big)\!\Bigg)u_k^{\ell}.
\label{product-terms}
\end{align}
Recalling the convention that $s_{k,m}=0$ if $m\le -1$ or $m\ge p+r$,
the coefficient of $u_k^{\ell}$ in (\ref{product-terms}) can be written as
\begin{align*}
\lqn{\sum_{m=\ell-q}^{\ell} (-1)^m s_{k,p+r-m-1}
\det\!\big(\vec{V}_r,\vec{V}_{r-1},\dots,
\vec{V}_{\ell+r-q-m+1},\vec{V}_{\ell+r-q-m-1},\dots,\vec{V}_{r-q}\big)}
&=(-1)^{\ell}\sum_{m=0}^{q} (-1)^{m+q} s_{k,p+q+r-\ell-m-1}
\det\!\big(\vec{V}_r,\vec{V}_{r-1},\dots,
\vec{V}_{r-m+1},\vec{V}_{r-m-1},\dots,\vec{V}_{r-q}\big).
\end{align*}
In this form, we see that the coefficient of $u_k^{\ell}$ in (\ref{product-terms})
is nothing but the product of $(-1)^{\ell}$ by the expansion of the determinant
\[
\mathcal{S}_{pqr,k\ell}=\begin{vmatrix}
s_{k,r}     & s_{k,r-1}   & \dots & s_{k,r-q}   \\
s_{k,r+1}   & s_{k,r}     & \dots & s_{k,r-q+1} \\
\vdots      & \vdots      &       & \vdots    \\
s_{k,r+q-1} & s_{k,r+q-2} & \dots & s_{k,r-1}   \\
s_{k,p+q+r-\ell-1} & s_{k,p+q+r-\ell-2} & \dots & s_{k,p+r-\ell-1}
\end{vmatrix}\!.
\]
%
\begin{proposition}\label{proposition-Wkl}
The determinant $W_{pqr,k\ell}(u_1,\dots,u_{k-1},u_{k+1},\dots,u_{p+r})$ admits the following expression:
\begin{align*}
\lqn{W_{pqr,k\ell}(u_1,\dots,u_{k-1},u_{k+1},\dots,u_{p+r})=(-1)^{k+\ell-1}\prod_{1\le i<j\le p+r\atop i,j\neq k} (u_j-u_i)}
&\times \left|\begin{array}{@{\hspace{0em}}c@{\hspace{.3em}}c@{\hspace{.3em}}c@{\hspace{0em}}}
s_{k,r}(u_1,\dots,u_{k-1},u_{k+1},\dots,u_{p+r})     & \dots & s_{k,r-q}(u_1,\dots,u_{k-1},u_{k+1},\dots,u_{p+r})   \\
s_{k,r+1}(u_1,\dots,u_{k-1},u_{k+1},\dots,u_{p+r})   & \dots & s_{k,r-q+1}(u_1,\dots,u_{k-1},u_{k+1},\dots,u_{p+r}) \\
\vdots           &       & \vdots    \\
s_{k,r+q-1}(u_1,\dots,u_{k-1},u_{k+1},\dots,u_{p+r}) & \dots & s_{k,r-1}(u_1,\dots,u_{k-1},u_{k+1},\dots,u_{p+r})   \\
s_{k,p+q+r-\ell-1}(u_1,\dots,u_{k-1},u_{k+1},\dots,u_{p+r}) & \dots & s_{k,p+r-\ell-1}(u_1,\dots,u_{k-1},u_{k+1},\dots,u_{p+r})
\end{array}\right|\!.
\end{align*}
\end{proposition}
%
As a consequence of Propositions~\ref{proposition-W} and~\ref{proposition-Wkl},
we get the result below. Set
\[
p_k:=p_k(u_1,\dots,u_{p+r})=\prod_{1\le i\le p+r\atop i\neq k}(u_k-u_i).
\]

%
\begin{proposition}\label{proposition-cramer-lacunary}
Let $p,q,r$ be positive integers, let $u_1,\dots,u_{p+r}$ be distinct complex numbers
and let $\alpha_1,\dots,\alpha_{p+r}$ be complex numbers.
Set $I=\{0,\dots,p-1\}\cup\{p+q,\dots,p+q+r-1\}$.
The solution of the system $\sum_{\ell\in I}u_k^{\ell} \,x_{\ell}=\alpha_k$, $1\le k\le p+r$,
or, more explicitly,
\[\left\{
\begin{array}{r@{\hspace{0.1em}}r@{\hspace{0.2em}}r@{\hspace{0.1em}}r@{\hspace{0.1em}}r@{\hspace{0.2em}}r@{\hspace{0.1em}}r@{\hspace{0.3em}}c@{\hspace{0.4em}}l}
x_0+ & u_1\,x_1+     &\cdots+ & u_1^{p-1}\,x_{p-1}+      & u_1^{p+q}\,x_{p+q}+     & \cdots+ & u_1^{p+q+r-1}\,x_{p+q+r-1}     &=       & \alpha_1
\\
     &               &        &                          &                         &         &                                & \vdots &
\\
x_0+ & u_{p+r}\,x_1+ & \cdots+ & u_{p+r}^{p-1}\,x_{p-1}+ & u_{p+r}^{p+q}\,x_{p+q}+ & \cdots+ & u_{p+r}^{p+q+r-1}\,x_{p+q+r-1} & =      & \alpha_{p+r}
\end{array}\right.
\]
is given by
\begin{equation}\label{cramer-lacunary}
x_{\ell}=\frac{(-1)^{\ell+p+r-1}}{\mathcal{S}_{pqr}(u_1,\dots,u_{p+r})} \sum_{k=1}^{p+r} \alpha_k\,
\frac{\mathcal{S}_{pqr,k\ell}(u_1,\dots,u_{k-1},u_{k+1},\dots,u_{p+r})}{p_k(u_1,\dots,u_{p+r})},
\quad \ell\in I.
\end{equation}
\end{proposition}
%
\begin{proof}
Cramer's formulae yield that
\[
x_{\ell}=\frac{W_{pqr,\ell}\!\left(\!\!\!\begin{array}{c} u_1,\dots,u_{p+r}\\[-.2ex]
\displaystyle\alpha_1,\dots,\alpha_{p+r}\end{array}\!\!\!\right)}
{W_{pqr}(u_1,\dots,u_{p+r})}
=\sum_{k=1}^{p+r}\alpha_k\,\frac{W_{pqr,k\ell}(u_1,\dots,u_{k-1},
u_{k+1},\dots,u_{p+r})}{W_{pqr}(u_1,\dots,u_{p+r})}.
\]
By using the factorisations provided by Propositions~\ref{proposition-W}
and~\ref{proposition-Wkl}, namely
\begin{align*}
\lqn{W_{pqr}(u_1,\dots,u_{p+r})}
&
=V_{p+r}(u_1,\dots,u_{p+r}) \,\mathcal{S}_{pqr}(u_1,\dots,u_{p+r})
\\
&
=(-1)^{p+r-k} p_k(u_1,\dots,u_{p+r})\,
V_{p+r-1}(u_1,\dots,u_{k-1},u_{k+1},\dots,u_{p+r}) \,\mathcal{S}_{pqr}(u_1,\dots,u_{p+r}),
\\
\lqn{W_{pqr,k\ell}(u_1,\dots,u_{k-1},u_{k+1},\dots,u_{p+r})}
&
=(-1)^{k+\ell-1}V_{p+r-1}(u_1,\dots,u_{k-1},u_{k+1},\dots,u_{p+r})
\,\mathcal{S}_{pqr,k\ell}(u_1,\dots,u_{k-1},u_{k+1},\dots,u_{p+r}),
\end{align*}
we immediately get~(\ref{cramer-lacunary}).
\end{proof}

\subsection{A combinatoric identity}\label{lemma-comb}

%
\begin{lemma}\label{lemma-sum}
The following identity holds for any positive integers $\alpha,\beta,n$:
\[
\sum_{k=0}^n (-1)^k\binom{n}{k} \frac{(k+\alpha)!}{(k+\beta)!}
=\frac{\alpha!}{(\beta+n)!}\,(\beta-\alpha+n-1)_n.
\]
It can be rewritten as
\begin{align*}
\sum_{k=0}^n (-1)^k\binom{n}{k} \!\binom{k+\alpha}{k+\beta}
&\displaystyle=(-1)^n \binom{\alpha}{\beta+n}
&\hspace{-5em}\text{if $\alpha\ge\beta$},
\\
\sum_{k=0}^n (-1)^k\binom{n}{k} /\binom{k+\beta}{k+\alpha}
&\displaystyle=\frac{\beta-\alpha}{\beta-\alpha+n}/\binom{\beta+n}{\alpha}
&\hspace{-5em}\text{if $\alpha<\beta$}.
\end{align*}
\end{lemma}
%
\begin{proof}
Suppose first that $\alpha\ge \beta$. Noticing that
\[
\frac{(k+\alpha)!}{(k+\beta)!}=\left.\frac{\mathrm{d}^{\alpha-\beta}}{\mathrm{d} x^{\alpha-\beta}}
\big(x^{k+\alpha}\big)\right|_{x=1},
\]
we immediately get that
\[
\sum_{k=0}^n (-1)^k\binom{n}{k} \frac{(k+\alpha)!}{(k+\beta)!}
=\left.\frac{\mathrm{d}^{\alpha-\beta}}{\mathrm{d} x^{\alpha-\beta}}
\left(\sum_{k=0}^n (-1)^k\binom{n}{k}x^{k+\alpha}\right)\right|_{x=1}
=\left.\frac{\mathrm{d}^{\alpha-\beta}}{\mathrm{d} x^{\alpha-\beta}}
\big(x^{\alpha}(1-x)^n\big)\right|_{x=1}.
\]
We expand the last displayed derivative by using Leibniz rule:
\[
\left.\frac{\mathrm{d}^{\alpha-\beta}}{\mathrm{d} x^{\alpha-\beta}}
\big(x^{\alpha}(1-x)^n\big)\right|_{x=1}
=\sum_{k=0}^{\alpha-\beta} \binom{\alpha-\beta}{k}
\left.\frac{\mathrm{d}^{\alpha-\beta-k}}{\mathrm{d} x^{\alpha-\beta-k}}
(x^{\alpha})\right|_{x=1}
\left.\frac{\mathrm{d}^k}{\mathrm{d} x^k}\big((1-x)^n\big)\right|_{x=1}.
\]
Since $\left.\big(\mathrm{d}^k/\mathrm{d} x^k\big)\big((1-x)^n\big)
\right|_{x=1}\!\!=(-1)^n\delta_{nk} \,n!$,
we have that $\left.\big(\mathrm{d}^{\alpha-\beta}/\mathrm{d} x^{\alpha-\beta}\big)
\big(x^{\alpha}(1-x)^n\big)\right|_{x=1}$ $=0$ if $\alpha-\beta<n$,
and, if $\alpha-\beta\ge n$,
\begin{align*}
\left.\frac{\mathrm{d}^{\alpha-\beta}}{\mathrm{d} x^{\alpha-\beta}}
\big(x^{\alpha}(1-x)^n\big)\right|_{x=1}\!
&
=(-1)^n n!\binom{\alpha-\beta}{n}
\left.\frac{\mathrm{d}^{\alpha-\beta-n}}{\mathrm{d} x^{\alpha-\beta-n}}
(x^{\alpha})\right|_{x=1}\!
\\
&
=(-1)^n \frac{\alpha!(\alpha-\beta)!}{(\alpha-\beta-n)!(\beta+n)!}
\end{align*}
which coincides with the announced result.
Second, suppose that $\alpha<\beta$. Noticing that
\[
\frac{(k+\alpha)!}{(k+\beta)!}=\frac{1}{(\beta-\alpha-1)!}
\int_0^1 x^{k+\alpha}(1-x)^{\beta-\alpha-1} \,\mathrm{d} x,
\]
we get that
\begin{align*}
\sum_{k=0}^n (-1)^k\binom{n}{k} \frac{(k+\alpha)!}{(k+\beta)!}
&=\frac{1}{(\beta-\alpha-1)!}\int_0^1 \left(\sum_{k=0}^n
(-1)^k\binom{n}{k} x^k\right) x^{\alpha}(1-x)^{\beta-\alpha-1} \,\mathrm{d} x
\\
&=\frac{1}{(\beta-\alpha-1)!}\int_0^1 x^{\alpha}(1-x)^{\beta-\alpha+n-1} \,\mathrm{d} x
=\frac{\alpha!(\beta-\alpha+n-1)!}{(\beta-\alpha-1)!(\beta+n)!}
\end{align*}
which coincides with the announced result.
\end{proof}

\subsection{Some matrices}\label{appendix-matrices}

Let $\alpha,\beta\in\mathbb{N}$ such that $\alpha>\beta\ge N$ and set
\[
A=\begin{bmatrix}\displaystyle\binom{j+\alpha}{i+N}\end{bmatrix}_{0\le i,j\le N-1},\quad
B=\begin{bmatrix}\displaystyle\binom{\beta}{i+N}\end{bmatrix}_{0\le i\le N-1}
\]
with the convention of settings $\binom{i}{j}=0$ if $i>j$.
These matrices have been used for solving systems (\ref{system-H(1)1})
and (\ref{system-H(1)2}) with the choices $\alpha=b-a+N-1$ and $\beta=N-a-1$.
The aim of this section is to compute the product of the inverse of $A$ by $B$,
namely $A^{-1}B$. For this, we use Gauss elimination.
The result is displayed in Theorem~\ref{theorem-matrix-inv}.
The calculations are quite technical, so we perform them progressively,
the intermediate steps being stated in several lemmas (Lemmas~\ref{lemma-sum},
\ref{lemma1}, \ref{lemma2} and \ref{lemma3}).
%
\begin{lemma}\label{lemma1}
The matrix $A$ can be decomposed into $A=LU^{-1}$ where the matrices
$U$ and $L$ are given by
\begin{align*}
U &=\begin{bmatrix}\displaystyle(-1)^{i+j}\ind_{\{i\le j\}}\binom{j}{i}
\frac{(j+\alpha)_N}{(i+\alpha)_N}
\end{bmatrix}_{0\le i,j\le N-1},\quad
\\
L &=\begin{bmatrix}\displaystyle\ind_{\{i\ge j\}}
\binom{j+\alpha}{i+N}\frac{(i)_j}{(j+\alpha-N)_j}\end{bmatrix}_{0\le i,j\le N-1}.
\end{align*}
The regular matrices $U$ and $L$ are respectively upper and lower triangular.
\end{lemma}
%
\begin{proof}
We begin by detailing the algorithm providing the matrix $U$.
Call $C_0^{(0)},C_1^{(0)},\dots,$ $C_{N-1}^{(0)}$ the columns of $A$, that is,
for $j\in\{0,\dots, N-1\}$,
\[
C_j^{(0)}=\begin{bmatrix}\displaystyle \binom{j+\alpha}{i+N}\end{bmatrix}_{0\le i\le N-1}.
\]
Apply to them, except for $C_0^{(0)}$, the transformation defined, for
$j\in\{1,\dots, N-1\}$, by
\[
C_j^{(1)}=C_j^{(0)}-\frac{j+\alpha}{j+\alpha-N}\, C_{j-1}^{(0)}.
\]
The $C_0^{(0)},C_1^{(1)},\dots, C_{N-1}^{(1)}$ are the columns of a new matrix $L_1$.
Actually, this transformation corresponds to a matrix multiplication acting on $A$:
$L_1=AU_1$ with
\begin{align*}
U_1=\left[\begin{array}{c@{\hspace{.8em}};{.08em/.3em}cccc}
1      & -\frac{\alpha+1}{\alpha+1-N} & 0                            &        & 0
\\
0      & 1                            & -\frac{\alpha+2}{\alpha+2-N} &        & \vdots
\\
0      & 0                            & 1                            & \ddots & 0
\\
\vdots & \vdots                       &                              & \ddots & -\frac{\alpha+N-1}{\alpha-1}
\\
0      & 0                            & 0                            &        & 1
\end{array}\right]\!.
\end{align*}
Simple computations show that
\begin{align*}
L_1
&
=\begin{bmatrix}\displaystyle \left(\delta_{j0}+\ind_{\{j\ge 1\}}\frac{i}{j+\alpha-N}\right)
\!\binom{j+\alpha}{i+N}\end{bmatrix}_{0\le i,j\le N-1}
\\
&
=\begin{bmatrix}\displaystyle \binom{j+\alpha}{i+N}\frac{(i)_{j\wedge 1}}{(j+\alpha-N)_{j\wedge 1}}
\end{bmatrix}_{0\le i,j\le N-1}.
\end{align*}
Next, apply the second transformation to the columns of $L_1$ except for $C_0^{(0)}$
and $C_1^{(1)}$, defined, for $j\in\{2,\dots, N-1\}$, by
\begin{align*}
C_j^{(2)}&=C_j^{(1)}-\frac{j+\alpha}{j+\alpha-N}\, C_{j-1}^{(1)}
\\
&=C_j^{(0)}-2\,\frac{j+\alpha}{j+\alpha-N}\, C_{j-1}^{(0)}
+\frac{(j+\alpha)(j+\alpha-1)}{(j+\alpha-N)(j+\alpha-N-1)}\, C_{j-2}^{(1)}.
\end{align*}
The $C_0^{(0)},C_1^{(1)},C_2^{(2)},\dots, C_{N-1}^{(2)}$ are the columns of
a new matrix $L_2=AU_2$ where
\begin{align*}
U_2=\left[\begin{array}{cc@{\hspace{.8em}};{.08em/.3em}ccc}
1      & -\frac{\alpha+1}{\alpha+1-N} & \frac{(\alpha+2)(\alpha+1)}{(\alpha+2-N)(\alpha+1-N)} &        & 0
\\
0      & 1                            & -2\,\frac{\alpha+2}{\alpha+2-N}                       & \ddots & \vdots
\\
0      & 0                            & 1                                                     & \ddots & \frac{(\alpha+N-1)(\alpha+N-2)}{(\alpha-1)(\alpha-2)}
\\
\vdots & \vdots                       &                                                       & \ddots & -2\,\frac{\alpha+N-1}{\alpha-1}
\\
0      & 0                            & 0                                                     &        & 1
\end{array}\right]\!.
\end{align*}
Straightforward algebra yields that
\begin{align*}
L_2
&
=\begin{bmatrix}\displaystyle \left(\delta_{j0}+\frac{i}{j+\alpha-N}\,\delta_{j1}
+\ind_{\{j\ge 2\}}\frac{i(i-1)}{(j+\alpha-N)(j+\alpha-N-1)}\right)\!
\binom{j+\alpha}{i+N}\end{bmatrix}_{0\le i,j\le N-1}
\\
&
=\begin{bmatrix}\displaystyle \binom{j+\alpha}{i+N} \frac{(i)_{j\wedge 2}}{(j+\alpha-N)_{j\wedge 2}}
\end{bmatrix}_{0\le i,j\le N-1}.
\end{align*}
This method can be recursively extended : apply the $k$th transformation
($1\le k\le N-1$) defined, for $j\in\{k,\dots, N-1\}$, by
\begin{align*}
C_j^{(k)}&=C_j^{(k-1)}-\frac{j+\alpha}{j+\alpha-N}\, C_{j-1}^{(k-1)}
\\
&=C_j^{(k-2)}-2\,\frac{j+\alpha}{j+\alpha-N}\, C_{j-1}^{(k-2)}
+\frac{(j+\alpha)(j+\alpha-1)}{(j+\alpha-N)(j+\alpha-N-1)}\, C_{j-2}^{(k-2)}
\\
&\;\;\vdots
\\
&=C_j^{(0)}-\binom{k}{1}\frac{j+\alpha}{j+\alpha-N}\, C_{j-1}^{(0)}
+\binom{k}{2}\frac{(j+\alpha)(j+\alpha-1)}{(j+\alpha-N)(j+\alpha-N-1)}\, C_{j-2}^{(0)}
\\
&\hphantom{=\;}
+\dots+(-1)^k\binom{k}{k}\frac{(j+\alpha)(j+\alpha-1)\cdots(j+\alpha-k+1)}
{(j+\alpha-N)(j+\alpha-N-1)\cdots(j+\alpha-k-N+1)}\, C_{j-k}^{(0)}.
\\
&
=\sum_{\ell=0}^k (-1)^{\ell} \binom{k}{\ell} \frac{(j+\alpha)_{\ell}}{(j+\alpha-N)_{\ell}}
\,C_{j-\ell}^{(0)}
=\sum_{\ell=0}^k (-1)^{\ell} \binom{k}{\ell} \frac{(j+\alpha)_N}{(j+\alpha-\ell)_N}
\,C_{j-\ell}^{(0)}
\\
&
=\sum_{\ell=j-k}^j (-1)^{j-\ell} \binom{k}{j-\ell} \frac{(j+\alpha)_{j-\ell}}{(j+\alpha-N)_{j-\ell}}
\,C_{\ell}^{(0)}
=\sum_{\ell=j-k}^j (-1)^{j-\ell} \binom{k}{j-\ell} \frac{(j+\alpha)_N}{(\ell+\alpha)_N}
\,C_{\ell}^{(0)}.
\end{align*}
In particular,
\[
C_k^{(k)}=\sum_{\ell=0}^k (-1)^{k-\ell} \binom{k}{\ell} \frac{(k+\alpha)_N}{(\ell+\alpha)_N}
\,C_{\ell}^{(0)}.
\]
The $C_0^{(0)},C_1^{(1)},C_2^{(2)},\dots, C_k^{(k)},C_{k+1}^{(k)},\dots,C_{N-1}^{(k)}$
are the columns of the $k$th matrix $L_k=AU_k$ with
$U_k=[\!\begin{array}{c;{.08em/.2em}c}U_k'&U_k''\end{array}\!]$ where
$U_k'$ is the matrix
\begin{align*}
\left[\begin{array}{ccccccc}
1      & -\frac{\alpha+1}{\alpha+1-N} & \frac{(\alpha+2)(\alpha+1)}{(\alpha+2-N)(\alpha+1-N)} &        &        &        & (-1)^{k-1}\binom{k-1}{k-1}\frac{(\alpha+k-1)\cdots(\alpha+1)}{(\alpha+k-N-1)\cdots(\alpha+1-N)}
\\
0      & 1                            & -2\,\frac{\alpha+2}{\alpha+2-N}                       & \ddots &        &        & (-1)^{k-2}\binom{k-1}{k-2}\frac{(\alpha+k-1)\cdots(\alpha+2)}{(\alpha+k-N-1)\cdots(\alpha+2-N)}
\\
0      & 0                            & 1                                                     & \ddots & \ddots &        & (-1)^{k-3}\binom{k-1}{k-3}\frac{(\alpha+k-1)\cdots(\alpha+3)}{(\alpha+k-N-1)\cdots(\alpha+3-N)}
\\
\vdots & \vdots                       &                                                       & \ddots & \ddots & \ddots & \vdots
\\
\vdots & \vdots                       &                                                       &        & \ddots & \ddots & \binom{k-1}{2} \frac{(\alpha+k-1)(\alpha+k-2)}{(\alpha+k-N-1)(\alpha+k-N-2)}
\\
\vdots & \vdots                       &                                                       &        &        & \ddots & -\binom{k-1}{1} \frac{\alpha+k-1}{\alpha+k-N-1}
\\[1ex]
0      & 0                            &                                                       &        &        &        & \binom{k-1}{0}
\\[.5ex]
\hdashline[.3ex/.5ex]
\\[-2.4ex]
0      & 0                            & \dots                                                 & \dots  & \dots  & \dots  & 0
\\
\vdots & \vdots                       &                                                       &        &        &        & \vdots
\\
0      & 0                            & \dots                                                 & \dots  & \dots  & \dots  & 0
\end{array}\right]
\end{align*}
and $U_k''$ is the matrix
\begin{align*}
\left[\begin{array}{@{\hspace{0em}}c@{\hspace{.5em}}c@{\hspace{.5em}}c@{\hspace{.5em}}c@{\hspace{0em}}}
(-1)^k\binom{k}{k} \frac{(\alpha+k)\cdots(\alpha+1)}{(\alpha+k-N)\cdots(\alpha+1-N)} & 0                                                                                      &        & 0
\\
\vdots                                                                               & (-1)^k\binom{k}{k}\frac{(\alpha+k+1)\cdots(\alpha+2)}{(\alpha+k-N+1)\dots(\alpha+2-N)} &        & \vdots
\\
\binom{k}{2} \frac{(\alpha+k)(\alpha+k-1)}{(\alpha+k-N)(\alpha+k-N-1)}               & \vdots                                                                                 & \ddots & 0
\\
-\binom{k}{1} \frac{\alpha+k}{\alpha+k-N}                                            & \binom{k}{2} \frac{(\alpha+k+1)(\alpha+k)}{(\alpha+k-N+1)(\alpha+k-N)}                 &        & (-1)^k\binom{k}{k} \frac{(\alpha+N-1)\cdots(\alpha+N-k)}{(\alpha-1)\dots(\alpha-k)}
\\
\binom{k}{0}                                                                         & -\binom{k}{1} \frac{\alpha+k+1}{\alpha+k-N+1}                                          & \ddots & \vdots
\\
0                                                                                    & \binom{k}{0}                                                                           & \ddots & \binom{k}{2} \frac{(\alpha+N-1)(\alpha+N-2)}{(\alpha-1)(\alpha-2)}
\\
\vdots                                                                               &                                                                                        & \ddots & -\binom{k}{1} \frac{\alpha+N-1}{\alpha-1}
\\[1ex]
0                                                                                    &                                                                                        &        & \binom{k}{0}
\end{array}\right]\!.
\end{align*}
The matrices $U_k'$ and $U_k''$ can be simply written as
\begin{align*}
U_k'&=\begin{bmatrix}\displaystyle(-1)^{j-i} \ind_{\{i\le j\}} \binom{j}{i}
\frac{(j+\alpha)_N}{(i+\alpha)_N}\end{bmatrix}_{0\le i\le N-1 \atop 0\le j\le k-1},
\\
U_k''&=\begin{bmatrix}\displaystyle(-1)^{j-i} \ind_{\{i\vee k\le j\le i+k\}} \binom{k}{j-i}
\frac{(j+\alpha)_N}{(i+\alpha)_N}\end{bmatrix}_{0\le i\le N-1 \atop k\le j\le N-1}.
\end{align*}
Clever algebra yields that
\begin{equation}\label{matrix-Lk}
L_k=\begin{bmatrix}\displaystyle\binom{j+\alpha}{i+N}\frac{(i)_{j\wedge k}}{(j+\alpha-N)_{j\wedge k}}
\end{bmatrix}_{0\le i,j\le N-1}.
\end{equation}
We shall not prove~(\ref{matrix-Lk}), we shall only check it below in the
case $k=N-1$.

We progressively arrive at the last transformation which corresponds to
$k=N-1$:
\begin{align*}
C_{N-1}^{(N-1)}&=C_{N-1}^{(0)}-\binom{N-1}{1} \frac{N+\alpha-1}{\alpha-1}\, C_{N-2}^{(0)}
+\binom{N-1}{2} \frac{(N+\alpha-1)(N+\alpha-2)}{(\alpha-1)(\alpha-2)}\, C_{N-3}^{(0)}
\\
&+(-1)^{N-1}\binom{N-1}{N-1} \frac{(N+\alpha-1)(N+\alpha-2)\cdots(\alpha+1)}
{(\alpha-1)(\alpha-2)\cdots(\alpha-N+1)}\, C_{0}^{(0)}
\\
&=\sum_{\ell=0}^{N-1} (-1)^{\ell} \binom{N-1}{\ell} \frac{(N+\alpha-1)_{\ell}}{(\alpha-1)_{\ell}}
\,C_{N-1-\ell}^{(0)}
\\
&=\sum_{\ell=0}^{N-1} (-1)^{N-1-\ell} \binom{N-1}{\ell}
\frac{(N+\alpha-1)_{N-\ell-1}}{(\alpha-1)_{N-\ell-1}} \,C_{\ell}^{(0)}.
\end{align*}
The $C_0^{(0)},C_1^{(1)},C_2^{(2)},\dots,C_{N-1}^{(N-1)}$
are the columns of the last matrix given by $L_{N-1}=AU_{N-1}$ where
\[
U_{N-1}=\begin{bmatrix}\displaystyle  (-1)^{j-i} \binom{j}{i}
\frac{(j+\alpha)_N}{(i+\alpha)_N}\end{bmatrix}_{0\le i,j\le N-1}.
\]
Formula (\ref{matrix-Lk}) gives the following expression for $L_{N-1}$ that will be
checked below:
\[
L_{N-1}=\begin{bmatrix}\displaystyle \binom{j+\alpha}{i+N}\frac{(i)_j}{(j+\alpha-N)_j}
\end{bmatrix}_{0\le i,j\le N-1}.
\]
Hence, by putting $L=L_{N-1}$ and $U=U_{N-1}$, we see that $L$ is a lower triangular matrix,
$U$ is an upper triangular matrix and we have obtained that $L=AU$.

Finally, we directly check the decomposition $L=AU$.
The generic term of $AU$ is
\[
\sum_{k=0}^j (-1)^{j+k} \binom{k+\alpha}{i+N} \!\binom{j}{k}
\frac{(j+\alpha)_N}{(k+\alpha)_N}.
\]
Observing that
\[
\binom{k+\alpha}{i+N}\frac{(j+\alpha)_N}{(k+\alpha)_N}
=\binom{k+\alpha-N}{i}\!\binom{j+\alpha}{N}/\binom{i+N}{N},
\]
this term can be rewritten as
\[
(-1)^j\binom{j+\alpha}{N}/\binom{i+N}{N}
\times\sum_{k=0}^j (-1)^k \binom{j}{k} \!\binom{k+\alpha-N}{i}.
\]
The sum $\sum_{k=0}^j (-1)^k \binom{j}{k} \!\binom{k+\alpha-N}{i}$
can be explicitly evaluated thanks to Lemma~\ref{lemma-sum}. Its value
is $(-1)^j \binom{\alpha-N}{i-j}$. Therefore, we can easily get the generic
term of $L$ and the proof of Lemma~\ref{lemma1} is finished.
\end{proof}

%
\begin{lemma}\label{lemma2}
The inverse of the matrix $L$ is given by
\[
L^{-1}=\begin{bmatrix}\displaystyle(-1)^{i+j} \ind_{\{i\ge j\}}
\frac{(j+N)!(i+\alpha-N)_{i+1}}{j!(i-j)!(i+\alpha)_{j+N+1}}
\end{bmatrix}_{0\le i,j\le N-1}\!.
\]
\end{lemma}
%
\begin{proof}
We simplify the entries of the product
\begin{align}
\lqn{\begin{bmatrix}\displaystyle\ind_{\{i\ge j\}}\binom{j+\alpha}{i+N}\frac{(i)_j}{(j+\alpha-N)_j}
\end{bmatrix}_{0\le i,j\le N-1}}
\times\begin{bmatrix}\displaystyle(-1)^{i+j} \ind_{\{i\ge j\}}
\frac{(j+N)!(i+\alpha-N)_{i+1}}{j!(i-j)!(i+\alpha)_{j+N+1}}
\end{bmatrix}_{0\le i,j\le N-1}\!.
\label{product}
\end{align}
The generic term of this matrix is
\begin{align*}
\lqn{\sum_{k=0}^{N-1} \ind_{\{i\ge k\}} \binom{k+\alpha}{i+N}\frac{(i)_k}{(k+\alpha-N)_k}
\times(-1)^{j+k} \ind_{\{k\ge j\}}\frac{(j+N)!(k+\alpha-N)_{k+1}}{j!(k-j)!(k+\alpha)_{j+N+1}}}
=\ind_{\{i\ge j\}}\frac{i!(j+N)!(\alpha-N)!}{j!(i+N)!(\alpha-N-1)!}
\sum_{k=j}^i (-1)^{j+k} \frac{(k+\alpha-j-N-1)!}{(i-k)!(k-j)!(k+\alpha-i-N)!}.
\end{align*}
The last sum can be computed as follows: clearly, it vanishes when $i<j$
and it equals $1$ when $i=j$. If $i>j$, by using Lemma~\ref{lemma-sum},
\begin{align*}
\lqn{\sum_{k=j}^i (-1)^{j+k} \frac{(k+\alpha-j-N-1)!}{(i-k)!(k-j)!(k+\alpha-i-N)!}}
&
=\frac{1}{(i-j)!}\sum_{k=j}^i (-1)^{j+k}\binom{i-j}{k-j}
\frac{(k+\alpha-j-N-1)!}{(k+\alpha-i-N)!}
\\
&
=\frac{1}{(i-j)!}\sum_{k=0}^{i-j} (-1)^k \binom{i-j}{k}(k+\alpha-N-1)_{i-j-1}=0.
\end{align*}
As a consequence, the entries of the product (\ref{product}) are $\delta_{ij}$
which proves that the second factor of (\ref{product}) coincides with $L^{-1}$.
\end{proof}
%
\begin{lemma}\label{lemma3}
The matrix $L^{-1}B$ is given by
\[
L^{-1}B=\begin{bmatrix}\displaystyle(-1)^i\binom{\beta}{N}\!\binom{i+\alpha-\beta-1}{\alpha-\beta-1}
/\binom{i+\alpha}{N}\end{bmatrix}_{0\le i\le N-1}\!.
\]
\end{lemma}
%
\begin{proof}
The generic term of $L^{-1}B$ is
\begin{align}
\lqn{\sum_{j=0}^{N-1} (-1)^{i+j} \ind_{\{i\ge j\}}
\frac{(j+N)!(i+\alpha-N)_{i+1}}{j!(i-j)!(i+\alpha)_{j+N+1}} \binom{\beta}{j+N}}
&
=\frac{\beta!(i+\alpha-N)!}{(i+\alpha)!(\alpha-N-1)!}
\sum_{j=0}^i (-1)^{i+j} \frac{(i-j+\alpha-N-1)!}{j!(i-j)!(\beta-j-N)!}
\nonumber\\
&
=\frac{\beta!(i+\alpha-N)!(i+\alpha-\beta-1)!}{i!(i+\alpha)!(\alpha-N-1)!}
\sum_{j=0}^i (-1)^{i+j} \binom{i}{j}\! \binom{i-j+\alpha-N-1}{i+\alpha-\beta-1}\!.
\label{sum1}
\end{align}
By performing the change of index $j\mapsto i-j$ and by using
Lemma~\ref{lemma-sum}, the sum in (\ref{sum1}) is equal to
\[
\sum_{j=0}^i (-1)^j \binom{i}{j}\! \binom{j+\alpha-N-1}{i+\alpha-\beta-1}
=(-1)^i \binom{\alpha-N-1}{\beta-N}
=(-1)^i \frac{(\alpha-N-1)!}{(\alpha-\beta-1)!(\beta-N)!}.
\]
By putting this into~(\ref{sum1}), we see that the generic term of $L^{-1}B$ writes
\[
(-1)^i\frac{\beta!(i+\alpha-N)!(i+\alpha-\beta-1)!}
{i!(i+\alpha)!(\alpha-\beta-1)!(\beta-N)!}
=(-1)^i\binom{\beta}{N}\!\binom{i+\alpha-\beta-1}{\alpha-\beta-1}
/\binom{i+\alpha}{N}
\]
which ends up the proof of Lemma~\ref{lemma3}.
\end{proof}

%
\begin{theorem}\label{theorem-matrix-inv}
The matrix $A^{-1}B$ is given by
\[
A^{-1}B=\begin{bmatrix}\displaystyle(-1)^i\frac{N}{i+\alpha-\beta}
\binom{\beta}{N} \! \binom{\alpha-\beta+N-1}{N} \! \binom{N-1}{i}
/\binom{i+\alpha}{N}\end{bmatrix}_{0\le i\le N-1}\!.
\]
\end{theorem}
%
\begin{proof}
Referring to Lemmas~\ref{lemma1} and \ref{lemma3}, we have that
\begin{align*}
A^{-1}B &=U(L^{-1}B)
\\
&=\begin{bmatrix}\displaystyle(-1)^{i+j}\ind_{\{i\le j\}}\binom{j}{i}
\!\binom{j+\alpha}{N}/\binom{i+\alpha}{N}\end{bmatrix}_{0\le i,j\le N-1}
\\
&
\hphantom{=\;\;}\times\begin{bmatrix}\displaystyle(-1)^i
\binom{\beta}{N}\!\binom{i+\alpha-\beta-1}{\alpha-\beta-1}
/\binom{i+\alpha}{N}\end{bmatrix}_{0\le i\le N-1}\!.
\end{align*}
The generic term of $A^{-1}B$ is
\begin{equation}\label{sum2}
\frac{(-1)^i\binom{\beta}{N}}{\binom{i+\alpha}{N}}
\sum_{j=i}^{N-1} \binom{j}{i}\!\binom{j+\alpha-\beta-1}{\alpha-\beta-1}.
\end{equation}
The foregoing sum can be easily evaluated as follows:
\begin{align*}
\sum_{j=i}^{N-1} \binom{j}{i}\!\binom{j+\alpha-\beta-1}{\alpha-\beta-1}
&
=\binom{i+\alpha-\beta-1}{\alpha-\beta-1} \sum_{j=i}^{N-1}
\binom{j+\alpha-\beta-1}{i+\alpha-\beta-1}
\\
&
=\binom{i+\alpha-\beta-1}{\alpha-\beta-1} \sum_{j=i}^{N-1}
\left[\binom{j+\alpha-\beta}{i+\alpha-\beta}
-\binom{j+\alpha-\beta-1}{i+\alpha-\beta}\right]
\\
&
=\binom{i+\alpha-\beta-1}{\alpha-\beta-1}\!\binom{\alpha-\beta+N-1}{i+\alpha-\beta}
\\
&
=\frac{N}{i+\alpha-\beta} \binom{\alpha-\beta+N-1}{N}\!\binom{N-1}{i}.
\end{align*}
Putting this into~(\ref{sum2}) yields the matrix $A^{-1}B$ displayed in
Theorem~\ref{theorem-matrix-inv}.
\end{proof}


\end{document}